\documentclass[11pt]{amsart}
\usepackage[margin=1in]{geometry}
\usepackage[multiple]{footmisc}
\usepackage[
colorlinks=true,
linkcolor=blue,
citecolor=blue,
urlcolor=blue]{hyperref}
\usepackage{enumitem}
\usepackage{cleveref}
\usepackage{amssymb}

\DeclareFontShape{OMX}{cmex}{m}{n}{<->cmex10}{}

\DeclareFontFamily{OMX}{MnSymbolE}{}

\DeclareFontShape{OMX}{MnSymbolE}{m}{n}{
    <-6>    MnSymbolE5
    <6-7>   MnSymbolE6
    <7-8>   MnSymbolE7
    <8-9>   MnSymbolE8
    <9-10>  MnSymbolE9
    <10-12> MnSymbolE10
    <12->   MnSymbolE12
}{}

\DeclareSymbolFont{MnLargeSymbols}
                  {OMX}{MnSymbolE}{m}{n}

\DeclareMathDelimiter{\llangle}{\mathopen}
    {MnLargeSymbols}{'164}
    {MnLargeSymbols}{'164}

\DeclareMathDelimiter{\rrangle}{\mathclose}
    {MnLargeSymbols}{'171}
    {MnLargeSymbols}{'171}

\newcommand{\R}{\mathbf R}
\newcommand{\sphere}{\mathbf{S}^{n-1}}
\newcommand{\E}{\mathbf E}
\newcommand{\Prob}{\mathbf P}

\newcommand{\rad}{\operatorname{rad}_2}
\newcommand{\DualDvoretzky}{k_\ast}
\newcommand{\Cov}{\operatorname{Cov}}
\newcommand{\Var}{\operatorname{Var}}
\newcommand{\Lip}{\operatorname{Lip}}
\newcommand{\tr}{\operatorname{tr}}
\newcommand{\eps}{\varepsilon}
\newcommand{\rank}{\operatorname{rank}}
\let\trace\tr
\newcommand{\SL}{\mathrm{SL}}
\newcommand{\cS}{\mathcal{S}}

\newcommand{\RandGel}{\mathrm{cr}}
\newcommand{\GL}{\mathrm{GL}}
\newcommand{\Gr}{\mathrm{Gr}}

\let\vr\vrad
\newcommand{\1}{\mathbf 1}
\newcommand{\vertiii}[1]{{\left\vert\kern-0.25ex\left\vert\kern-0.25ex\left\vert #1 
\right\vert\kern-0.25ex\right\vert\kern-0.25ex\right\vert}}
\newcommand{\gauge}[1]{\vertiii{#1}}
\newcommand{\opnorm}[1]{\vertiii{#1}_{\rm op}}
\newcommand{\fronorm}[1]{\vertiii{#1}_{\rm F}}
\newcommand{\e}{\mathrm{e}}
\newcommand{\ud}{\mathrm{d}}
\newcommand{\PoinConst}{C_{\rm P}}
\newcommand{\ie}{\emph{i.e.,}}
\newcommand{\eg}{\emph{e.g.,}}
\newcommand{\Bobkov}{\mathrm{B}}
\newcommand{\cL}{\mathcal{L}}
\let\vol\Vol
\newcommand{\floor}[1]{\left\lfloor #1 \right\rfloor}

\newcommand{\cent}{\mathrm{cent}}
\newcommand{\intr}{\operatorname{int}}
\let\tilde\widetilde

\newcommand{\supp}{\mathrm{supp}}
\newcommand{\T}{\mathsf{T}}
\newcommand{\crad}{\mathrm{crad}_2}

\theoremstyle{plain}
\newtheorem{theorem}{Theorem}[section]
\newtheorem{lemma}[theorem]{Lemma}
\newtheorem{proposition}[theorem]{Proposition}
\newtheorem{corollary}[theorem]{Corollary}
\theoremstyle{remark}
\newtheorem{remark}[theorem]{Remark}
\theoremstyle{plain}

\title[Mean width and metric entropy of convex bodies]
{Optimal mean width and metric entropy \\ estimates for convex bodies}

\author{Grigoris Paouris}
\thanks{GP gratefully acknowledges support from the NSF under grant DMS-2405441 and from AFOSR under grant FA9550-25-1-0284.}
\address{Department of Mathematics, 
Texas A\&M University, \and\newline\hspace*{\parindent}%
Department of Mathematics, Princeton University}
\email{grigoris@tamu.edu}

\author{Reese Pathak}
\thanks{RP gratefully acknowledges support from the NSF under grant DMS-2503579.}
\address{Department of Statistics, UC Berkeley,\and\newline\hspace*{\parindent}%
School of Oper. Res. \& Inf. Engineering (ORIE), Cornell University}
\email{pathakr@berkeley.edu}

\begin{document}

\begin{abstract}
We show that for any $n \geq 1$ and any convex body $K \subset \R^n$,
there exists $T \in \SL(n)$ such  that
\[
\mathrm{diam}(TK) \lesssim \sqrt{n} \vr(K), \quad \mbox{and} \quad  M^\ast\Big(T(K-x) \cap r\vr(K)\,B^n_2\Big)\lesssim 
\vr(K) \sqrt{\log(\e r^2)}, 
\]
for every $x \in K$ and every $r \geq 1$. 
Above, $M^\ast(\cdot)$ denotes the spherical mean width and $\vr(\cdot)$ denotes the volume radius. 
As a consequence,  we establish for any convex body $K \subset \R^n$ that
\[
1 \leq \inf_{T \in \SL(n)} \, \frac{M^\ast(TK)}{\vr(K)} 
\lesssim \sqrt{\log(\e n)}.
\]
The estimates above are sharp, up to universal constants, as they are attained for the crosspolytope and any regular $n$-simplex. Up to universal constants, our results show that all quermassintegrals, and the logarithm of the covering numbers for all scales, simultaneously for $K$ and the polar body $K^\circ$, are maximized by the simplex and crosspolytope. 
Our proof makes use of Eldan's stochastic localization. To establish the results, we work with an extension of Bobkov's maximal Gaussian measure position to possibly non-symmetric convex bodies. Our results imply that this position is an optimal ``regular'' Milman position, thereby improving a result of G.~Pisier. We establish a non-symmetric analogue of the strong Gaussian (B)-theorem which may be of independent interest.
\end{abstract}

\maketitle

\section{Introduction}

\subsection{Optimal covering for convex bodies}
Covering numbers quantify the compactness and optimal discretizations of a set; see \S\ref{sec:covering-numbers} for further details and historical remarks.
Given compact sets $K,L\subset \R^{n}$, with $L$ having
nonempty interior, the covering number  of $K$ with respect to $L$ is \[
N(K, L) = \min\bigg\{\, m \geq 1 \mid  \text{there exist}~x_1, \dots, x_m \in \R^n~\text{s.t.}~ K\subseteq \bigcup_{i=1}^{m}(x_{i}+L)\,\bigg\}.
\]
Our first result involves the covering numbers with respect to $B_{2}^{n}$,  the closed, origin-centered Euclidean unit ball.
We define the polar body $K^\circ$ of a convex body $K \subset \R^n$ as 
\[
K^\circ = \Big\{\,y \in \R^n : \langle x, y\rangle \leq 1,~~ \mbox{for every}~x \in K\,\Big\}.
\]
Above, $\langle \cdot, \cdot \rangle$ denotes the usual scalar product on $\R^n$. For $r > 0$, we define 
\[
\mathcal{H}(K, r) = \log \Big(\max\big\{ N(K, r B),
N(B^\circ, rK^\circ), N(K^\circ, r B^\circ), N(B, rK)\big\}\Big), \quad B = \tfrac{1}{\sqrt{n}}B^n_2.
\]
For $\mathcal{H}(K, r)$ to be finite, $K^\circ$ must be bounded, \ie~the origin must lie in the interior of $K$. 

\medskip 
Our first result shows that for a convex body $K \subset \R^n$, there always exists an ellipsoid that can be used to cover $K$ with essentially no more copies of the ellipsoid than are required to optimally cover the crosspolytope $B^n_1$ of the same volume. 

\begin{theorem}
\label{thm:covering-numbers}
There exist constants $c_{\ref{thm:covering-numbers}}, C_{\ref{thm:covering-numbers}} \in (0, \infty)$ such that for any $n \geq 1$ and any convex body ${K\subset\R^{n}}$, there exists an invertible affine transformation ${\psi \colon \R^n \to \R^n}$ with  ${\vol_n(\psi(K)) = \vol_n(B^n_1)}$, and \[
\mathcal{H}(\psi(K), r)
\leq
  C_{\ref{thm:covering-numbers}}
  \, 
  \inf_{\psi'} \, \mathcal{H}(\psi'(B^n_1), c_{\ref{thm:covering-numbers}}r), \quad \mbox{for all}~r > 0.
\]
The infimum above ranges over volume-preserving affine maps $\psi' \colon \R^n \to \R^n$. Moreover, the same result holds with $B^n_1$ replaced by any regular $n$-simplex with the same volume. 
\end{theorem}
\medskip 
By a classical result of Sch\"{u}tt \cite{Sch84},
the Euclidean metric entropy of the 
crosspolytope satisfies%
\footnote{We refer to the logarithm of the covering number as the \emph{metric entropy}.}%
\textsuperscript{,}%
\footnote{The notation $a \lesssim b$ (or $b \gtrsim a$) means that there exists a universal constant
$C>0$ such that
$
a\leq Cb$. We write $a \simeq b$ if $a\lesssim b$ and $b \lesssim a$.
We reserve $c,C,C_{1}, c', C'$ for universal constants.}
\begin{equation}
\label{Carsten}
\inf_{\psi} \mathcal{H}(\psi(B^n_1),r)
\simeq  \log N\bigl(nB_{1}^{n},r \sqrt{n} B_{2}^{n}\bigr)
\simeq 
\begin{cases}
n\log\left(\frac{\e}{r}\right),
& 0<r\leq 1, \\
\frac{n}{r^{2}}\log\left(\e r^2\right),
& 1 <r\leq c \sqrt{n}
\end{cases}.
\end{equation}
Above, the infimum ranges over volume-preserving affine maps $\psi \colon \R^n \to \R^n$.
The relations above also hold for regular $n$-simplices with the same volume as $B^n_1$; see~\Cref{app:computations-for-crosspolytope-and-simplex} for additional details.

\begin{remark}
\label{rem:pisier-comparison}
We compare~\Cref{thm:covering-numbers} to  a fundamental result of G.~Pisier on ``regular'' positions for convex bodies; also see \S\ref{sec:positions}. We define the \emph{entropy numbers} of a convex body $K \subset \R^n$ as follows:\footnote{If $K \subset \R^n$ is a centrally symmetric convex body, the entropy numbers $e_k(K)$ and $e_k^\ast(K)$ correspond to the entropy numbers of the identity map $(\R^n, \|\cdot\|_K)\to (\R^n, \|\cdot\|_2)$ and its adjoint, respectively.}
\begin{equation}
\label{eqn:entropy-of-body}
e_k(K) = \inf\big\{r > 0 : N(K, rB^n_2) \leq 2^{k-1}\big\}, 
\quad \mbox{and} \quad 
e_k^\ast(K) = \inf\big\{r > 0 : N(B^n_2, rK^\circ) \leq 2^{k-1}\big\},
\end{equation}
for any integer $k \geq 1$. We can define the quantity
\[
\mathcal{E}_k(K) =
\max\{e_k(K),e_k(K^\circ), 
e_k^\ast(K), e_k^\ast(K^\circ)\}.
\]
Pisier showed~\cite[Corollary 2.2]{Pis89paper} that for any centrally symmetric convex body $K \subset \R^n$ and any $\beta > 1/2$, there exists $T = T_\beta \in \GL(n)$ such that \emph{simultaneously} for all $k \geq 1$,
\begin{equation}
\label{eqn:regularity-of-pisier}
\mathcal{E}_k(T_\beta K) 
\leq C_\beta  \begin{cases}
(\frac{n}{k})^\beta, & 
k  \in [1, n] \\ 
2^{-k/n}, & k \geq n + 1
\end{cases}
\quad \mbox{where}  \quad 
C_\beta \lesssim \max\Big\{1,\frac{1}{\sqrt{2\beta - 1}}\Big\}. 
\end{equation}
Note that in the range $k \geq n + 1$, the estimate follows from the $k=n$ estimate by volumetric considerations. 
Optimizing with respect to $\beta$, we obtain a $k$-dependent estimate from~\eqref{eqn:regularity-of-pisier}. In particular, for every centrally symmetric convex body $K \subset \R^n$ and for any integer $k \in \{1, 2, \dots, n\}$, there is $S_k = T_{\beta_k} \in \GL(n)$ such that
\[
\mathcal{E}_k(S_k K) \lesssim \sqrt{\frac{n}{k} \log \Big(\frac{\e n}{k}\Big)}.
\]
\Cref{thm:covering-numbers} can, in contrast, be formulated as follows: for any (possibly non-symmetric) convex body $K \subset \R^n$, there is an affine map $\psi \colon \R^n \to \R^n$ such that $\vol_n(\psi(K)) = \vol_n(B^n_2)$, and \begin{equation}
\label{eqn:optimal-entropy-numbers}
\mathcal{E}_k(\psi(K)) \lesssim 
\inf_{\psi'} \, \mathcal{E}_k(\psi'(B^n_1)) \simeq 
\begin{cases}
\sqrt{n}, & k \in [1, \log(\e n)] \\ 
\sqrt{\frac{n}{k} \log\big(\frac{\e n}{k}\big)}, & k \in (\log(\e n), n] \\ 
2^{-k/n}, & k \geq n + 1,
\end{cases}
\end{equation}
\emph{simultaneously} for all $k \geq 1$. 
The infimum above ranges over affine maps $\psi' \colon \R^n \to \R^n$ such that $\vol_n(\psi'(B^n_1)) = \vol_n(B^n_2)$. Above, the crosspolytope can be replaced by any regular $n$-simplex. The estimate \eqref{eqn:optimal-entropy-numbers} is optimal, even in the non-symmetric case. Notably, if $k \in [1, \log(\e n)]$, the estimate~\eqref{eqn:optimal-entropy-numbers} is a multiplicative factor of order $\sqrt{\log(\e n/k)/k}$ smaller than  the $k$-dependent estimate of Pisier. 
\end{remark}

\subsection{Optimal mean width and quermassintegrals for convex bodies}

A fundamental parameter associated with a convex body
is its spherical mean width, 
\[
M^\ast(K) = \E_{\theta \sim \sigma_{n-1}} \Big[ \sup_{x \in K} \, \langle x, \theta \rangle \Big]
= \E_{\sigma_{n-1}} \, h_K(\theta).
\]
Above, $\sigma_{n-1}$ denotes the uniform measure on the sphere $\sphere$. The spherical mean width (or its Gaussian counterpart, the Gaussian width) appears naturally in analysis, probability, and geometry; see~\S\ref{sec:mean-width} for additional remarks. 

A fundamental question regarding the mean width concerns the maximal order of $M^\ast(K)$ over convex bodies $K\subset \R^n$ with fixed volume.
By Urysohn's inequality (see \cite {Sch14}), it is known that
$M^\ast(K) \geq \vr(K)$, with equality for $K = B^n_2$, where $\vr(K) = (\tfrac{\vol_n(K)}{\vol_n(B^n_2)})^{1/n}$ denotes the volume radius of a convex body. 
To obtain upper bounds, one must first take a suitable linear image $TK$ (also referred to as a \emph{position} of $K$), where $T \in \GL(n)$. The question is then to determine how large the \emph{minimal} mean width
\begin{equation}\label{eqn:minimal-mean-width}
\inf_{T \in \GL(n)} \frac{M^\ast(T K)}{\vr(T K)} 
\end{equation}
can be over all convex bodies $K \subset \R^n$.
Despite considerable interest, the order of this minimal mean width remained unknown.

Our second result resolves this question. In fact, our result provides significantly more information: it establishes the order for both $K$ and $K^\circ$ simultaneously in the same position. Moreover, it establishes the order for the higher-dimensional quantities, known as the  \emph{quermassintegrals}. 
Following the normalization arising from Kubota's
formula~\cite{Sch14}, the $k$th quermassintegral of a convex body $K \subset \R^n$ is
\[
W_{[k]}(K)
=
\left(
\int_{\Gr_{n,k}}
\frac{\vol_k(P_F K)}{\vol_k(B_2^k)}
d\nu_{n,k}(F)
\right)^{\frac{1}{k}}, \qquad \mbox{for}~k=1,\dots,n.
\]
Above, $\Gr_{n,k}$ denotes the Grassmann manifold of $k$-dimensional
subspaces of $\R^n$, $\nu_{n,k}$ denotes its Haar probability
measure, and $P_F$ denotes the orthogonal projection onto a subspace $F \subset \R^n$. 
Note that $W_{[n]}(K) = \vr(K)$.
We also introduce the ratios 
\[
\mathcal{U}_k(K) = \frac{W_{[k]}(K)}{\vr(K)}, \qquad \mbox{for}~k=1,\dots, n.
\]
Our result shows that the maximal order, simultaneously for $K, K^\circ$ and all integers $k$, is attained (in the right position) by the crosspolytope $B^n_1$ (with polar body equal to the cube $B^n_\infty$).
\begin{theorem}
\label{thm:quermassintegrals}
There exists $C_{\ref{thm:quermassintegrals}} \in (0, \infty)$ such that for any $n \geq 1$ and every convex body $K \subset \R^n$, there exists an invertible affine transformation $\psi \colon \R^n \to \R^n$ such that for every $k \in \{1, 2, \dots, n\}$,
\begin{equation}
\label{THM3-1}
1 \leq \max\Big\{\mathcal{U}_k(\psi(K)),\; \mathcal{U}_k(\psi(K)^\circ)\Big\}
\leq C_{\ref{thm:quermassintegrals}} 
\,
\inf_{\psi'}\max\Big\{\mathcal{U}_k(\psi'(B^n_1)),\;\mathcal{U}_k(\psi'(B^n_1)^\circ)\Big\}\simeq  \sqrt{\log\left(\frac{\e n}{k}\right)}.
\end{equation}
The infimum above ranges over volume-preserving affine transformations ${\psi' \colon \R^n \to \R^n}$. Additionally, the same result holds when $B^n_1$ is replaced by $S_n$, any regular $n$-simplex with ${\vol_n(S_n) = \vol_n(B^n_1)}$.
\end{theorem}

The ``Urysohn ratios'' $\mathcal{U}_k$ appearing in~\Cref{thm:quermassintegrals} are ordered. 
By the Alexandrov inequalities, for any convex body $K \subset \R^n$,
\[
\mathcal{U}_{1}(K) 
\; \geq \; 
\mathcal{U}_2(K)
\; \geq \;
\cdots
\geq \;
\mathcal{U}_n(K).
\]
Returning to the question of minimal mean width~\eqref{eqn:minimal-mean-width}, we see that \Cref{thm:quermassintegrals}
and the translation invariance of the mean width $M^\ast$ immediately imply the following corollary.

\begin{corollary}
\label{cor:reverse-urysohn}
There exists $C_{\ref{cor:reverse-urysohn}} \in (0, \infty)$ such that for any $n \geq 1$ and any convex body $K\subset\R^{n}$,
\[
1\leq \inf_{T \in \GL(n)} \, \frac{M^\ast(TK)}{\vr(TK)}
\leq C_{\ref{cor:reverse-urysohn}} \,   
\inf_{\psi} \frac{M^\ast(\psi(B^n_1))}{\vr(\psi(B^n_1))}\simeq
\sqrt{\log(\e n)}.
\]
Above, the infimum ranges over volume-preserving affine transformations $\psi' \colon \R^n \to \R^n$. Additionally, the same result holds when $B^n_1$ is replaced by $S_n$, any regular $n$-simplex with $\vol_n(S_n) = \vol_n(B^n_1)$. 
\end{corollary}

\begin{remark}
\label{rem:dvoretzky-version}
In fact, the proof of~\Cref{thm:quermassintegrals} 
also yields more information. In particular, as shown in~\Cref{cor:subsets-of-K}, for any convex body $K \subset \R^n$, there exists an operator $T \in \SL(n)$ such that 
\[
M^\ast(TK)
\lesssim \vr(K) \sqrt{\log(\e  + \crad(TK)^2/\vr(K)^2)}
\lesssim 
\vr(K)\sqrt{\log(\e n/\DualDvoretzky(TK))}.
\]   
Above, $\crad(A) =\inf_{x \in A} \sup_{y \in A} \|x - y\|_2$ and $\DualDvoretzky(A) = (w(A)/\crad(A))^2$ denote the circumradius and Dvoretzky number associated to a convex body $A \subset \R^n$. Note that the estimates above are sharp, up to universal constants, for example, if $K$ is an ellipsoid. Additionally, as described in~\Cref{cor:general-operator}, an interpolation argument gives
\begin{align*}
M^\ast(K) &\lesssim \vr(K) \cdot \inf_{\lambda \geq 0} \Big\{\, 
\Big(\sum_{i=1}^n (\sigma_i-\lambda)^2_+\Big)^{1/2} + \lambda \sqrt{\log(\e +\crad(TK)^2/\vr(K)^2)}\,\Big\}
\\
&\lesssim \vr(K) \cdot \inf_{\lambda \geq 0} \Big\{\, 
\Big(\sum_{i=1}^n (\sigma_i-\lambda)^2_+\Big)^{1/2} + \lambda \sqrt{\log(\e n/\DualDvoretzky(TK))}\,\Big\}.
\end{align*}
Above, $\sigma_1\geq \sigma_2 \geq \cdots \geq \sigma_n$ denote the singular values of the operator $T^{-1}$.
\end{remark}

\subsection{History and additional remarks} 

\subsubsection{Covering numbers}
\label{sec:covering-numbers}
For symmetric convex bodies, covering numbers are precisely the geometric counterpart of entropy numbers of operators between Banach spaces, which are fundamental quantities in operator and approximation theory; see \cite{Pin85, CarSte90} and the references therein. For instance, Sch\"utt's result~\eqref{Carsten} is a direct translation from the original formulation for entropy numbers. 

A breakthrough by V.\ Milman \cite{Mil86} was the realization that by fixing the Euclidean structure in dimension $n$, the entropy numbers for all convex bodies are the same. Milman's M-ellipsoid theorem asserts that for every symmetric convex body $K\subset \R^n$ there exists an ellipsoid $E$ of equal volume, with $\log{N(K, E)} \leq c n$~\cite{Mil86}. This also extends to the non-symmetric case; see \cite{MilPaj00}. \Cref{thm:covering-numbers}, with $r \simeq 1$, corresponds to V.\ Milman's relation: \ie~the position used in our result is an M-position (in the sense of V.\ Milman).\footnote{To avoid confusion with the position of minimal $M$, we also refer to this position as ``Milman's position.'' }

Milman's position is a central tool in the local theory of Banach spaces and in asymptotic convex geometry. Among its principal consequences are Milman's reverse Brunn--Minkowski inequality and the Bourgain--Milman ``reverse Santal\'o'' inequality \cite{BouMil87}. Our results also give, as a consequence, another proof of the reverse Blaschke-Santal\'o inequality; see~\Cref{rem:blashcke-santalo}.

Pisier considerably strengthened Milman's theorem by proving the existence of regular ellipsoids. More precisely, for every $\alpha \in (0, 2)$ and every symmetric convex body $K$, there is an ellipsoid $E_\alpha$ such that
\begin{equation}
\label{Pisier}
 \max\left\{ \log{N (K, t E_{\alpha}) },\log{ N( K^{\circ}, t E_{\alpha}^\circ )}\right\} \leq C_{\alpha} \frac{n}{ t^{\alpha}} , \ t >0.
 \end{equation}
See \cite[Chapter 7, Theorem 7.13]{Pis89}. This is known as Pisier's $\alpha$-regular M-position. In the above setting, as $\alpha$ tends to $2$, $C_{\alpha}$ tends to infinity. The role of symmetry and the barycenter of a convex body was the focus of several works in the late 1990s and early 2000s, initiated for example by Rudelson~\cite{Rud00} and Milman-Pajor~\cite{MilPaj00}. The non-symmetric analogue of \eqref{Pisier} was an open problem. Partial progress in that direction was initiated by Vritsiou~\cite{Vri24} and then later by Bizeul-Klartag~\cite{BizKla26}. This is now essentially resolved by~\Cref{thm:covering-numbers}.
Observe that \Cref{thm:covering-numbers} introduces the necessary logarithmic correction in $t$, but otherwise obtains a nearly-$2$-regular M-position.

Artstein-Milman-Szarek settled the Hilbertian case of Pietsch's long-standing duality of entropy conjecture \cite[Theorem 1]{ArtMilSza04}. Their result implies that when $K$ is symmetric and $E$ is an ellipsoid, the covering numbers satisfy $\log{N(K,E)} \simeq \log{N(E^{\circ},cK^{\circ})}$. Hence,~\Cref{thm:covering-numbers} has an immediate dual counterpart.
 
 Finally, we mention K.\ Ball's observation~\cite{PHDBall} that, assuming the hyperplane conjecture, the isotropic position of a convex body is an M-position. The recent breakthrough paper of Bizeul--Klartag~\cite{BizKla26} confirms that isotropic position is, up to polylogarithmic factors, a $2$-regular Pisier M-position.
 
 \subsubsection{Mean width of convex bodies}
 \label{sec:mean-width}

 The mean width of a convex body is a fundamental geometric parameter with applications in asymptotic convex geometry, the local theory of Banach spaces, the geometry of numbers, high-dimensional probability, random matrix theory, empirical-process theory, compressed sensing, high-dimensional statistics, and convex inverse problems; see, for example, \cite{AubSza17, Ver18, Reg22}.
 
 In probability theory, it frequently appears in its Gaussian normalization: the Gaussian mean width of a bounded set $K$ is $w(K) =\E\sup_{x\in K}\langle G,x\rangle
= \E h_K(G)$. Note that $\sqrt{n} \, M^\ast(K) \simeq w(K)$.
Hence, these parameters are equivalent to the supremum of a canonical Gaussian process. Beginning with the work of Dudley, Sudakov, and Fernique, a close connection between expected suprema of processes and the metric properties (in particular, the metric entropy) of the index set was established. This culminated in Talagrand's majorizing measures theorem \cite{Tal87,Tal05}: the Gaussian width is entirely determined, up to universal constants, by the metric structure of the process.

A parallel development took place in the local theory of Banach spaces, initiated primarily by Maurey and Pisier. For a Banach space $X$, its $K$-convexity constant is the norm $K(X)=\|{\rm {Rad}}_{X}\|$ of the Rademacher projection (see the survey \cite{Mau03}). Pisier's fundamental theorem (see \cite{Pis82}) states that bounded $K$-convexity is equivalent to the absence of uniformly embedded copies of $\ell _1^m$. Figiel and Tomczak-Jaegermann introduced the $\ell$-norm and established its fundamental connection with mean width and $K$-convexity \cite{FigTomJae79}. Their result implies that, for every symmetric convex body $K\subset\R^n$, there exists $T\in \GL(n)$ such that
\[
M(TK) \, M^\ast(TK)\leq CK(X_K),
\]
where $M(K)=M^\ast(K^{\circ})$ and $X_{K}$ is the Banach space whose unit ball is $K$. Pisier subsequently proved the quantitative estimate
\[
K(X)\leq C\bigl(1+\log d(X,\ell_2^n)\bigr)\leq C\, \log(\e n)
\]
for every $n$-dimensional normed space $X$ \cite{Pis81}; see also \cite[Chapter~2]{Pis89}. Combining these results gives the classical ``$MM^\ast$''-estimate
\[
\inf_{T\in \GL(n)}\, M(TK)\, M^\ast(TK)\leq C\log(\e n).
\]
This yields a (suboptimal) estimate of order $\log n$ for the reversal of Urysohn's inequality.
Note that it extends easily to the non-symmetric case via the Rogers--Shephard inequality.

For Banach lattices, Pisier obtained the sharper estimate
$K(X)\leq C\sqrt{\log(\e n)}$, which leads to optimal estimates in that case. Other notable cases include zonoids and the polars of unit balls of subspaces of $L_{p}$, for which sharp results were previously known \cite{GiaMilRud00,PaoVal18}.

It has been conjectured that every symmetric convex body satisfies
$$
\inf_{T\in \GL(n)} \, M(TK)M^\ast(TK)\leq C\sqrt{\log(\e n)};$$
see the discussion in the book of Aubrun-Szarek \cite[p.~207]{AubSza17}. The authors also conjectured a $\sqrt{\log (\e n)}$ upper bound for the normalized minimal mean width, which we establish in~\Cref{cor:reverse-urysohn}. The question was also stated in the paper of Giannopoulos-Milman-Rudelson~\cite{GiaMilRud00} as well as the book of Artstein-Avidan, Giannopoulos, and Milman~\cite{ArtAviGiaMil15}.

Prior to our work, the paper by Bizeul--Klartag~\cite{BizKla26} reduced the minimal mean width conjecture to the KLS conjecture. Their work and ours depend on Eldan's stochastic localization~\cite{Eld12, Eld13} in the form introduced by Lee and Vempala~\cite{LeeVem24}. Moreover, these works make use of the idea initiated by Eldan--Lehec~\cite{EldLeh14} to obtain comparison theorems for log-concave isotropic measures via stochastic localization.

The first version of our paper appeared online on July 31, 2026. Independently, on the same date, a paper of P. Bizeul appeared online~\cite{Biz26}, which focuses on the isotropic position of convex bodies. In particular, he establishes the optimal mean width bound for bodies in the isotropic position, a consequence of which is the optimal minimal mean width for general convex bodies, as in \Cref{cor:reverse-urysohn}. 

\subsubsection{Quermassintegrals and intrinsic volumes of convex bodies}
\label{sec:quermassintegrals}

Let $K$ be a convex body in $\R^{n}$. The Steiner formula states that
\[
\vol_n(K+tB_2^n)
=\sum_{k=0}^{n} \vol_k(B_{2}^{k}) \, V_{n-k}(K) \, t^{k},
\]
where $V_j(K)$ denotes the $j$th intrinsic volume of $K$. In our notation,
\[
V_{k}(K) = \binom{n}{k} \frac{\vol_n(B_{2}^{n})}{ \vol_{n-k}(B_{2}^{n-k})} W_{[k]}^{k}(K). 
\]
The Alexandrov inequalities imply that among convex bodies of fixed volume, every $V_k$ (or $W_{[k]}$) is minimized by Euclidean balls. The case $k=n-1$ corresponds to the classical isoperimetric inequality for surface area, while $k=1$ corresponds to Urysohn's inequality~\cite[Theorem~7.3.1]{Sch14}.

The reverse problem is to minimize $W_{[k]}(TK)$ over volume-preserving linear (or affine) transformations $T$ and determine the bodies for which this affine minimum is largest. A complete sharp answer is presently known only for $k=n-1$. Ball's reverse isoperimetric theorem shows that the extremal body is the cube in the centrally symmetric class and the simplex in the general class \cite{Ball91}. Ball's argument is based on the geometric Brascamp--Lieb inequality; see also Barthe~\cite{Bar98}, who developed the reverse Brascamp--Lieb inequality and treated its equality cases. For $1\leq k\leq n-2$, no comparable sharp identification of the extremal body is known in full generality. \Cref{thm:quermassintegrals} shows that, up to universal constants, the simplex is a maximizer.

Quermassintegrals are also closely connected with metric entropy. Our normalization is intended to make this connection transparent. In particular, it is not difficult to see that $K$ is in Milman's position if and only if
$
W_{[ n/2]}(K)\simeq {\rm vr}(K)
$; see \cite[Chapter 8]{ArtAviGiaMil15}.
  
\subsubsection{Positions of convex bodies}
\label{sec:positions}
For the reverse isoperimetric inequalities mentioned above, one wants to find the linear transformation that minimizes the particular quantity associated with the convex body of interest. This leads to the notion of ``positions'' of convex bodies. As explained by Giannopoulos-Milman \cite{GiaMil00}, building on results dating back to F.\ John in 1948, these positions typically lead to ``isotropic'' conditions. In \cite{GiaMil00}, the authors provide conditions for minimizing the $k$th quermassintegral. Notably, we \emph{are not} working with these positions, and in fact, they can differ. The position we use was introduced by Bobkov \cite{Bob11} in the symmetric case and was shown to be a Milman position. It is known that~\Cref{thm:covering-numbers,thm:quermassintegrals} do \emph{not} hold in the minimal mean width position. Markessinis-Paouris-Saroglou \cite{MarPaoSar12} provide an example of a convex body that is in the minimal mean width position but is far from Milman's position.

We emphasize that our results also establish the order of several geometric quantities (the mean width of truncation bodies, the quermassintegrals, the Gelfand numbers, the metric entropy) all in the same position, and for the body and its polar. It appears that once the diameter of the body and the spherical mean width of the truncation bodies are established (and similarly for the polar body), one can establish all of the results in this paper. There are several other consequences which we have decided not to record here for the sake of presentation; we plan to present those consequences in a forthcoming short note. As one illustration, and as discussed in~\Cref{rem:qs-rand-gelf}, by following an argument of E.\ Milman and Yifrach~\cite{MilYif21}, we are able to establish the optimal order of the randomized Gelfand numbers (for the body and its polar), which also implies an improved estimate for the quotient-of-subspace (QS) theorem. In the non-symmetric case, this appears to improve on the previously best-known estimate, which was due to Rudelson~\cite{Rud00}.

\subsection{Ideas in the proof}

One of the main observations in this work is that in order to establish the estimates for the metric entropy, the quermassintegrals, and consequently, the mean width, it would suffice to show that, for unit-volume convex bodies $K \subset \R^n$,
\begin{equation}
\label{eqn:key-ineq-for-everything}
\rad(L) \lesssim n, \quad \mbox{and} \quad   \sup_{x \in L} w(L \cap (x + r \sqrt{n}B^n_2)) \lesssim n
\sqrt{\log(\e r^2)} \quad \mbox{for all}~r \geq 1, 
\end{equation}
simultaneously for $L \in \{K, nK^\circ\}$. We refer to the sets $L \cap (x + r\sqrt{n} B^n_2)$ or $(L-x) \cap r \sqrt{n} B^n_2$ as the \emph{truncated bodies} of a convex body $L \subset \R^n$. In order for~\eqref{eqn:key-ineq-for-everything} to hold, of course, $K$ must be in the right position. We establish our estimates by assuming that $K$ is in a suitable generalization of Bobkov's maximal Gaussian measure position. In fact, our proof would allow for us to put either $K$ or $K^\circ$ in this position, after normalizing the body to have unit volume. In particular, this also indicates that there may be several positions (\ie~multiple choices of linear image of $K$) in which our estimates hold. This is related to the fact that the notion of Milman's position is, in general, not unique. 

To establish the estimates~\eqref{eqn:key-ineq-for-everything}, we apply a martingale comparison theorem for convex functionals which, to our knowledge, dates back at least to Borkar~\cite{Bor87}. It implies that, subject to the appropriate regularity assumptions, for a continuous martingale $\{M_t\}_{t\geq 0}$ on $\R^n$ with $M_0 = 0$ and quadratic variation satisfying $\ud [M]_t \succeq \rho I_n \, \ud t$, it holds for each $T > 0$ that
\[
\E f(M_T) \geq \sqrt{\rho T}\,  \E f(G), \quad \mbox{where} \quad G \sim \gamma_n,
\]
for every positively $1$-homogeneous convex functional $f \colon \R^n \to \R$.
To construct the martingales $\{M_T\}$, we use an idea originating with Eldan--Lehec~\cite{EldLeh14}, which was later utilized by Bizeul--Klartag~\cite{BizKla26} for the purposes of lower bounding the norms of general log-concave random vectors. The idea is to use the stochastic localization process, in the form of Lee-Vempala~\cite{LeeVem24} and originally introduced by Eldan~\cite{Eld12}, to produce the martingales $\{M_t\}$. 

Let $\psi_n$ denote the KLS constant in dimension $n$. The main result in~\cite{BizKla26} shows that, by working with the uniform measure on the convex body $K^\circ$ and the mean process $a_t = \E_{X \sim \mu_t} X$, where $\mu_t$ denotes the stochastic localization process, one obtains for convex bodies $K \subset \R^n$ with $K^\circ$ in isotropic position,
\[
M^\ast(K) \lesssim \psi_n \, \vr(K) \sqrt{\log (\e n)}.
\]
Thus, if the KLS conjecture were resolved (\ie~if $\psi_n = O(1)$), then it would imply~\Cref{cor:reverse-urysohn}.

Our approach, which allows us to avoid the reduction to the KLS conjecture and dealing with general log-concave measures, is to replace the uniform measure on $K^\circ$ with the Gaussian measure conditioned on $K^\circ$. Assuming that this measure has scalar covariance and is centered is equivalent to assuming that $K^\circ$ lies in Bobkov's maximal Gaussian measure position~\cite{Bob11}. More generally, as we show in \S2, by suitably generalizing the strong Gaussian (B) theorem of Cordero-Erausquin-Fradelizi-Maurey~\cite{CorEraFraMau04}, one can define a maximal Gaussian measure position in the non-symmetric case; see \Cref{prop:log-concavity-of-max-shift}.  Moreover, since Gaussian measure is already strongly log-concave, the Brascamp-Lieb inequality automatically allows us to control the Poincar\'e constants of the measures $\{\mu_t\}_{t \geq 0}$.

Adapting the argument from~\cite{BizKla26} to the Gaussian measure conditioned on a convex body in (the non-symmetric analogue of) Bobkov's maximal Gaussian measure position would essentially only establish the estimate~\eqref{eqn:key-ineq-for-everything} at $r \simeq \sqrt{n}$. This is insufficient to control covering numbers (\Cref{thm:covering-numbers}), the quermassintegrals for $k > 1$ (\Cref{thm:quermassintegrals}), the Gelfand numbers (\Cref{prop:gelfand-in-bobkov-position}), or the sharpened inequalities involving the Dvoretzky number or the circumradius of the body~(\Cref{rem:dvoretzky-version}). For the entire family of bounds~\eqref{eqn:key-ineq-for-everything}, we establish the following comparison theorem:
\begin{equation}\label{eqn:comparison-in-intro}
\E_{G \sim \gamma_n}[f(G)] \lesssim  \,\inf_{u \in (0, n]} \, \Big(\sqrt{\log(\e n/u)} \, \E f(X) + \Lip(f)\sqrt{u} \Big),
\end{equation}
where $X$ is distributed according to a Gaussian measure conditioned on a convex body in Bobkov's position and $f\colon \R^n \to \R_+$ is any positively $1$-homogeneous convex function; see~\Cref{thm:comparison-theorem} for a formal statement. This bound is obtained by showing that up to time $T \simeq 1/(\log(\e n/u))$, the covariance $A_t=\Cov(\mu_t)$ along the stochastic localization process has ``many'' (at least $n - u$) eigenvalues above $1/2$, for each $u \in (0, n]$. This estimate is obtained by applying It\^o's formula to a smoothed version of the count of eigenvalues below $1/2$.

Optimizing~\eqref{eqn:comparison-in-intro} with respect to $u$ and taking $f$ to be the support function of the sets $(L-x) \cap rB^n_2$, where $X$ is distributed according to the Gaussian measure conditioned on the polar body $L^\circ$, yields~\cref{eqn:key-ineq-for-everything}. We obtain the covering estimates by relating differences of the entropy to the minoration of the Gaussian width of the truncated bodies; see~\Cref{prop:truncation-width-to-entropy}. We leverage known relations between the metric entropy and width of truncated bodies (in particular, leveraging a result of Mourtada~\cite{Mou25}), along with our estimates on the entropy and the width of truncated bodies; see~\Cref{prop:quermassintegrals-to-entropy}.

Finally, to obtain estimates for the polar body (and not just the original body which is assumed to lie in Bobkov's maximal Gaussian measure position), we need to construct a martingale which essentially lies in the polar body. For this, we consider the typical outer normal unit vector on the boundary of the body assumed to be in position.
We can establish the following analogue of~\cref{eqn:comparison-in-intro}:
\begin{equation}
\label{eqn:polar-ineq-introduction}
\E_{G \sim \gamma_n} f(G) \lesssim 
\inf_{u \in (0, n]}
\Big(
n\sqrt{\log(\e n/u)} \E f(\nabla h_{K}(X'))  +
\Lip(f) \sqrt{u} \Big),
\end{equation}
for every positively $1$-homogeneous convex functional $f \colon \R^n \to \R_+$, where $X'$ is drawn from a suitably normalized probability measure on the boundary of the body which is assumed to be in Bobkov's maximal Gaussian measure position; see~\Cref{thm:comparison-theorem-polar} for a formal statement. In order to establish this comparison theorem, we study a different continuous martingale, namely
the process 
\[
a_t^\circ = \int_{\partial L} n_{L}(x) \, \ud \mu^{\partial}_t(x), 
\]
where $L$ is the body originally assumed to be in position and $n_L(x)$ is the unit outer normal for $x \in \partial L$. Here, $\mu^{\partial}_t$ is the boundary measure defined by $\ud \mu^{\partial}_t(x) = p_t(x)\,\ud \mathcal{H}^{n-1}(x)$ on $\partial L$, where $p_t$ denotes the density of the measures $\mu_t$ arising during the stochastic localization. The divergence theorem implies that $a_t^\circ$ is an affine function of the mean process $a_t$ and the random tilt $\theta_t$ arising in the stochastic localization process. This enables repeating the same approach as for the mean process: an It\^o computation to control the eigenvalues to thereby obtain
\cref{eqn:polar-ineq-introduction}. To obtain~\cref{eqn:key-ineq-for-everything} for the polar body, we then follow the same approach as for the original body. Similarly, the estimates for the quermassintegrals and metric entropy then follow via the same approach from the width of the intersection bodies.

\subsection*{AI disclosure} 

No AI models were used in this project 
or in the preparation of this manuscript.

\section{Maximal Gaussian measure position and a non-symmetric $(\mathrm{B})$-theorem}

We will make use of conditioned Gaussian measures on $\R^n$. We denote the standard Gaussian measure and the Gaussian measure with variance $\tau^{-2}$ by
\[
\gamma_n = N(0, I_n), \quad \mbox{and} \quad 
\gamma_{\tau, n} = N(0, \tau^{-2} I_n).
\]
The conditioned versions of these measures on some measurable set $L \subset \R^n$ are given by 
\[
\gamma_{L}(A) = \frac{\gamma_n(A \cap L)}{\gamma_n(L)} 
\quad \mbox{and} \quad 
\gamma_{\tau, L}(A) = \frac{\gamma_{\tau,n}(A \cap L)}{\gamma_{\tau, n}(L)} 
\]
for every measurable set $A \subset \R^n$ and $\tau > 0$. 

\medskip 

\subsection{Bobkov position in the non-symmetric case} 
\label{sec:bobkov-non-symmetric}
In the paper~\cite{Bob11}, Bobkov introduced a position, using Gaussian measures, for centrally symmetric convex bodies. Bobkov defines a convex body $K \subset \R^n$ with $K = -K$ to be in \emph{maximal Gaussian measure position} (abbreviated \emph{$\Bobkov$-position}) if
\begin{equation}\label{eqn:variational-condition}
\gamma_n(K) \geq \gamma_n(TK) 
\qquad\text{for every} \quad 
T\in\SL(n).
\end{equation}
Using the strong Gaussian $(\mathrm{B})$-theorem, as established by Cordero-Erausquin, Fradelizi, and Maurey \cite{CorEraFraMau04}, Bobkov established~\cite[Proposition 3.11 and p.~32]{Bob11} that if $K \subset \R^n$ is a convex body such that $K = -K$ and is in $\Bobkov$-position, then the condition~\eqref{eqn:variational-condition} is equivalent to $\gamma_K$ having a scalar covariance. Moreover, if additionally $\vol_n(K) = 1$, then
\[
\alpha_K^2 \equiv \frac{\trace(\Cov(\gamma_K))}{n} \gtrsim 1.
\]

\medskip 

As we will apply the $\Bobkov$-position in the case when $K \subset \R^n$ is a convex body which may \emph{not} be centrally symmetric, we now develop an affine analogue which extends these properties to the non-symmetric case. Notably, the strong (B)-theorem is false, at least as formulated by~\cite{CorEraFraMau04} without considering the maximum measure translates of the body. The failure of the strong (B) theorem in the non-symmetric case was shown by Nayar and Tkocz~\cite{NayTko13}.

We define a (possibly non-symmetric) convex body $K \subset \R^n$ to be in $\Bobkov$-position if
\begin{equation}\label{eqn:new-variational-condition}
\gamma_n(K) \geq \gamma_n(f(K))~\quad \mbox{for every volume-preserving, affine}~f \colon \R^n \to \R^n.
\end{equation}
Throughout the remainder of this paper, when referring to $\Bobkov$-position, we mean the condition~\eqref{eqn:new-variational-condition}. 

By the following basic lemma, if additionally $K =-K$, then the definition~\eqref{eqn:new-variational-condition} generalizes Bobkov's definition in the symmetric case. 

\begin{lemma}
\label{lem:centering-implies-maximal-shift}
Let $K \subset \R^n$ be a convex body. Then, the probability measure $\gamma_K$ is centered if and only if $x \mapsto \gamma_n(K + x)$ attains its maximum at the origin.     
\end{lemma}
\begin{proof}
Set $F_K(x) = \log \gamma_n(K+x)$. A change of variables yields
\[
F_K(x) = \log \int_{\R^n} \1_K(w) \phi(w + x) \, \ud w, \quad \mbox{where} \quad \phi(z) = \frac{1}{(2\pi)^{n/2}} \e^{-\|z\|_2^2/2}.
\]
By Pr\'ekopa's theorem, $F_K$ is concave. Observe that 
\[
\nabla F_K(x) = -\frac{1}{\gamma_n(K+x)} \int_{\R^n} \1_K(w) (w+x) \phi(w+x)\, \ud w = -\E Y,
\]
where the expectation is over $Y \sim \gamma_{K+x}$. Thus, $F_K$ is maximized at $x = 0$ if and only if $\gamma_{K}$ is centered, as required.
\end{proof}

In fact, more generally, one can show the following equivalence of the isotropy and centering of $\gamma_K$ with the condition~\eqref{eqn:new-variational-condition}.

\begin{theorem}
\label{thm:Bobkov-non-symmetric}
    Let $K \subset \R^n$ be a convex body. Then  the probability measure  $\gamma_K$ is centered and has scalar covariance if and only if the condition~\eqref{eqn:new-variational-condition} holds.
\end{theorem}

As a consequence of this equivalence, we obtain the following properties when $K \subset \R^n$ is a unit-volume convex body in $\Bobkov$-position.

\begin{corollary}
\label{cor:bobkov}
There exist universal constants $C_{\ref{cor:bobkov}}, c_{\ref{cor:bobkov}} \in (0, \infty)$ such that the following holds. For any $n \geq 1$ and any convex body $K \subset \R^n$ in $\Bobkov$-position with $\vol_n(K) = 1$, it holds that $\alpha_{K}^2 \geq c_{\ref{cor:bobkov}}$ and $\rad(K^\circ) \leq C_{\ref{cor:bobkov}}$.
\end{corollary}

To prove the corollary, we recall the following extension due to Milman-Pajor, which establishes the existence of $M$-ellipsoids in the non-symmetric setting. We use the (Euclidean) centroid of a convex body $K \subset \R^n$, defined as
\[
\cent(K) = \frac{1}{\vol_n(K)} \int_K x \,\ud x.
\]

\begin{proposition}[Milman-Pajor~{\cite[Theorem 11]{MilPaj00}}]
\label{prop:milman-pajor}
There exists a universal constant ${C_{\ref{prop:milman-pajor}} \in (0, \infty)}$ such that for any $n \geq 1$ and any convex body $K \subset \R^n$, there is $T \in \SL(n)$ such that
\[
\log N\Big(T(K-\cent(K)), \vr(K) B^n_2\Big) \leq C_{\ref{prop:milman-pajor}} \, n.
\]    
\end{proposition}

\begin{proof}[Proof of~\Cref{cor:bobkov}]
We begin by estimating the Gaussian measure of $K$ from below. Let ${T \in \SL(n)}$ be as in~\Cref{prop:milman-pajor}.
By~\Cref{thm:Bobkov-non-symmetric}, it holds that 
\begin{equation}\label{ineq:consequence-of-bobkov}
\gamma_n(K) \geq \gamma_n\Big(T\big(K -\cent(K)\big)\Big) = \gamma_n(L), \quad \mbox{where} \quad L = T\big(K -\cent(K)\big). 
\end{equation}
Set $\tilde L = L \cap (-L)$, and put $B = \vol_n(B^n_2)^{-1/n} B^n_2$. Since $\vol_n(K) = 1$, 
\[
\log N(\tilde L, B) \leq \log N(L, B) \leq C_{\ref{prop:milman-pajor}} n. 
\]
Hence, for a positive integer $M \leq \exp(C_{\ref{prop:milman-pajor}} n)$ and some $x_1, \dots, x_M \in \tilde L$, it holds that 
\[
\tilde L \subset \bigcup_{i=1}^M (x_i + B).
\]
In particular, by the pigeonhole principle, there exists a positive integer $i \leq M$ such that 
\[
\vol_n(\tilde L \cap (x_i + B))^{1/n} \geq M^{-1/n}\vol_n(\tilde L)^{1/n} \geq \e^{-C_{\ref{prop:milman-pajor}}}\vol_n(\tilde L)^{1/n}.
\]
Put $C = \tilde L \cap (x_i + B)$. Then 
$\tfrac{1}{2}(C-C) \subset \tilde L \cap B$. The Brunn-Minkowski inequality and the display above yield
\begin{multline}
\label{ineq:lower-on-intersection-volume}
\vol_n(L \cap B)^{1/n} \geq \vol_n(\tilde L \cap B)^{1/n} \\\geq \vol_n(\tfrac{C -C}{2})^{1/n} \geq 
\vol_n(C)^{1/n} \geq \e^{-C_{\ref{prop:milman-pajor}}}\vol_n(\tilde L)^{1/n} \gtrsim \vol_n(L)^{1/n}.
\end{multline}
The final inequality follows as $\cent(L) = 0$; see~\cite[Corollary 3]{MilPaj00}.
Hence, applying inequalities~\eqref{ineq:consequence-of-bobkov}~and~\eqref{ineq:lower-on-intersection-volume}, we conclude that there are universal constants $c, c' > 0$ such that 
\[
\gamma_n(K)^{1/n} \geq \gamma_n(L)^{1/n} \geq \gamma_n(L\cap B)^{1/n} \geq c \vol_n(L \cap B)^{1/n} \geq c' \vol_n(L)^{1/n} = c'.
\]
The last relation used $\vol_n(L) = \vol_n(K) = 1$, by assumption. 

To deduce the lower bound on $\alpha_K^2$, we will need the following elementary estimate on the small ball probabilities of a Gaussian random vector: there exists a universal constant $C > 0$ such that for $\eps > 0$, 
\[
\gamma_n(\eps \sqrt{n} B^n_2)^{1/n} \leq \frac{1}{\sqrt{2 \pi}} \eps \vol_n(\sqrt{n} B^n_2)^{1/n} \leq C\eps.  
\]
Take $\eps$ to be a sufficiently small constant (\ie~ $\eps \leq c'/(2C)$). From the previous two displays,
\[
\frac{\gamma_n(K \cap \eps \sqrt{n} B^n_2)}{\gamma_n(K)} \leq \Big(\frac{C\eps}{c'}\Big)^n \leq \frac{1}{2^n}.
\]
Hence, with our choice of $\eps > 0$,
\[
\alpha_K^2 =
\frac{1}{\gamma_n(K)} \int_{K} \phi(x) \frac{\|x\|_2^2}{n} \, \ud x  \geq \frac{1}{\gamma_n(K)} \int_{K \setminus \eps \sqrt{n} B^n_2} \phi(x) \frac{\|x\|_2^2}{n}\, \ud x  \geq \eps^2 \Big(1 - \frac{\gamma_n(K \cap \eps \sqrt{n} B^n_2)}{\gamma_n(K)} \Big) \gtrsim 1.
\]
Finally, let $y \in K^\circ$. Observe that if $X \sim \gamma_K$, then $W = 1 - \langle X, y\rangle$ is a log-concave random variable. Since $X$ is centered, $\E[W] = 1$. As $(X, y) \in K \times K^\circ$, the random variable $W$ is also nonnegative. Thus, by Borell's lemma,
\[
\alpha_K^2 \|y\|_2^2= \Var(\langle X, y\rangle) = \E[\langle X, y\rangle^2] = \E[(1-W)^2] = -1 + \E W^2 \lesssim (\E W)^2 = 1.
\]
Passing to the supremum over $y \in K^\circ$, we obtain $\rad(K^\circ) \leq \alpha_K^{-1} \lesssim 1$, as required. 
\end{proof}

We also recall the Poincar\'{e} constant of a probability measure $\nu$ on $\R^n$. It is the smallest $C > 0$ such that 
\[
\Var_\nu(f) \leq C\, \E_\nu \|\nabla f\|_2^2 
\]
holds for all locally Lipschitz $f \colon \R^n \to \R$. We denote the smallest such constant by $\PoinConst(\nu)$.

\begin{lemma}
\label{lem:poin-const-of-conditioned-gaussians}
Let $\nu = \cL(Z \mid \{Z \in K\})$, where $K \subset \R^n$ is a convex body and $Z \sim N(z, \tau^2 I_n)$ for some $\tau > 0$ and some $z  \in \R^n$. Then $\PoinConst(\nu) < \tau^2$. 
\end{lemma}
\begin{proof}
There are several ways to obtain the result; the argument we present allows us to obtain the result by passing to the covariance operator of $\gamma_K$. By Klartag's improved Lichnerowicz inequality~\cite{Kla23} and using that $\nu$ is $\tau^{-2}$-strongly log-concave, it holds that
\[
\PoinConst(\nu)^2 \leq \tau^2 \opnorm{\Cov(\nu)} < \tau^4,
\]
with the final inequality following from the Brascamp-Lieb inequality and its equality cases~\cite[Theorem 1.5]{Liv24}.
\end{proof}

To establish~\Cref{thm:Bobkov-non-symmetric}, we require the following result, which can be regarded as a non-symmetric analogue of the strong $(\mathrm{B})$-theorem. In the symmetric case, this was previously established by Cordero-Erausquin, Fradelizi, and Maurey~\cite{CorEraFraMau04}. Our main observation is that the following affine formulation allows it to extend to the non-symmetric setting.

\begin{proposition}
\label{prop:log-concavity-of-max-shift}
For any real symmetric matrix $A \in \R^{n \times n}$ and any convex body $K \subset \R^n$, the map 
\[
t \mapsto \sup_{x \in \R^n} \gamma_n(\e^{tA} K + x)
\]
is log-concave on $\R$. 
\end{proposition}

Again, by invoking \Cref{lem:centering-implies-maximal-shift}, one can check that~\Cref{prop:log-concavity-of-max-shift} also recovers the ``usual'' form of the strong (B)-theorem when $K$ is additionally centrally symmetric. 

We also use~\Cref{prop:log-concavity-of-max-shift} to establish the uniqueness of the Bobkov position in the non-symmetric case. 
\begin{proposition}
\label{prop:existence-uniqueness-of-Bobkov}
Let $K \subset \R^n$ be a convex body. Then there exists a volume-preserving, affine transformation $\psi \colon \R^n \to \R^n$ such that
\[
\gamma_n(\psi(K)) \geq \gamma_n(\psi'(K)), 
\]
for all volume-preserving affine $\psi' \colon \R^n \to \R^n$. 
The body $\psi(K)$ contains the origin in its interior. 
Moreover, $\psi$ is unique, up to orthogonal transformations. \end{proposition}
\begin{proof}
For some $(T, x) \in \SL(n) \times \R^n$, write $L = TK + x$. Observe 
that 
\[
\inf_{\theta \in \sphere} \big(h_L(\theta) + h_L(-\theta)\big) =
\inf_{\theta \in \sphere} \sup_{z \in K - K} \langle z, T^\ast\theta\rangle
\leq \mathrm{diam}(K) \sigma_n(T),
\]
where $\sigma_n$ denotes the minimal singular value of $T \in \GL(n)$. In particular, this implies by rotational invariance that
\[
\gamma_n(L) \leq \sup_{\substack{\text{interval}~I \subset \R \\ 
|I| \leq \mathrm{diam}(K) \sigma_n(T)}} \gamma_1(I) \leq \frac{1}{\sqrt{2\pi}} |\mathrm{diam}(K)| \sigma_n(T).
\]
We obviously can restrict to pairs $(T,x)$ with $\gamma_n(TK+x)\geq\gamma_n(K)$. Consequently, for such pairs, \[
\gamma_n(K)\leq \gamma_n(TK + x) \leq
\frac{1}{\sqrt{2\pi}} |\mathrm{diam}(K)| \sigma_n(T).
\]
Thus, we may restrict the supremum of $\gamma_n(TK+x)$ to the set of operators
\[
\Big\{T \in \SL(n) : \sigma_n(T) \geq \eps_K\Big\} \quad \mbox{where} \quad 
\eps_K = \sqrt{2\pi}\frac{\gamma_n(K)}{|\mathrm{diam}(K)|} > 0,
\]
which is compact. Additionally, over this class of operators, the bodies $TK$ are uniformly bounded, so $\gamma_n(TK+x)\to0$ uniformly in $T$ as $\|x\|_2\to\infty$. We may therefore also restrict to bounded translates. Continuity of $(T,x)\mapsto\gamma_n(TK+x)$ and compactness yield a maximizing volume-preserving affine map $\psi$.
By~\Cref{lem:centering-implies-maximal-shift}, $\gamma_{\psi(K)}$ is centered, and hence $0 \in \intr(\psi(K))$.

Regarding the uniqueness, let $\psi, \psi'$ denote two maximizing maps. By the polar decomposition, we may write 
\[
(\psi' \circ \psi^{-1})(y) = Q(\e^A y + b), 
\]
for some $b \in \R^n$ and some symmetric traceless $A \in \R^{n \times n}$ and some $Q \in O(n)$. 
Let 
\begin{equation}
\label{eqn:maximal-map}
F(t) = \sup_{x \in \R^n} \log \gamma_n(\e^{tA} \psi(K) + x).
\end{equation}
By~\Cref{prop:log-concavity-of-max-shift}, the map $F$ is concave. However, observe that by the maximality of $\psi, \psi'$, 
\[
F(0) = \log \gamma_n(\psi(K)) = \log\gamma_n(\psi'(K)) = F(1). 
\]
Consequently, $F$ is constant on $[0, 1]$.
Let $x(t)$ achieve the supremum in display~\eqref{eqn:maximal-map}.
By~\Cref{lem:differentiability}, this point exists, is unique, and depends smoothly on $t$. Set $L_t = \e^{tA}\psi(K)+x(t)$.
The second derivative computation in the proof of~\Cref{prop:log-concavity-of-max-shift} and the variance inequality in \Cref{prop:new-poincare}
yield
\[
F''(t)\leq -\frac{2(1-\PoinConst(\gamma_{L_t}))}{1+\PoinConst(\gamma_{L_t})}\, \E_{Y \sim \gamma_{L_t}} \|A Y\|_2^2.
\]
Since $F$ is constant (for $t \in (0, 1)$), and since $\PoinConst(\gamma_{L_t}) < 1$ by \Cref{lem:poin-const-of-conditioned-gaussians}, we conclude  $A = 0$ as ${\E_{Y \sim \gamma_{L_t}}\|AY\|_{2}^2 = 0}$. Finally, uniqueness of the maximizing translate of $\psi(K)$ gives $b=0$. Thus, $\psi'=Q\circ\psi$, as required.
\end{proof}

\subsection{Proof of~\Cref{thm:Bobkov-non-symmetric}}
\label{sec:proof-thm-bobkov-non-symmetric}
For a real symmetric matrix $A \in \R^{n \times n}$, define
\[
F_{A, K}(t) = \sup_{x \in \R^n} \log \gamma_n(\e^{tA} K + x), \quad t \in \R.
\]
Let $x_{A, K}(t)$ achieve the supremum above. Throughout the proof, we use the shorthand notation 
\[
x(t) = x_{A, K}(t), \quad v(t) = x'(t) - Ax(t), \quad w(t) = x''(t) - A^2 x(t), \quad \mbox{and} \quad K_t = \e^{tA} K + x(t).
\]
The following result shows that $x(\cdot)$ is a sufficiently smooth function of $t$. 

\begin{lemma}
\label{lem:differentiability}
    For any real symmetric matrix $A \in \R^{n \times n}$ and any convex body $K \subset \R^n$, there exists a unique point $x_{A, K}(t)$ which satisfies
    \[
    \gamma_n(\e^{tA} K + x) \leq \gamma_n(\e^{tA} K + x_{A, K}(t)), \quad \mbox{for all}~x \in \R^n.
    \]
    Moreover, the map $t \mapsto x_{A, K}(t)$ is $C^\infty$-smooth. 
\end{lemma}
\begin{proof}
Define 
\[
G(t,x)=\log\gamma_n(\e^{tA}K+x)
=t\tr(A)+\log\int_K\phi(\e^{tA}y+x)\,\ud y.
\]
Clearly, $G$ is $C^\infty$-smooth and, for every $t$, $G$ attains its maximum as $\lim_{\|x\|_2 \to \infty} G(t,x) = -\infty$.
Following the computation in~\Cref{lem:centering-implies-maximal-shift}, for  $Y_{t, x} \sim \gamma_{\e^{tA}K+x}$, it holds that 
\[
\nabla_x \, G(t,x)=-\E Y_{t,x},
\qquad
\nabla_x^2 \, G(t,x)=\Cov(Y_{t,x})-I_n \preceq 0.
\]
The last inequality holds, for instance, by~\Cref{lem:poin-const-of-conditioned-gaussians}. In fact, strict inequality holds:
\[
\opnorm{\Cov(Y_{t,x})} = \sup_{\theta \in \sphere} \Var(\langle Y_{t,x}, \theta\rangle) < 1.
\]
This can be verified by checking the equality cases of the Brascamp-Lieb inequality~\cite[Theorem 1.5]{Liv24}; in particular, note that the body $\e^{tA} K + x$ is compact and the map $y \mapsto \langle y, \theta \rangle$ is non-constant. Thus, $\nabla_x^2 \, G(t,x) \prec 0$ for every $t, x$, and hence $G(t, \cdot)$ is strictly concave for each $t$. Hence, there is a unique maximizer $x_{A, K}(t)$, which satisfies the nonlinear equation
\[
\nabla_xG(t,x_{A,K}(t))=0.
\]
The Jacobian (in $x$) is given by the nonsingular matrix $\nabla_x^2G(t,x_{A,K}(t))$. Hence, the implicit function theorem guarantees a $C^\infty$ choice of $x_{A, K}(t)$ on an open neighborhood of $t$. Repeating the argument for $s < t$ or $s > t$ and invoking the uniqueness of $x_{A, K}(s)$, we conclude that the map $t \mapsto x_{A,K}(t)$ is $C^\infty$, as claimed.\end{proof}

By a change of variables and the identity $\det(\e^{s A}) = \exp(s \tr(A))$, we obtain 
\begin{equation}
\label{eqn:integral-form-of-F}
F_{A,K}(s) =\log \int_{\R^n} \1_{K_s}(x) \, \phi(x) \, \ud x = \log \int_{\R^n} \1_K(y) \e^{s \tr(A)} \phi(\e^{sA} y+ x(s)) \, \ud y. 
\end{equation}
Now observe that 
\begin{align}
F_{A,K}'(s) &= \frac{1}{\gamma_n(K_s)} \int_{\R^n} \1_K(y) (\tr(A) - \langle \e^{sA} y + x(s), A \e^{sA} y + x'(s)\rangle) \, \e^{s \tr(A)} \phi(\e^{sA} y + x(s)) \, \ud y 
\label{eqn:integral-form-of-F-prime}\\ 
&= \frac{1}{\gamma_n(K_s)} \int_{\R^n} \1_{K_s}(z) (\tr(A) - \langle z, Az + v(s)\rangle) \,  \phi(z) \, \ud z  \nonumber \\
&= \tr(A) - \E[\langle Y_s, AY_s + v(s)\rangle],  \nonumber 
\end{align}
where $Y_s \sim \gamma_{K_s}$. Note that $\gamma_n(K_s) \geq \gamma_n(K_s + x)$ for any $x \in \R^n$, by construction. Hence, by \Cref{lem:centering-implies-maximal-shift}, $\gamma_{K_s}$ is a centered probability measure. Thus, 
\[
F'_{A,K}(s) = \tr(A) - \llangle A, \E(Y_s \otimes Y_s)\rrangle. 
\]
Since, by~\Cref{prop:log-concavity-of-max-shift}, the map $F_{A, K}$ is concave,  it is maximized at the origin for every real symmetric traceless $A$ if and only if
\begin{equation}~\label{eqn:relation}
\tr(A) =  \llangle A, \E[Y\otimes Y]\rrangle = \llangle A, \Cov(\gamma_K)\rrangle, \quad \mbox{for all real, symmetric, traceless}~A,
\end{equation}
where $Y \sim \gamma_{K}$. Here, we used that $x \mapsto \gamma_n(K + x)$ is maximized at the origin, as $\gamma_K$ is centered.
The family of matrices 
\[
\big\{e_i\otimes e_i - e_1 \otimes e_1\big\}_{i=2}^n \, \bigcup \, \big\{e_i \otimes e_j + e_j \otimes e_i\big\}_{1 \leq i < j \leq n}
\]
forms a basis of the space of symmetric traceless matrices $A \in \R^{n \times n}$.
Hence, it suffices to check the (linear) relation~\eqref{eqn:relation} on these matrices; this immediately implies that $\Cov(\gamma_K)$ is a multiple of the identity. Thus, combining this with \Cref{lem:centering-implies-maximal-shift}, we conclude that
the probability measure $\gamma_K$ being centered and having scalar covariance is equivalent to
\begin{equation}
\label{ineq:maximizing-inequality}
\gamma_n(\e^{A} K + x) \leq \gamma_n(K), 
\end{equation}
for every symmetric traceless $A \in \R^{n \times n}$ and every $x \in \R^n$. Now, fix a volume-preserving affine transformation $f \colon \R^n \to \R^n$. Note that $f(x) = Tx + v$ for some $T \in \SL(n)$ and some $v \in \R^n$. If $G \sim \gamma_n$, then $T^{-1} G$ has the law $N(0, (T^\ast T)^{-1})$. It follows that
\begin{equation}
\label{eqn:reduction-to-PSD}    
\gamma_n(f(K)) = \gamma_n(TK + v) = \gamma_n((T^\ast T)^{1/2}K + (T^\ast T)^{1/2} T^{-1}v). 
\end{equation}
Now observe that $(T^\ast T)^{1/2} T^{-1} \in \SL(n)$ and the family of operators $\{(T^\ast T)^{1/2}  : T \in \SL(n)\}$ coincides with the family of operators $\e^{A}$ where $A$ ranges over symmetric traceless matrices. Thus, combining displays~\eqref{ineq:maximizing-inequality} and~\eqref{eqn:reduction-to-PSD}, we obtain the claim.

\subsection{Proof of~\Cref{prop:log-concavity-of-max-shift}}

Let the notation of~\Cref{sec:proof-thm-bobkov-non-symmetric} stand. Following the computations~\eqref{eqn:integral-form-of-F} and~\eqref{eqn:integral-form-of-F-prime}, we obtain
\begin{align*}
F_{A, K}''(s) &= \frac{1}{\gamma_n(K_s)} \int_{\R^n} \1_{K_s}(x) \{(\tr(A) - \langle Ax + v(s), x\rangle)^2 \\ 
&\qquad\qquad\qquad- \langle A^2 x +  w(s) , x \rangle - \|Ax + v(s)\|_2^2\}\, \phi(x) \, \ud x - \big(F'_{A,K}(s)\big)^2 \\ 
&= \Var(\langle AY + v(s), Y\rangle)-2\E \|AY\|_2^2 - \|v(s)\|_2^2.
\end{align*}
The last line follows with $Y \sim \gamma_{K_s}$, because $\gamma_{K_s}$ is centered. We denote 
\[
q_s(y) = \langle Ay + v(s) , y \rangle, \quad \mbox{and} \quad 
\nabla q_s(y) = 2Ay + v(s).
\]
Again using that $Y \sim \gamma_{K_s}$ is centered, we obtain
\[
\E \|\nabla q_s(Y)\|_2^2 = \|v(s)\|_2^2 + 4 \E \|AY\|_2^2, 
\quad \mbox{and} \quad  \|\E \nabla q_s(Y)\|_2^2 = \|v(s)\|_2^2.
\]
Hence, 
\[
F''_{A, K}(s) = \Var(q_s(Y)) - \frac{1}{2}\Big\{ \E \|\nabla q_s(Y)\|_2^2 + \|\E \nabla q_s(Y)\|_2^2\Big\}.
\]
The conclusion that $F''_{A,K}(s) \leq 0$ and thus that $F_{A, K}$ is concave follows from the following functional inequality. Specifically, we may apply it with $f = q_s$ and $\mu = \gamma_{K_s}$. By \Cref{lem:poin-const-of-conditioned-gaussians}, we may take $\rho = 1$ and $\PoinConst(\gamma_{K_s}) \leq 1$.

\begin{proposition}
\label{prop:new-poincare}
Suppose that $\mu$ is a log-concave probability measure, \ie~ $\ud \mu = \e^{-\Phi(x)} \1_\Omega(x) \, \ud x$ for some convex set $\Omega \subset \R^n$ and some convex map $\Phi \colon \R^n \to \R$. If for every $x \in \mathrm{int}(\Omega)$ it holds that $\nabla^2 \Phi(x) \succeq \rho I_n$ for some $\rho > 0$, then for every locally Lipschitz $f \colon \R^n \to \R$, it holds that
\[
\Var_\mu(f) \leq \frac{\PoinConst(\mu)}{1+\rho \PoinConst(\mu)}\E_\mu \|\nabla f\|_2^2 + \frac{1}{\rho(1 + \rho \PoinConst(\mu))} \|\E_\mu \nabla f\|_2^2.
\]
\end{proposition}
\begin{proof}
We sketch the argument, as it closely follows the arguments of Barthe and Cordero-Erausquin~\cite[Theorem 4]{BarCorEra13} and of Cordero-Erausquin, Fradelizi, and Maurey \cite[Lemma 2]{CorEraFraMau04}. 
We prove the result for the case when $\mu$ has a smooth, positive density on
$\R^n$.
For general $\mu$, the result follows by applying the result to 
$\mu_\varepsilon=\mu*N(0,\varepsilon I_n)$; note that these measures have curvature at least
$\tfrac{\rho}{1+\rho\varepsilon}$, and satisfy
$\PoinConst(\mu_\varepsilon)\leq\PoinConst(\mu)+\varepsilon$. Thus the result follows by passing to the limit $\eps \to 0^+$. 

We will use the notation 
    \[
    \langle f, g \rangle_{L^2(\mu)} = \int \langle f(x), g(x) \rangle \, \ud \mu(x), \quad \|f\|_{L^2(\mu)}^2 = \langle f,f\rangle_{L^2(\mu)}, 
    \]
    for any square-integrable $f, g \colon \R^n \to \R^m$. 
    Consider the differential operator $\cL_\mu$ which acts on $C^2$-smooth functions $u \colon \R^n \to \R$ via 
    \[
    (\cL_\mu u)(x) = \Delta u(x) - \langle \nabla \Phi(x),  \nabla u(x) \rangle = \tr(\nabla^2 u(x)) - \langle \nabla \Phi(x),  \nabla u(x) \rangle, \quad x \in \R^n.
    \]
Note that the image of the space $C^2_{c}$ of compactly supported twice continuously differentiable maps under the operator $\cL_\mu$ is dense in the space $L^2_0(\mu)$ of functions $f \in L^2(\mu)$ with $\int f \, \ud \mu = 0$.
We will use two key formulas, which follow by integration by parts with respect to the measure $\mu$. First, for any $u \in C^2_c$, 
\[
\langle f, \cL_\mu u\rangle_{L^2(\mu)} = - \int \langle \nabla f(x), \nabla u(x)\rangle \, \ud \mu(x) = -\langle \nabla f, \nabla u\rangle_{L^2(\mu)}.
\]
Secondly, it holds that 
\[
\|\cL_\mu u\|_{L^2(\mu)}^2  = \int \Big\{ \fronorm{\nabla^2 u(x)}^2 + \langle \nabla^2 \Phi(x) \nabla u(x), \nabla u(x)\rangle \Big\} \, \ud \mu(x).
\]
Thus, fix $u \in C^2_c$. Note that for $f \in L^2_0(\mu)$ which is locally Lipschitz, it holds that
\begin{align*}
\|f\|_{L^2(\mu)}^2 &= \|f - \cL_\mu u\|_{L^2(\mu)}^2 +  2\langle f, \cL_\mu u \rangle_{L^2(\mu)} - \|\cL_\mu u \|_{L^2(\mu)}^2 \\
&= \|f - \cL_\mu u\|_{L^2(\mu)}^2  - 2 \langle \nabla f, \nabla u\rangle_{L^2(\mu)} \\ 
&\qquad\qquad  -\int \Big\{\fronorm{\nabla^2 u(x)}^2 + \langle \nabla^2 \Phi(x) \nabla u(x), \nabla u(x)\rangle \Big\}\, \ud \mu(x) \\ 
&\leq \|f - \cL_\mu u\|_{L^2(\mu)}^2  - 2\langle \nabla f, \nabla u\rangle_{L^2(\mu)} -
\PoinConst(\mu)^{-1} \|\nabla \overline{u}\|_{L^2(\mu)}^2 - \rho \|\nabla u\|_{L^2(\mu)}^2.
\end{align*}
Above, for a map $u \colon \R^n \to \R$, we have defined $\overline{u}(x) = u(x) - \langle x, \E_\mu \nabla u\rangle$. It satisfies the relation $\nabla \overline{u} = \nabla u - \E_\mu \nabla u$, which is clearly centered under $\mu$. Additionally, $\nabla^2 u = \nabla^2 \overline{u}$, and hence the final inequality above follows from applying the Poincar\'e inequality to the coordinates $\langle \nabla \overline u, e_j\rangle$ and from the assumption that $\nabla^2 \Phi \succeq \rho I_n$. We similarly center $f$ and obtain:
\begin{align*}
\|f\|_{L^2(\mu)}^2  &\leq \|f - \cL_\mu u\|_{L^2(\mu)}^2  - 2 \langle \nabla \overline{f}, \nabla \overline{u}\rangle_{L^2(\mu)} \\ &\qquad \quad -2 \langle \E_\mu \nabla f, \E_\mu \nabla u\rangle -(\PoinConst(\mu)^{-1} + \rho) \|\nabla \overline{u}\|_{L^2(\mu)}^2  - \rho \|\E_\mu \nabla u\|_2^2\\ 
&\leq \|f - \cL_\mu u\|_{L^2(\mu)}^2 + (\PoinConst(\mu)^{-1} + \rho)^{-1}\|\nabla \overline{f}\|_{L^2(\mu)}^2 -2 \langle \E_\mu \nabla f, \E_\mu \nabla u\rangle-\rho \|\E_\mu \nabla u\|_2^2 \\ 
&\leq \|f - \cL_\mu u\|_{L^2(\mu)}^2 + (\PoinConst(\mu)^{-1} + \rho)^{-1}\|\nabla \overline{f}\|_{L^2(\mu)}^2 + \rho^{-1} \|\E_\mu \nabla f\|_2^2 \\ 
&=  \|f - \cL_\mu u\|_{L^2(\mu)}^2 + \frac{\PoinConst(\mu)}{1 + \PoinConst(\mu) \rho}\|\nabla f\|_{L^2(\mu)}^2 + \frac{1}{\rho(1 + \PoinConst(\mu) \rho)}\|\E_\mu \nabla f\|_2^2.
\end{align*}
By density, the quantity $\|f - \cL_\mu u\|_{L^2(\mu)}^2$ can be made arbitrarily small; the result follows.
\end{proof}

\section{Relations between truncation bodies, metric entropy, and quermassintegrals}

In this section, we explain how control of the mean width of the truncation bodies, 
\[
(L - x) \cap \rho B^n_2, \quad \mbox{for}~\rho \geq 0, x \in L,
\]
for a given convex body $L \subset \R^n$ 
implies control of the metric entropy and quermassintegrals of $L$.

Because we will make use of the (essential) equivalence between packing and covering in this section, we also recall some covering and packing notions. 
The separation (or packing) number $M(A, B)$ is the maximal number of points in $A$ which are separated by $B$, \ie 
\[
M(A, B) = \max_{A_0 \subset A}\Big\{ \#(A_0) : x \not \in (y + B)~\mbox{for distinct}~x,y \in A_0\Big\}.
\]
These are related to the covering numbers by the obvious inequalities (\eg~\cite[Fact~4.1.11]{ArtAviGiaMil15}),
\[
M(A, 2B) \leq N(A, B) \leq M(A, B).
\]
Above, we assume $B = -B$.

\medskip 

To bound the entropy in terms of the Gaussian mean width of these truncation bodies, we invoke the Sudakov minoration and its dual form. We now recall these inequalities. The Sudakov minoration, in the form stated below, is taken from~\cite[Proposition 4.2.7]{ArtAviGiaMil15}. It was originally established by Sudakov~\cite[Theorem 2]{Sud71}, actually by establishing a more general Gaussian comparison theorem which implies the formulation given below. 
 
\begin{lemma}[Sudakov minoration]
\label{lem:sudakov-minoration}
There exists a universal constant $C_{\ref{lem:sudakov-minoration}} \in (0, \infty)$ such that for any $n \geq 1$ and any convex body $K \subset \R^n$, it holds that
\[
\log M(K, tB^n_2) \leq C_{\ref{lem:sudakov-minoration}} \frac{w(K)^2}{t^2},
\]
for every $t > 0$.
\end{lemma}

The dual Sudakov minoration is due to Pajor and Tomczak-Jaegermann~\cite{PajTom85}, with an alternative proof due to Pajor-Talagrand~\cite[Proposition~4.2]{BouLinMil89} and Ledoux-Talagrand~\cite[Equation (3.15)]{LedTal91}.

\begin{lemma}[Sudakov minoration, dual form]
\label{lem:dual-sudakov}
There exists a universal constant $C_{\ref{lem:dual-sudakov}} \in (0, \infty)$ such that for any $n \geq 1$ and any convex body $K \subset \R^n$ which contains the origin in its interior, it holds that
\[
\log N(B^n_2, tK) \leq C_{\ref{lem:dual-sudakov}} \frac{w(K^\circ)^2}{t^2},
\]
for every $t > 0$.
\end{lemma}
\begin{proof}
Recall the ``usual'' form of the dual Sudakov minoration (\eg~\cite[Theorem~4.2.2]{ArtAviGiaMil15}): there exists a universal constant $C > 0$ such that for any origin-symmetric convex body $L \subset \R^n$,
\[
\log M(B^n_2, tL) \leq C \, \frac{w(L^\circ)^2}{t^2}, \quad \mbox{for any}~t > 0.
\]
For a (possibly non-symmetric) convex body $K$ which contains the origin in its interior, set $L = K \cap (-K)$. Since $L \subset K$ and $L = -L$, we have
\[
\log N(B^n_2, tK)
\leq \log N(B^n_2, tL)
\leq \log M(B^n_2, tL) 
\leq C \frac{w(L^\circ)^2}{t^2}
\leq 4C \frac{w(K^\circ)^2}{t^2}.
\]
For the last inequality, observe that $L^\circ = \mathrm{conv}(K^\circ \cup (-K^\circ)) \subset K^\circ + (-K^\circ)$, so, for $G \sim \gamma_n$,
\[
w(L^\circ)
= \E \max\{h_{K^\circ}(G), h_{K^\circ}(-G)\}
\leq \E[h_{K^\circ}(G) + h_{K^\circ}(-G)]
= 2w(K^\circ).
\]
Here we used the nonnegativity of $h_{K^\circ}$ and the evenness of the Gaussian measure. Taking $C_{\ref{lem:dual-sudakov}} = 4C$ proves the claim.
\end{proof}

The next result shows that from control on the Gaussian width of the truncation bodies, we obtain control on the logarithm of the Euclidean covering numbers associated to $L$. 

\begin{proposition}
\label{prop:truncation-width-to-entropy}
Suppose that $L \subset \R^n$ is a convex body with the origin in its interior. Suppose that for some constant $C>0$,
\begin{equation}
\label{ineq:assumption-on-truncation-body}    
\sup_{x \in L}w((L-x) \cap \rho\sqrt{n} B^n_2)
\leq C n\sqrt{\log(\e \rho^2)}, \quad \mbox{for all}~\rho \geq 1.
\end{equation}
Then, it holds that 
\[
\max\Big\{\log  N(L, r \sqrt{n} B^n_2), \log N(B^n_2, r \sqrt{n} L^\circ) \Big\} 
\leq C' \cdot \begin{cases}n \log(\frac{\e}{r}), & r \in (0, 1) \\ 
\frac{n}{r^2} \log(\e r^2), & r \geq 1
\end{cases}.
\]
Above, the constant $C' > 0$ only depends on $C > 0$ (but not $n$, $r$, or $L$).
\end{proposition}

For the quermassintegrals, we require a lemma, which connects quermassintegrals to truncation bodies and to metric entropy. Recall the intrinsic volumes of a convex body $K \subset \R^n$, for $j \in \{1, \dots, n\}$, are given by
\[
V_j(K) = \binom{n}{j} \, \frac{\vol_n(B^n_2)}{\vol_j(B^j_2) \vol_{n-j}(B^{n-j}_2)} \, \E \vol_j(P_EK), 
\]
where the expectation is over a uniformly distributed $j$-dimensional subspace $E \subset \R^n$. By convention, we define $V_0(K) = 1$. 
\begin{lemma}[Mourtada~{\cite[Corollary 4.3]{Mou25}}]
\label{lem:bound-on-wills-new}
There is a constant ${C_{\ref{lem:bound-on-wills-new}} \in (0, \infty)}$ such that 
for any $n \geq 1$ and any convex body $K \subset \R^n$,
\[
\log \Big(\sum_{j=0}^n V_j(\lambda K)\Big) 
\leq C_{\ref{lem:bound-on-wills-new}} \, \inf_{r > 0} \Big( \lambda \, \sup_{x \in K}w((K-x) \cap r B^n_2) + \log N(K, r B^n_2) 
\Big),
\]
for any $\lambda > 0$. 
\end{lemma}

In fact, as shown also in the paper~\cite{Mou25}, the inequality above can be reversed, after adjusting the constants. We will not need that direction, however, for our argument below, which establishes control of the quermassintegrals from control of the Gaussian width of the truncation bodies.

\begin{proposition}
\label{prop:quermassintegrals-to-entropy}
Suppose that $L \subset \R^n$ is a convex body with the origin in its interior. Suppose that for some constant $C>0$, the assumption~\eqref{ineq:assumption-on-truncation-body} holds. 
Then, it holds that 
\[
W_{[j]}(L)
\leq C'' \sqrt{n \log\Big(\frac{\e n}{j}\Big)},
\quad \mbox{for every}~j \in \{1, 2, \dots, n\}.
\]
Above, the constant $C'' > 0$ only depends on $C > 0$ (but not $n$, $r$, or $L$).
\end{proposition}

\subsection{Proof of~\Cref{prop:truncation-width-to-entropy}}

Throughout this section, we use the shorthand notations
\[
H_L(t) = \log N(L,tB_2^n),\quad \mbox{and} \quad
H_L^\ast(t) = \log N(B_2^n, tL^\circ), 
\]
for $t > 0$. First, we show that the differences of the metric entropy are controlled by the Gaussian width of the truncation bodies.

\begin{lemma}
\label{lem:one-step-bound-new}
There is a constant $C_{\ref{lem:one-step-bound-new}} \in (0, \infty)$ such that, for every $n \geq 1$, every convex body $L \subset \R^n$ containing the origin in its interior, and every $t > 0$,
\[
\max\Big\{H_L(t/2) - H_L(t), H_L^\ast(t/2) - H_L^\ast(t)\Big\}
\leq \frac{C_{\ref{lem:one-step-bound-new}}}{t^2} 
\Big(\sup_{x \in L} w\big((L-x) \cap t B^n_2\big)\Big)^2.
\]
\end{lemma}
\begin{proof}
Let $\mathcal{P}_{t/2}$ and $\mathcal{N}_t$ denote, respectively, a maximal packing of $L$ at scale $t/2$ and a minimal covering at scale $t$. 
Such maximal packings are also coverings, 
so $|\mathcal{P}_{t/2}| \geq N(L, \tfrac{t}{2} B^n_2)$. 
For each $x \in \mathcal{N}_{t}$, we denote
\[
N_x = \#\Big((x + tB^n_2) \cap \mathcal{P}_{t/2}\Big) =  \#\{y \in \mathcal{P}_{t/2} : y \in x + tB^n_2\}.
\]
Because $\mathcal{N}_t$ is a covering at scale $t$, it clearly holds that 
\[
|\mathcal{P}_{t/2}| \leq \sum_{x \in \mathcal{N}_t} N_x.
\]
Hence, by the pigeonhole principle, there exists some $x^\star \in \mathcal{N}_t \subset L$ for which
\[
N_{x^\star} \geq \frac{|\mathcal{P}_{t/2}|}{|\mathcal{N}_t|} \geq \frac{N(L, \tfrac{t}{2} B^n_2)}{N(L, t B^n_2)}.
\]
Note that 
\[
\{y \in \mathcal{P}_{t/2} : y \in x^\star + t B^n_2\} \subset L \cap (x^\star+ tB^n_2).
\]
Moreover, the collection of points on the left-hand side above is, by definition, $t/2$-separated in the Euclidean metric. Therefore, by the Sudakov minoration (\Cref{lem:sudakov-minoration}),
\[
w((L-x^\star) \cap tB^n_2) = w(L\cap (x^\star + tB^n_2)) \geq c  \, t \sqrt{\log N_{x^\star}} 
\geq c\, t \, \sqrt{H_L(t/2) - H_L(t)}. 
\]
Rearranging the inequality gives
\begin{subequations}
\begin{equation}
\label{eqn:primal-entropy-inequality}
H_L(t/2) - H_L(t) \leq \frac{1}{c^2} \, \Big(\frac{w(L \cap(x^\star + t B^n_2))}{t}\Big)^2.
\end{equation}
For the dual inequality, consider the body 
\[
M_t = \frac{t}{2} L^\circ +\frac{1}{4} B^n_2, \quad t > 0.
\]
Note that 
\[
N(B^n_2, tL^\circ) 
\leq N(B^n_2, M_t) N(M_t, tL^\circ) 
\leq N(B^n_2, M_t) N(B^n_2, 2t L^\circ).
\]
Above, the final inequality used that from the centers of a covering of $\tfrac{1}{4} B^n_2$ by translates of $\tfrac{t}{2}L^\circ$, we immediately obtain a cover of $M_t$ by translates of $tL^\circ$, formed by the same centers. Taking the logarithm on both sides, and scaling $t \mapsto t/2$, 
\[
H_L^\ast(t/2) - H_L^\ast(t) \leq \log N(B^n_2, M_{t/2})  \leq c \, w(M_{t/2}^\circ)^2.
\]
The final inequality follows from the dual Sudakov minoration (\Cref{lem:dual-sudakov}).
By directly computing the support function of $M_s$ for $s > 0$, we see that 
\[
\frac{1}{4} M^\circ_s = \Big\{y \in \R^n : h_{2s L^\circ}(y) + \|y\|_2 \leq 1\Big\} \subset (2s L^\circ)^\circ \cap B^n_2 = \frac{1}{2s} (L \cap 2s B^n_2).
\]
Combining the two preceding displays and taking $s = t/2$, 
\begin{equation}
\label{eqn:dual-entropy-inequality}
H_L^\ast(t/2) - H_L^\ast(t) \leq \frac{16 c}{t^2} w(L \cap t B^n_2)^2 \leq \frac{16c}{t^2} \Big[\sup_{x \in L} w\big((L-x) \cap t B^n_2\big)\Big]^2,
\end{equation}
\end{subequations}
as $0 \in L$ by assumption. Combining inequalities~\eqref{eqn:primal-entropy-inequality}
and~\eqref{eqn:dual-entropy-inequality} yields the claim.
\end{proof}

To conclude, let $F(t)$ denote either $H_L(t\sqrt{n})$ or $H_{L}^\ast(t\sqrt{n})$, and assume first that $t \geq 1$. 
Observe that by 
\Cref{lem:one-step-bound-new}, 
it clearly holds that 
\[
F(t) - F(2t) \lesssim \frac{1}{nt^2} \Big[\sup_{x \in L} w\big((L-x) \cap 2t\sqrt{n}B^n_2\big)\Big]^2 
\lesssim \frac{n}{t^2} \log(\e t^2), \quad \mbox{for all}~t \geq 1. 
\]
The final inequality above made use of our assumption~\eqref{ineq:assumption-on-truncation-body}.
Since $B^n_2$, $L$, and $L^\circ$ are compact and have nonempty interior, it follows that for $t \geq 1$,
\begin{multline*}
F(t) = \lim_{m \to \infty} \{F(t) - F(2^m t)\} \\= \lim_{m \to \infty} \sum_{j=0}^{m-1} \big\{ F(2^j t) - F(2^{j+1} t)\big\} \lesssim \frac{n}{t^2} \sum_{j=0}^\infty \frac{\log(\e 4^j t^2)}{4^j}
\lesssim \frac{n}{t^2} \log(\e t^2),
\end{multline*}
as required. In the case $t \in (0, 1)$, let $m \geq 1$ denote the integer such that
\begin{equation}
\label{ineq:sandwich-for-2mt}
2^{m-1} t < 1 \leq 2^{m} t.
\end{equation}
Then, by the computation above, 
\begin{equation}
\label{eqn:bound-on-the-m-end-point}
F(2^m t) \lesssim \frac{n}{4^m t^2} \log(\e 4^m t^2) \lesssim n.
\end{equation}
Additionally, for each $j \in \{0, \dots, m-1\}$, by 
\Cref{lem:one-step-bound-new},
\begin{equation}
\label{ineq:bound-on-the-increments}
F(2^{j} t) - F(2^{j+1} t) 
\lesssim
\frac{1}{n4^{j} t^2} \Big[\sup_{x \in L} w\big((L-x) \cap 2^{j+1}t\sqrt{n}B^n_2\big)\Big]^2
\lesssim 
n.
\end{equation}
The last inequality follows from the monotonicity of Gaussian width and the bound $w(B^n_2) \leq \sqrt{n}$, since
\[
w\big((L-x) \cap 2^{j+1}t\sqrt{n}B^n_2\big)
\leq 2^{j+1}t\sqrt{n}\,w(B^n_2)
\leq 2^{j+1}tn,
\]
for every $x \in L$.
Combining the bounds~\eqref{eqn:bound-on-the-m-end-point} and~\eqref{ineq:bound-on-the-increments}, we obtain
\[
F(t) \leq \sum_{j=0}^{m-1} \{F(2^{j} t) - F(2^{j+1} t)\}  + F(2^m t) 
\lesssim (m+1) n.
\]
To conclude we only need to observe that from the inequalities~\eqref{ineq:sandwich-for-2mt}, it holds that
\[
m + 1 \leq 2 + \log_2\frac{1}{t} \lesssim \log\Big(\frac{\e}{t}\Big).
\]

\subsection{Proof of~\Cref{prop:quermassintegrals-to-entropy}}

We start by relating, quantitatively, the $j$th intrinsic volume to the $j$th quermassintegral.

\begin{lemma}
\label{lem:quermass-and-intrinsic}
There is a constant $C_{\ref{lem:quermass-and-intrinsic}}\in (0, \infty)$ such that for every $n\geq 1$
and every convex body $K \subset \R^n$ and every integer $j \in \{1, 2, \dots, n\}$, 
\[
W_{[j]}(K) \leq 
C_{\ref{lem:quermass-and-intrinsic}}
\, \frac{j}{\sqrt{n}} \, V_j(K)^{1/j}.
\]
\end{lemma}
\begin{proof}
Define the quantity
\[
C_{n,j} = \Big(\frac{\vol_{n-j}(B^{n-j}_2)}{\vol_n(B^n_2) \binom{n}{j}}\Big)^{1/j}.
\]
By definition of the quermassintegrals and intrinsic volumes, for a uniformly chosen $j$-dimensional subspace $E \subset \R^n$,
\[
 W_{[j]}(K) =  \vol_j(B^j_2)^{-1/j} 
\Big(\E \vol_j(P_EK) \Big)^{1/j}
= C_{n,j} 
V_j(K)^{1/j}.
\]
Recall the formula 
$\vol_m(B^m_2) = \pi^{m/2}/\Gamma(\tfrac{m}{2} + 1)$.
Using $\binom{n}{j} \geq (\tfrac{n}{j})^j$ and $\Gamma(x+a) \leq (x+a)^a \Gamma(x)$ for $x, a >0$,  we have 
\[
C_{n,j} = 
\frac{1}{\sqrt{\pi}} \Big(\frac{\Gamma(\tfrac{n}{2} + 1)}{\Gamma(\tfrac{n-j}{2} + 1)}\Big)^{1/j} \binom{n}{j}^{-1/j}  \leq 
\frac{1}{\sqrt{\pi}} \frac{j}{n} \sqrt{\frac{n}{2} + 1} \leq 2\frac{j}{\sqrt{\pi n}},
\]
as needed.
\end{proof}

To complete the proof of~\Cref{prop:quermassintegrals-to-entropy}, we set  
\[
\rho^2=\frac{n}{j}\log\Big(\frac{\e n}{j}\Big), \quad  \mbox{and} \quad
\lambda=\frac{j}{n\sqrt{\log(\e n/j)}}.
\]
Note that $\rho \geq 1$, and in particular
\[
\log(\e\rho^2)
= \log\Big(\frac{\e n}{j}\Big) + \log \log \Big(\frac{\e n}{j}\Big) 
\leq 2\log\Big(\frac{\e n}{j}\Big).
\]
Thus, by the assumption,
\[
\sup_{x\in L}w\big((L-x)\cap \rho \sqrt{n} B_2^n\big)
\leq \sqrt{2}Cn\sqrt{\log(\e n/j)}.
\]
Additionally, by \Cref{prop:truncation-width-to-entropy}, for some $C'=C'(C)$,
\[
\log N(L,\rho \sqrt{n}B_2^n)
\leq C'\frac{n}{\rho^2}\log(\e
\rho^2)
\leq 2C'j.
\]
Set $C''' = C_{\ref{lem:bound-on-wills-new}}(C \sqrt{2} + 2C')$. Then, \Cref{lem:bound-on-wills-new} gives 
\[
\log\Big(\sum_{k=0}^nV_k(\lambda L)\Big)
\leq C_{\ref{lem:bound-on-wills-new}}
\big(C\sqrt2\,j+2C'j\big)
\leq C''' j.
\]
Note that $V_j(\lambda L)=\lambda^jV_j(L)$. Hence, by the nonnegativity of the intrinsic volumes, 
\[
V_j(L)^{1/j} 
\leq \frac{1}{\lambda}
\Big(\sum_{k=0}^n V_k(\lambda L)\Big)^{1/j}
\leq \frac{\e^{C'''}}{\lambda}
=\e^{C'''}\frac{n}{j}\sqrt{\log \frac{\e n}{j}}.
\]
Set $C'' = \e^{C'''} C_{\ref{lem:quermass-and-intrinsic}}$. Applying \Cref{lem:quermass-and-intrinsic}, we conclude that 
\[
W_{[j]}(L)
\leq
C_{\ref{lem:quermass-and-intrinsic}}
\frac{j}{\sqrt n}V_j(L)^{1/j}
\leq C''\sqrt{n\log\Big(\frac{\e n}{j}\Big)},
\]
as required.

\section{Comparison theorems and estimates for truncation bodies}

In this section, we establish our main comparison theorems, \Cref{thm:comparison-theorem,thm:comparison-theorem-polar}. These will be used to deduce our main results in~\Cref{sec:proofs-of-main-results}.

\subsection{Preliminaries on stochastic localization and continuous martingales}

We make use of Eldan's stochastic localization process~\cite{Eld12, Eld13}, which we introduce now. Let $\mu$ be a compactly supported probability measure on $\R^n$ which is absolutely continuous with respect to Lebesgue measure. We write $p = \tfrac{\ud \mu}{\ud x}$ for its density. 
For $t \geq 0$ and $\theta \in \R^n$, we define the probability measure $\mu_{t, \theta}$ which satisfies
\[
\frac{\ud \mu_{t,\theta}}{\ud \mu}(x) = \e^{\langle \theta, x \rangle - t \|x\|_2^2/2 - 
\Lambda_t(\theta)}, 
\quad 
\Lambda_t(\theta) = 
\log \E_{X \sim \mu} \e^{\langle \theta, X\rangle - t \|X\|_2^2/2}.
\]
We make use of the following notation for the mean and covariance:
\[
a_t(\theta) = \E_{X \sim \mu_{t, \theta}}[X] 
\quad 
\mbox{and} \quad 
A_t(\theta) = \Cov(\mu_{t, \theta}).
\]
Let $\{B_t\}_{t \geq 0}$ be a standard Brownian motion in $\R^n$, with initial condition $B_0 = 0$. Consider the process $\{\theta_t\}_{t \geq 0}$ defined as the solution of the stochastic differential equation
\begin{equation}
\label{eqn:sde-for-theta}
\ud \theta_t = a_t(\theta_t)\, \ud t + \, \ud B_t,
\end{equation}
with initial condition $\theta_0 = 0$.
With a slight abuse of notation, we refer to the 
family of (random) probability measures $\mu_t \equiv \mu_{t, \theta_t}$ for $t \geq 0$ as the \emph{stochastic localization process}. We also write 
\[
a_t \equiv a_t(\theta_t), \quad A_t \equiv A_t(\theta_t), \quad t \geq 0,
\]
for the corresponding (random) mean and covariance processes. From~\eqref{eqn:sde-for-theta} and It\^o's formula, it follows that the Lebesgue density $p_t$ of the probability measure $\mu_t$ satisfies
\[
\ud p_t(x) = p_t(x)\langle x-a_t, \ud B_t\rangle \quad \mbox{for}~x \in \R^n. 
\]
Hence, for a bounded continuous function $f \colon \R^n \to \R$, the process $\{\E_{X \sim \mu_t}[f(X)]\}_{t \geq 0}$ is a continuous martingale. Explicitly, we have
\[
\ud \E_{X \sim \mu_t} f(X) = 
\langle \Cov_{\mu_t}(X, f(X)), \ud B_t\rangle.
\]
Above, we denote the cross-covariance by
\[
\Cov_{X \sim \mu_t}(X, f(X)) = \E_{X \sim \mu_t} [f(X) X ] - \E_{X \sim \mu_t}[f(X)] \E_{X \sim \mu_t}[X]. 
\]
Applying this to $x \mapsto x_i$ or $x\mapsto x_ix_j$, we obtain the dynamics
\begin{equation}
\label{eqn:dynamics-of-at}    
\ud a_t = A_t \,\ud B_t \quad \mbox{and} \quad \ud A_t = \sum_{i=1}^n H_{e_i, t} \ud B_t^i - A_t^2 \, \ud t.
\end{equation}
Above, we use the third-order tensor
\[
H_{\theta, t} = \E_{X \sim \mu_t}\Big[\langle X - a_t, \theta\rangle (X-a_t) \otimes (X -a_t)\Big], \quad \theta \in \sphere.
\]
We refer to a probability measure $\mu$ as being isotropic if, for $X \sim\mu$,\footnote{To avoid confusion, a random vector $X \sim \mu$ (or its law $\mu$) has \emph{scalar covariance} if $\Cov(\mu)=\Cov(X) = \rho I$ for some $\rho > 0$. We reserve \emph{isotropic} for the case $\rho = 1$.}
\[
\E_{\mu} X = 0, \quad \mbox{and} \quad 
\E_{\mu}\, \Big[X \otimes X\Big] =  I_n.
\]

A key lemma in the proofs of the comparison theorems in this section is the following comparison inequality for continuous martingales which have quadratic variation bounded from below. The result below can be traced back at least to a result of Borkar~\cite[Theorem~1 and Remark 3]{Bor87}. An alternative formulation, with a proof using the heat semigroup, is given in 
~\cite[Proposition 3.2]{BizKla26}.
In the result below, we denote by $\cS^n_+$ the family of real, symmetric, positive semidefinite $n \times n$ matrices.


\begin{lemma}
\label{lem:comparison-result}
Let $\{M_t\}_{0\leq t\leq T}$ be an $\R^n$-valued, continuous $L^2$ martingale, with $M_0=0$. 
Suppose that, for some predictable, $\cS^n_+$-valued process $\{C_t\}_{0 \leq t \leq T}$, $\ud [M]_t = C_t \, \ud t$. If, for some $\rho>0$,
$C_t \succeq \rho I_n$ holds almost surely, then 
for every convex function $F \colon \R^n\to\R$,
\[
\E F(M_T)\geq \E F\bigl(\sqrt{\rho}\,B_T\bigr),
\]
where $\{B_t\}_{0 \leq t\leq T}$ is a standard Brownian motion in $\R^n$ indexed by $[0, T]$.
\end{lemma}


The next lemma bounds the Poincar\'{e} constants of the measures arising from stochastic localization on 
$\gamma_{\delta, K}$. 

\begin{lemma}
\label{lem:computation-for-conditional-measure}
Let $\mu = \gamma_{\delta, K}$, where $\delta > 0$ and $K \subset \R^n$ is a convex body. With probability $1$, the stochastic localization process satisfies
\[
\opnorm{A_t} = \opnorm{\Cov(\mu_t)} 
\leq \frac{1}{\delta^2 + t} \leq \frac{1}{\delta^2}
\quad \mbox{for all}~t \geq 0.
\]    
\end{lemma}
\begin{proof}
At time $t$, the measure $\mu_t$ has density relative to $\mu$ given by
\[
\frac{\ud \mu_t}{\ud \mu}(x) \propto \e^{\langle \theta_t, x\rangle - t\|x\|_2^2/2}.
\]
In particular, since $\mu = \gamma_{\delta, K}$, 
completing the square shows that for $X_t \sim \mu_t$,
\[
\cL(X_t) = \cL(Z_t \mid \{ Z_t \in  K \})
\quad \mbox{where} \quad 
Z_t \sim N\Big(\frac{\theta_t}{\delta^2 + t}, \frac{1}{\delta^2 + t} I_n\Big).
\]
By~\Cref{lem:poin-const-of-conditioned-gaussians}, for all $t \geq 0$, we have
\[
\PoinConst(\mu_t) \leq \frac{1}{\delta^2 + t} \leq \frac{1}{\delta^2}. 
\]    
Hence, by considering linear functionals, $\opnorm{A_t} \leq \PoinConst(\mu_t) \leq 1/\delta^2$. 
\end{proof}

We will need the following standard lemma, which controls the third-order tensors. The proof is simply an application of Cauchy-Schwarz. 
\begin{lemma}
\label{lem:mean-zero-third-tensor}
Let $Z \sim \mu$ be a mean-zero random vector in $\R^n$. Define the tensor 
\[
H_{\theta} = \E\Big[\langle Z, \theta \rangle\,  Z \otimes Z \Big], \quad \theta \in \sphere.
\]
It follows that
\[
\opnorm{\sum_{i=1}^n H_{e_i}^2} \leq 4\, \PoinConst(\mu) \, \opnorm{\Cov(\mu)}^2.
\]
\end{lemma}

\subsection{Comparison theorem for $K^\circ$}

The main result in this section is the following comparison inequality between the Gaussian measure $\gamma_n$ and the conditioned Gaussian measure $\gamma_{K^\circ}$ for $K^\circ$ in Bobkov's position. Recall that a \emph{gauge} is a nonnegative, positively $1$-homogeneous, convex function.

\begin{theorem}
\label{thm:comparison-theorem}
There exists a constant~$C_{\ref{thm:comparison-theorem}} \in (0, \infty)$ such that for any $n \geq 1$ and any convex body $K \subset \R^n$ containing the origin in its interior, with $K^\circ$ in $\Bobkov$-position and $\vol_n(K^\circ) = 1$, we have
\[
\E_{\gamma_n} \gauge{G} \leq C_{\ref{thm:comparison-theorem}}
 \Big(\sqrt{\log(\e n/u)} \E_{\gamma_{K^\circ}} \gauge{X}  + \Lip\big(\gauge{\cdot}\big) \sqrt{u} \Big),
\]
for any gauge $\gauge{\cdot} \colon \R^n \to \R_+$ and any $u \in (0, n]$. 
\end{theorem}

An important corollary of this result is the following estimate for the truncation bodies. 

\begin{corollary}
\label{cor:truncation-body}
There exists a constant~$C_{\ref{cor:truncation-body}} \in (0, \infty)$ such that for any $n \geq 1$ and any convex body $K \subset \R^n$ containing the origin in its interior, with $K^\circ$ in $\Bobkov$-position and $\vol_n(K^\circ) = 1$, we have
\[
\sup_{x \in K}w\Big((K-x) \cap \frac{s}{\sqrt{n}} B^n_2\Big) \leq  C_{\ref{cor:truncation-body}}\sqrt{\log(\e s^2)}, 
\]
for every $s \geq 1$. 
\end{corollary}
\begin{proof}
Apply~\Cref{thm:comparison-theorem} to the gauge $\gauge{\cdot} = h_{n(K-x) \cap s\sqrt{n} B^n_2}$. Observe that for $X \sim \gamma_{K^\circ}$, it holds that 
\[
\E h_{n(K-x) \cap s \sqrt{n} B^n_2}(X) \leq n \E h_{K-x}(X) = n \E h_{K}(X) \leq n, 
\]
with the equality arising as $X$ is centered (\Cref{thm:Bobkov-non-symmetric}) and the final inequality as $X \in K^\circ$ almost surely. 
Moreover, it clearly holds that $\Lip(h_{n(K-x) \cap s\sqrt{n} B^n_2}) \leq s\sqrt{n}$. Hence,~\Cref{thm:comparison-theorem} yields 
\begin{align*}
\sup_{x \in K} w(n(K-x) \cap s\sqrt{n} B^n_2) &\leq 
C_{\ref{thm:comparison-theorem}}\, 
\inf_{0 < u \leq n} 
\Big\{n \sqrt{\log(\e n/u)} + s\sqrt{n u}\Big\}\\ 
&= 
C_{\ref{thm:comparison-theorem}}\, 
\inf_{t \geq 1} 
\Big\{n \sqrt{\log(\e t^2)} + \frac{sn}{t}\Big\}\\ 
&\lesssim n \sqrt{\log(\e s^2)}.
\end{align*}
The final inequality arose by taking $t^2 \simeq s^2/\log(\e s^2)$.
Dividing by $n$ yields the claim. 
\end{proof}

Next, we show that control of the trace exponential of the covariance $A_t$ up to time $T$ yields a Gaussian comparison inequality. In the result below, we define
\[
f_\beta(A) = \trace\Big(\exp\big(\beta(\tfrac{1}{2} I - A)\big)\Big), \quad \beta \geq 0, 
\]
for a real symmetric matrix $A$. 

\begin{proposition}
\label{prop:new-f-beta}
Let $X \sim \mu$ for a centered, compactly supported probability measure $\mu$ on $\R^n$. Let $A_t$ denote the covariance of $\mu_t$, the stochastic localization process associated to $\mu$.  Then, for any $T > 0$ 
and any gauge $\gauge{\cdot}$ on $\R^n$, we have, for any $\beta \geq 0$,
\[
\E_{\gamma_n} \gauge{G} \leq \frac{2}{\sqrt{T}}\,  \E_\mu \gauge{X} + \Lip\big(\gauge{\cdot} \big) \, \sqrt{\sup_{0 \leq t \leq T} \E f_\beta(A_t)}.
\]
\end{proposition}
\begin{proof}
Let $\phi(x) = \max\{x, 1/2\}$. We write the eigendecompositions 
\[
A_t = \sum_{i=1}^n \lambda_i^{(t)} z_i^{(t)} \otimes z_i^{(t)}, \quad \mbox{and} \quad 
\Sigma_t 
= \sum_{i=1}^n \phi\big(\lambda_i^{(t)}\big) z_i^{(t)} \otimes z_i^{(t)},
\]
where $z_i^{(t)}$ and $\lambda_i^{(t)}$ denote eigenvector-eigenvalue pairs of the covariance matrix $A_t$.
We define 
\[
v_t = \int_0^t \, \Sigma_s \, \ud B_s,
\]
which is a continuous martingale. Note that 
it also holds that 
\[
a_t = \int_0^t A_s \, \ud B_s.
\]
For any $t \geq 0$, it holds that
$\opnorm{\Sigma_t - A_t} \leq 1/2$. Consequently, 
\[
\fronorm{\Sigma_t - A_t}^2 \leq \frac{1}{4}  \rank(\Sigma_t - A_t) \leq \frac{1}{4} f_\beta(A_t).
\] 
The final inequality used that if $\lambda_i(A_t) \leq 1/2$, then $\e^{\beta(1/2 -\lambda_i(A_t))} \geq 1$.  
By It\^o's isometry, it follows that
\begin{equation*}
\E \|v_T - a_T\|_2^2 = 
\int_0^{T} \E \fronorm{\Sigma_t - A_t}^2 
\, \ud t \leq \frac{T}{4}  \, \sup_{0 \leq t \leq T} \E f_\beta(A_t).
\end{equation*}
The triangle inequality, Cauchy-Schwarz, and the above bound yield
\begin{align}
\E \gauge{v_T}
&\leq \E \gauge{a_T} + \E \gauge{v_T -a_T} \nonumber \\
&\leq \E \gauge{a_T} + \Lip\big(\gauge{\cdot} \big)
\sqrt{\E \|v_T - a_T\|_2^2} \nonumber \\
&\leq 
\E \gauge{a_T} + \Lip\big(\gauge{\cdot} \big) \frac{\sqrt{T}}{2} \sup_{0 \leq t \leq T} \sqrt{\E f_\beta(A_t)} \; \leq  \;
\E \gauge{X} + \Lip\big(\gauge{\cdot} \big) \frac{\sqrt{T}}{2} \sup_{0 \leq t \leq T} \sqrt{\E f_\beta(A_t)}\label{ineq:bound-on-v-t-new},
\end{align}
where we applied Jensen's inequality to $a_T = \E[X \mid \mathcal{F}_T]$, where $\mathcal{F}_T$ denotes the filtration at time $T$. Observe that $\ud [v]_t = \Sigma_t^2 \, \ud t$ and  $\Sigma_t^2 \succeq \tfrac{1}{4} I_n$. Thus, \Cref{lem:comparison-result}, applied to $F = \gauge{\cdot}$, $M_t = v_t$, $\rho = 1/4$, yields, together with~\eqref{ineq:bound-on-v-t-new},
\begin{equation*}
\frac{1}{2} \sqrt{T} \E_{\gamma_n} \gauge{G} 
= \frac{1}{2} \E \gauge{B_T} \leq \E \gauge{v_T} \leq 
\E \gauge{X} + \Lip\big(\gauge{\cdot} \big) \frac{\sqrt{T}}{2} \sup_{0 \leq t \leq T} \sqrt{\E f_\beta(A_t)}.
\end{equation*}
Rearranging yields the claimed result.
\end{proof}

Next, we control $f_\beta$ along the stochastic localization process. For this, we require some computations for the derivatives of the trace exponential of a real symmetric matrix.

\begin{lemma}
\label{lem:trace-exp-deriv}
Let $g(Y) = \tr(\e^Y)$. Then 
\[
D g(Y)(H) = \tr(\e^Y H), \quad \mbox{and} \quad 
D^2 g(Y)(H, H) \leq \tr(\e^Y H^2), 
\]    
for any real, symmetric $Y, H \in \R^{n \times n}$. 
\end{lemma}

We now carry out an It\^o computation to bound the expected trace exponential of the covariance process. 

\begin{lemma}
\label{lem:ito-for-matrix-exp}
For any integer $n \geq 1$, 
any $\delta > 0$, and any convex body $K \subset \R^n$ for which $\gamma_{\delta, K}$ is centered and isotropic,
\[
\E f_\beta(A_t) \leq n\,\exp\Big\{-\frac{\beta}{2} + \int_0^t \Big(\frac{\beta}{\rho_s^2} + \frac{2\beta^2}{\rho_s^3}\Big) \, \ud s\Big\},
\]
for all $t \geq 0$ and any $\beta \geq 0$. Above, $\rho_t = \delta^2 + t$. 
\end{lemma}
\begin{proof}
Since $f_\beta(A_t) = \tr(\e^{\beta X_t})$ with $X_t = \tfrac{1}{2}I_n - A_t$, It\^o's formula, the dynamics~\eqref{eqn:dynamics-of-at} and~\Cref{lem:trace-exp-deriv} yield
\begin{align*}
\ud f_\beta(A_t) &= -\beta \sum_{i=1}^n \tr(\e^{\beta X_t}H_{e_i, t}) \, \ud B^i_t + \beta \tr(\e^{\beta X_t} A_t^2) \, \ud t + \frac{1}{2} \sum_{i=1}^n D^2f_\beta(A_t)(H_{e_i, t}, H_{e_i, t}) \, \ud t \\ 
&\leq  -\beta \sum_{i=1}^n \tr(\e^{\beta X_t}H_{e_i, t}) \, \ud B^i_t + \beta \opnorm{A_t}^2 f_\beta(A_t) \, \ud t + \frac{\beta^2}{2} \tr\Big(\e^{\beta X_t} \sum_{i=1}^n H_{e_i, t}^2\Big) \, \ud t  \\ 
&\leq 
 -\beta \sum_{i=1}^n \tr(\e^{\beta X_t}H_{e_i, t}) \, \ud B^i_t + \bigg[\beta \opnorm{A_t}^2 + \frac{\beta^2}{2} \opnorm{\sum_{i=1}^n H_{e_i,t}^2}\bigg] \, f_\beta(A_t)\, \ud t  \\ 
&\leq 
 -\beta \sum_{i=1}^n \tr(\e^{\beta X_t}H_{e_i, t}) \, \ud B^i_t + \Big(\frac{\beta}{\rho_t^2} + \frac{2\beta^2}{\rho_t^3}\Big) \, f_\beta(A_t)\, \ud t.
\end{align*}
The final inequality used \Cref{lem:computation-for-conditional-measure,lem:mean-zero-third-tensor}.
Observe that the first term is a continuous martingale. Indeed, by Cauchy-Schwarz,
\[
\sum_{i=1}^n \tr(\e^{\beta X_t}  H_{e_i,t})^2  \leq 
f_\beta(A_t) \sum_{i=1}^n  \tr(\e^{\beta X_t} H_{e_i, t}^2) 
\leq f_\beta(A_t)^2 \opnorm{\sum_{i=1}^n H_{e_i,t}^2} \leq 4\delta^{-6}n^2 \e^{\beta},
\]
as $A_t\succeq 0$. The last inequality used~\Cref{lem:mean-zero-third-tensor} again. This shows that for each $T \in (0, \infty)$, 
\[
\E \int_0^T \sum_{i=1}^n \tr(\e^{\beta X_t}H_{e_i, t})^2 \, \ud t \leq 4\delta^{-6}n^2 \e^{\beta} T < \infty.
\]
Thus, we obtain the differential inequality (for $t \mapsto \E f_\beta(A_t)$):
\[
\frac{\ud}{\ud t}\E f_\beta (A_t) \leq \Big(\frac{\beta}{\rho_t^2} + \frac{2\beta^2}{\rho_t^3}\Big) \E f_\beta(A_t).
\]
Since $f_\beta(A_0) = f_\beta(I_n) = n\e^{-\beta/2}$, Gr\"onwall's lemma implies the claim.
\end{proof}

We now establish~\Cref{thm:comparison-theorem}. 

\begin{proof}[Proof of~\Cref{thm:comparison-theorem}]
Since $K^\circ$ is in $\Bobkov$-position, the probability measure $\gamma_{K^\circ}$ is centered and has scalar covariance, by~\Cref{thm:Bobkov-non-symmetric}. Write $\Cov(\gamma_{K^\circ}) = \alpha_{K^\circ}^2 I_n$. Let $X \sim \gamma_{K^\circ}$ and $Y = \alpha_{K^\circ}^{-1} X$. Note that the law of $Y$ is the probability measure $\gamma_{\alpha_{K^\circ}, \alpha_{K^\circ}^{-1} K^\circ}$, which is centered, isotropic, and supported on $\alpha_{K^\circ}^{-1} K^\circ$. Fix $u \in (0, n]$. 
Set
\[
\beta = 4\log (\e n/u), \quad \mbox{and} \quad T = \frac{\alpha_{K^\circ}^6}{12\beta} = \frac{\alpha_{K^\circ}^6}{48 \log (\e n/u)}.
\]
By~\Cref{lem:ito-for-matrix-exp} and using $\beta \geq 1 \geq \alpha_{K^\circ}^2$ and $\rho_t \geq \alpha_{K^\circ}^2$, it holds that 
\[
\sup_{0 \leq t \leq T} \, \E f_\beta(A_t) \leq n\,\exp\Big\{-\frac{\beta}{2} + \Big(\frac{\beta}{\alpha_{K^\circ}^4} + \frac{2\beta^2}{\alpha_{K^\circ}^6}\Big) T\Big\} \leq 
n \exp\Big\{-\frac{\beta}{2} + 3\frac{\beta^2}{\alpha_{K^\circ}^6} T\Big\} \leq u.
\]
Therefore, by~\Cref{prop:new-f-beta} for any gauge $\gauge{\cdot}$, we obtain
\begin{align*}
\E_{\gamma_n} \gauge{G} &\leq \frac{8\sqrt{3}}{\alpha_{K^\circ}^3} \sqrt{\log(\e n/u)}\,  \E \gauge{Y} + \Lip\big(\gauge{\cdot}\big) \sqrt{u}  \\ 
&= \ \frac{8 \sqrt{3}}{\alpha_{K^\circ}^4} \sqrt{\log(\e n/u)}\,  \E \gauge{X} + \Lip\big(\gauge{\cdot}\big) \sqrt{u} \\ 
&\leq C' \Big(\sqrt{\log(\e n/u)} \E \gauge{X}  + \Lip\big(\gauge{\cdot}\big) \sqrt{u} \Big).
\end{align*}
The final display used \Cref{cor:bobkov} to bound $\alpha_{K^\circ}^2$ below by $c_{\ref{cor:bobkov}}$ and we took
$C' = \frac{8\sqrt{3}}{c_{\ref{cor:bobkov}}^2}$.
\end{proof}

\subsection{Comparison theorem for $K$}

We now use the stochastic localization technique to provide a comparison inequality between the Gaussian measure $\gamma_n$ and a measure which is supported on the boundary of $K$, again when $K^\circ$ is in $\Bobkov$-position.
To state the result, we require the probability measure on $\partial K^\circ$ which has density
\[
\frac{\ud \nu_{K^\circ}}{\ud \mathcal{H}^{n-1}}(x)  = 
\frac{\langle x, n_{K^\circ}(x)\rangle \e^{-\|x\|_2^2/2}}{\int_{\partial K^\circ} \langle x, n_{K^\circ}(x)\rangle \e^{-\|x\|_2^2/2} \, \ud \mathcal{H}^{n-1}(x)}.
\]
Above, $n_{K^\circ}$ is the outer unit normal on $\partial K^\circ$, and $\mathcal{H}^{n-1}$ denotes the $(n-1)$-dimensional Hausdorff measure. Note that throughout this section, we will always assume, without loss of generality, that $K^\circ$ has (say) $C^\infty$-boundary.
\begin{theorem}
\label{thm:comparison-theorem-polar}
There exists a constant~$C_{\ref{thm:comparison-theorem-polar}} \in (0, \infty)$ such that for any $n \geq 1$ and any convex body $K$ with $K^\circ$ in $\Bobkov$-position and $\vol_n(K^\circ) = 1$, we have
\[
\E_{\gamma_n} \gauge{G} \leq C_{\ref{thm:comparison-theorem-polar}}
 \Big(n\sqrt{\log(\e n/u)} \E_{\nu_{K^\circ}} \gauge{\nabla h_{K}(X)}  + \Lip\big(\gauge{\cdot}\big) \sqrt{u} \Big),
\]
for any gauge $\gauge{\cdot} \colon \R^n \to \R_+$ and any $u \in (0, n]$. 
\end{theorem}

An important corollary of this result is the following estimate for the truncation bodies. 

\begin{corollary}
\label{cor:truncation-body-polar}
There exists a constant~$C_{\ref{cor:truncation-body-polar}} \in (0, \infty)$ such that for any $n \geq 1$ and any convex body $K \subset \R^n$ containing the origin in its interior, with $K^\circ$ in $\Bobkov$-position and $\vol_n(K^\circ) = 1$, we have
\[
\sup_{x \in K^\circ} w\Big((K^\circ-x) \cap s \sqrt{n} B^n_2\Big) \leq  C_{\ref{cor:truncation-body-polar}}\, n\sqrt{\log(\e s^2)}, 
\]
for every $s \geq 1$. 
\end{corollary}
\begin{proof}
Apply~\Cref{thm:comparison-theorem-polar} to the gauge $\gauge{\cdot} = h_{(K^\circ-x) \cap s\sqrt{n} B^n_2}$ with $x \in K^\circ$. Observe that for $X \sim \nu_{K^\circ}$, it holds that 
\begin{align*}
\E_{X \sim \nu_{K^\circ}} \nabla h_K(X) &\propto \int_{\partial K^\circ} \nabla h_K(x) \langle x, n_{K^\circ}(x)\rangle \e^{-\|x\|_2^2/2} \, \ud \mathcal{H}^{n-1}(x) \\ 
&= \int_{\partial K^\circ} n_{K^\circ}(x) \e^{-\|x\|_2^2/2} \, \ud \mathcal{H}^{n-1}(x) \\ 
&= -\int_{K^\circ} x \e^{-\|x\|_2^2/2} \, \ud x \propto - \E_{Y \sim \gamma_{K^\circ}}[Y] = 0.
\end{align*}
The last display follows from the divergence theorem and from~\Cref{thm:Bobkov-non-symmetric}. Hence, for $X \sim \nu_{K^\circ}$, it holds that 
\[
\E h_{(K^\circ-x) \cap s \sqrt{n} B^n_2}(\nabla h_{K}(X)) \leq \E h_{K^\circ-x}(\nabla h_{K}(X)) = \E h_{K^\circ}(\nabla h_{K}(X)) \leq 1.
\]
The final inequality holds because $\nabla h_K(X) \in K$ almost surely.
Moreover, it clearly holds that
\[\Lip(h_{(K^\circ-x) \cap s\sqrt{n} B^n_2}) \leq s\sqrt{n}.
\] Hence,~\Cref{thm:comparison-theorem-polar} yields 
\begin{align*}
\sup_{x \in K^\circ} w((K^\circ-x) \cap s\sqrt{n} B^n_2) &\leq 
C_{\ref{thm:comparison-theorem-polar}}\, 
\inf_{0 < u \leq n} 
\Big\{n \sqrt{\log(\e n/u)} + s\sqrt{n u}\Big\}\\ 
&= 
C_{\ref{thm:comparison-theorem-polar}}\, 
\inf_{t \geq 1} 
\Big\{n \sqrt{\log(\e t^2)} + \frac{sn}{t}\Big\}\lesssim n \sqrt{\log(\e s^2)}.
\end{align*}
The final inequality arose by taking $t^2 \simeq s^2/\log(\e s^2)$.
\end{proof}

For a convex body $L \subset \R^n$ which contains the origin in its interior, we denote the outer unit normal by $n_L \colon \partial L \to \sphere$. If $L$ has a sufficiently smooth boundary, we can write 
\[
n_L(x) = r_L(x) \nabla h_{L^\circ}(x), \quad \mbox{where} \quad r_L(x) = \frac{1}{\|\nabla h_{L^\circ}(x)\|_2} = \langle x, n_L(x)\rangle.
\]
Given the stochastic localization process $\{\mu_t\}$ with densities $p_t= \tfrac{\ud \mu_t}{\ud x}$, let $\mu^{\partial}_t$ denote the boundary measure on $\partial L$ defined by
\[
\ud \mu^{\partial}_t(x) = p_t(x)\,\ud \mathcal{H}^{n-1}(x).
\]
We define the family of probability measures $\{\nu_t\}$ which are supported on $\partial L$ and are given by
\[
\nu_t(A) = \frac{1}{Z_t}\int_{A \cap \partial L} r_L(x)\, \ud \mu^{\partial}_t(x), \quad \mbox{where} \quad Z_t = \int_{\partial L} r_L(x) \, \ud \mu^{\partial}_t(x),
\]
for any Borel set $A \subset \partial L$. 

\medskip 

For the stochastic localization process associated to $\gamma_{\tau,L}$ and a fixed $\eps\in(0,1)$, write
\[
\rho_t = \tau^2+t, \quad \mbox{and} \quad
g_\beta(M) = \tr\Big(\exp\big(\beta(M-(1-\eps)I_n)\big)\Big), \quad \beta \geq 0.
\]
We first show how a bound on the expected trace yields a comparison inequality. 
\begin{proposition}
\label{prop:new-g-beta-polar}
Let $L\subset\R^n$ be a convex body with the origin in its interior. Let $\tau>0$ be such that $\mu=\gamma_{\tau,L}$ is centered. Let $\{A_t\}_{t\geq0}$ be the covariance process associated to its stochastic localization process. Set $\rho_t=\tau^2+t$. Then, for every $T>0$, $\beta\geq0$, $\eps\in(0,1)$, and every gauge $\gauge{\cdot}$ on $\R^n$, we have
\[
\E_{\gamma_n}\gauge{G}
\leq\frac{1 -\tfrac{\tau^2}{n}\tr(\Cov(\mu))}{\eps\sqrt{T}}\,n\,
\E_{\nu_0}\gauge{\nabla h_{L^\circ}(X)} 
+\Lip\big(\gauge{\cdot}\big)
\sqrt{\sup_{0\leq t\leq T}\E g_\beta(\rho_t A_t)}.
\]
\end{proposition}
\begin{proof}
Consider the process
\[
a_t^\circ=\theta_t-\rho_ta_t.
\]
Since $a_0=\theta_0=0$, It\^o's product rule yields
\[
a_t^\circ=\int_0^t(I_n-\rho_sA_s)\,\ud B_s.
\]
We have $0\preceq\rho_tA_t\preceq I_n$ by~\Cref{lem:computation-for-conditional-measure}. Write
\[
A_t=\sum_{i=1}^n\lambda_i^{(t)}z_i^{(t)}\otimes z_i^{(t)}, \quad \mbox{and} \quad
A_t^+=\sum_{i=1}^n\min\Big\{\lambda_i^{(t)},\frac{1-\eps}{\rho_t}\Big\}\,
z_i^{(t)}\otimes z_i^{(t)}.
\]
Then $0\preceq\rho_t(A_t-A_t^+)\preceq\eps I_n$, and a nonzero eigenvalue of $A_t-A_t^+$ corresponds to an eigenvalue $\rho_t\lambda_i^{(t)}>1-\eps$, whose contribution to $g_\beta(\rho_t A_t)$ is at least $1$. Consequently,
\begin{equation}
\label{ineq:upper-on-frob-polar}
\rho_t^2\fronorm{A_t-A_t^+}^2\leq\eps^2\rank(A_t-A_t^+)
\leq\eps^2g_\beta(\rho_t A_t).
\end{equation}
Define the continuous martingale
\[
a_t^{\circ,+}=\int_0^t(I_n-\rho_sA_s^+)\,\ud B_s.
\]
Observe that
\[
\ud[a^{\circ,+}]_t=(I_n-\rho_tA_t^+)^2\,\ud t\succeq\eps^2I_n\,\ud t, \quad \mbox{and} \quad 
a_t^{\circ,+}-a_t^\circ=\int_0^t\rho_s(A_s-A_s^+)\,\ud B_s.
\]
It\^o's isometry and the bound~\eqref{ineq:upper-on-frob-polar} therefore give
\[
\E\|a_T^{\circ,+}-a_T^\circ\|_2^2
=\int_0^T\rho_t^2\E\fronorm{A_t-A_t^+}^2\,\ud t
\leq\eps^2T
\sup_{0\leq t\leq T}\E g_\beta(\rho_t A_t).
\]
Applying~\Cref{lem:comparison-result}, and then the triangle inequality and Cauchy--Schwarz, we obtain
\begin{equation}
\eps\sqrt{T}\,\E_{\gamma_n}\gauge{G}
\leq \E\gauge{a_T^{\circ,+}}
\leq\E\gauge{a_T^\circ}
+\Lip\big(\gauge{\cdot}\big)\eps\sqrt{T}
\sqrt{\sup_{0\leq t\leq T}\E g_\beta(\rho_t A_t)}.
\label{ineq:trace-exp-polar-martingale-comparison-new}
\end{equation}
It remains to bound $\E\gauge{a_T^\circ}$. Using $\nabla p_t(x)=(\theta_t-\rho_t x)p_t(x)$ and the divergence theorem, we obtain
\[
a_t^\circ
=\int_L(\theta_t-\rho_t x)p_t(x)\,\ud x
=\int_{\partial L}n_L(x)\,\ud \mu^{\partial}_t(x)\\
=\int_{\partial L}r_L(x)\nabla h_{L^\circ}(x)\,\ud \mu^{\partial}_t(x).
\]
By convexity and positive homogeneity of the gauge,
\[
\gauge{a_t^\circ}\leq\int_{\partial L}
r_L(x)\gauge{\nabla h_{L^\circ}(x)}\,\ud \mu^{\partial}_t(x).
\]
Recall that the density $p_t$ satisfies, for each $x$,
$\ud p_t(x)=p_t(x)\langle x-a_t,\ud B_t\rangle$.
Hence $\{p_t(x)\}_{t \geq 0}$ is a continuous martingale and $\E p_t(x)=p_0(x)$.  Taking expectations and using Fubini's theorem, we obtain
\[
\E\gauge{a_T^\circ}\leq\int_{\partial L}
r_L(x)\gauge{\nabla h_{L^\circ}(x)}\,\ud \mu^{\partial}_0(x)
=Z_0\E_{\nu_0}\gauge{\nabla h_{L^\circ}(X)}.
\]
Finally, another application of the divergence theorem, together with the centering of $\mu$, gives
\[
Z_0=\int_{\partial L}\langle x,n_L(x)\rangle\,\ud \mu^{\partial}_0(x)
=\int_L(n-\tau^2\|x\|_2^2)p_0(x)\,\ud x
=n\Big(1 -\frac{\tau^2}{n}\tr(A_0)\Big).
\]
Substituting into~\eqref{ineq:trace-exp-polar-martingale-comparison-new} and dividing by $\eps\sqrt{T}$ yields the claim. 
\end{proof}

We now bound the expected trace exponential of $\rho_t A_t$, using an It\^o computation.

\begin{lemma}
\label{lem:ito-for-matrix-exp-polar-new}
Let $\tau>0$ and let $L\subset\R^n$ be a convex body such that $\gamma_{\tau,L}$ is centered and isotropic. Then the covariance process $\{A_t\}_{t\geq0}$ associated to the stochastic localization of $\gamma_{\tau, L}$ satisfies
\[
\E g_\beta(\rho_t A_t) \leq n\exp\Big\{-\beta((1 - \tau^2)  - \eps ) + \int_0^t \frac{\beta + 2\beta^2}{\rho_s} \, \ud s\Big\},
\]
for every $t\geq0$, every $\beta\geq 0$, and every $\eps\in(0,1)$.
\end{lemma}
\begin{proof}
By~\Cref{lem:computation-for-conditional-measure} and~\Cref{lem:mean-zero-third-tensor}, we have
\[
0\preceq \rho_t A_t\preceq
I_n, \qquad
\PoinConst(\mu_t)\leq\frac{1}{\rho_t}, \quad \mbox{and}\quad
\opnorm{\sum_{i=1}^n H_{e_i,t}^2}\leq\frac{4}{\rho_t^3}.
\]
Set $Y_t=\rho_tA_t-(1-\eps)I_n$. It\^o's product rule and the dynamics~\eqref{eqn:dynamics-of-at} yield
\[
\ud Y_t = \rho_t\sum_{i=1}^n H_{e_i,t}\,\ud B_t^i
+ (A_t-\rho_tA_t^2)\,\ud t.
\]
Write $g(Y)=\tr(\e^Y)$, as in~\Cref{lem:trace-exp-deriv}, so that $g_\beta(\rho_t A_t)=g(\beta Y_t)$. Define 
\[
N_t = \beta\int_0^t\rho_s\sum_{i=1}^n
\tr\big(\e^{\beta Y_s}H_{e_i,s}\big)\,\ud B_s^i.
\]
It\^o's formula yields
\begin{align*}
\ud g_\beta(\rho_t A_t)
&= \ud N_t + \beta\tr\big(\e^{\beta Y_t}(A_t-\rho_tA_t^2)\big)\,\ud t \\
&\hspace{1cm}+\frac{\beta^2\rho_t^2}{2}\sum_{i=1}^n
D^2g(\beta Y_t)(H_{e_i,t},H_{e_i,t})\,\ud t \\
&\leq \ud N_t + \beta\opnorm{A_t}\,g_\beta(\rho_t A_t)\,\ud t
+\frac{\beta^2\rho_t^2}{2}\tr\Big(\e^{\beta Y_t}\sum_{i=1}^n H_{e_i,t}^2\Big)\,\ud t \\
&\leq \ud N_t + \frac{\beta+2\beta^2}{\rho_t}\,g_\beta(\rho_t A_t)\,\ud t.
\end{align*}
In fact, $N_t$ is a continuous martingale, since it is locally square-integrable. Indeed, first observe that $Y_t$ satisfies the bounds ${-(1-\eps)I_n\preceq Y_t\preceq\eps I_n}$, and hence $g_\beta(\rho_t A_t)\leq n\e^{\beta\eps}$. Cauchy--Schwarz gives
\[
\frac{\ud[N]_t}{\ud t}
=\beta^2\rho_t^2\sum_{i=1}^n\tr\big(\e^{\beta Y_t}H_{e_i,t}\big)^2 
\leq\beta^2\rho_t^2g_\beta(\rho_t A_t)^2\opnorm{\sum_{i=1}^nH_{e_i,t}^2}
\leq\frac{4\beta^2n^2\e^{2\beta\eps}}{\tau^2}.
\]
Hence, we obtain the differential inequality (for $t \mapsto \E g_\beta(\rho_t A_t)$):
\[
\frac{\ud}{\ud t}\E g_\beta(\rho_t A_t)\leq \frac{\beta+2\beta^2}{\rho_t}\E g_\beta(\rho_t A_t).
\]
Gr\"onwall's lemma and $A_0=I_n$ yield the claim.
\end{proof}

Finally, to obtain the comparison theorem, we need to ensure that $\alpha_{K^\circ}^2$ is bounded away from $1$. 

\begin{proposition}
\label{prop:universal-constant-away-from-1}
There is a constant $c_{\ref{prop:universal-constant-away-from-1}} >0$ such that for any convex body $K \subset \R^n$ with $K^\circ$ in $\Bobkov$-position and with $\vol_n(K^\circ) = 1$, it holds that 
\[
1-\alpha_{K^\circ}^2 \geq c_{\ref{prop:universal-constant-away-from-1}}.
\]
\end{proposition}
\begin{proof}
Combining~\Cref{prop:truncation-width-to-entropy} and~\Cref{cor:truncation-body}, we obtain
\[
\log N(\delta \sqrt{n} B^n_2, r K^\circ) \leq  C \frac{n \delta^2}{r^2} \log \Big( \frac{\e r^2}{\delta^2}\Big), \quad \mbox{for all}~r \geq \delta \geq 1.
\]
Hence, with $M = N(\delta \sqrt{n} B^n_2, r K^\circ)$ there is a minimal cover with centers $x_1, \dots, x_M$ for which
\[
\delta \sqrt{n} B^n_2 \subset \bigcup_{i=1}^M (x_i + r K^\circ).
\]
Markov's inequality gives $\gamma_n((\delta \sqrt{n} B^n_2)^c)\leq \delta^{-2}$, for any $\delta > 0$.  
Denote, for $r > 0$, 
\[
F(r) = \sup_{x \in \R^n}\,\gamma_n(rK^\circ + x), \quad \mbox{and} \quad 
f(\delta, r) =  C \frac{\delta^2}{r^2} \log\Big(\frac{\e r^2}{\delta^2}\Big).
\]
Hence, using $\gamma_n(x_i + rK^\circ) \leq F(r)$ for all $i$, we obtain
\[
(1 - \delta^{-2})^{1/n} \leq \gamma_n(\delta \sqrt{n} B^n_2)^{1/n} \leq \e^{f(\delta, r)} F(r)^{1/n}.
\]
On the other hand, since $\vol_n(K^\circ) = 1$, we obtain by~\Cref{thm:Bobkov-non-symmetric} that 
\[
F(1)^{1/n} =\gamma_n(K^\circ)^{1/n} \leq (2\pi)^{-1/2}.
\]
 Combining these observations, we may find $r \geq \delta > 1$ such that
\[
\bigg(\frac{F(r)}{F(1)}\bigg)^{1/n}\geq \sqrt{2\pi} (1-\delta^{-2})^{1/n} \e^{-f(\delta, r)} \geq 
\sqrt{2\pi} (1-\delta^{-2}) \e^{-f(\delta, r)} > 1.
\]
Indeed, taking logarithms, the final inequality above is equivalent to 
\[
\frac{1}{2} \log(2\pi) - C \frac{1 + 2 \log(r/\delta)}{(r/\delta)^2} > \log\Big(1 +  \frac{1}{\delta^2 - 1}\Big), 
\]
which holds for $r/\delta, \delta$ sufficiently large. By the non-symmetric analogue of the strong (B)-theorem,~\Cref{prop:log-concavity-of-max-shift}, applied with $A = I_n$, the map $ t \mapsto \log F(\e^t)$ is concave.
Hence, with our choice of $r, \delta$, 
\[
n(1-\alpha_{K^\circ}^2) = \Big(\frac{\ud}{\ud t} \log F(\e^t)\Big)\Big|_{t = 0} \geq \frac{n}{\log r} \log\Bigg( \Big(\frac{F(r)}{F(1)}\Big)^{1/n}\Bigg)\gtrsim n.
\]
Dividing the inequality above by $n$, we obtain the result.    
\end{proof}

We now put the computations above together to establish~\Cref{thm:comparison-theorem-polar}. 

\begin{proof}[Proof of~\Cref{thm:comparison-theorem-polar}]
Write $\alpha = \alpha_{K^\circ}, L=\alpha^{-1}K^\circ$. By~\Cref{thm:Bobkov-non-symmetric}, the measure $\mu=\gamma_{\alpha,L}$ is centered and isotropic. 
By~\Cref{prop:universal-constant-away-from-1}, $1-\alpha^2\geq c_{\ref{prop:universal-constant-away-from-1}}>0$. Fix $u\in(0,n]$ and choose
\[
\eps=\frac{1-\alpha^2}{2},\qquad
\beta=\frac{4\log(\e n/u)}{1-\alpha^2},\quad \mbox{and} \quad
T=\frac{\alpha^2(1-\alpha^2)^2}{48\log(\e n/u)}.
\]
By~\Cref{lem:ito-for-matrix-exp-polar-new}, using $\beta \geq 1$ and $\rho_t \geq \alpha^2$, we have 
\[
\sup_{0\leq t\leq T}\E g_\beta(\rho_t A_t)
\leq n\exp\Big\{-\beta((1 - \alpha^2)  - \eps ) + \frac{3\beta^2}{\alpha^2}T \Big\}
=n\exp\Big\{-\frac{\beta (1-\alpha^2)}{4}\Big\}
=\frac{u}{\e}\leq u.
\]
\Cref{prop:new-g-beta-polar} yields
\begin{align*}
\E_{\gamma_n}\gauge{G}
&\leq\frac{8\sqrt{3}}{\alpha (1-\alpha^2)}\,n\sqrt{\log(\e n/u)}\,
\E_{\nu_0}\gauge{\nabla h_{L^\circ}(X)}
+\Lip\big(\gauge{\cdot}\big)\sqrt{u} \\ 
&= 
\frac{8\sqrt{3}}{1-\alpha^2} \, n\sqrt{\log(\e n/u)}\,
\E_{\nu_{K^\circ}}\gauge{\nabla h_K(X)}
+\Lip\big(\gauge{\cdot}\big)\sqrt{u}.
\end{align*}
The upper bound now follows by taking $C_{\ref{thm:comparison-theorem-polar}}=
\tfrac{8\sqrt{3}}{c_{\ref{prop:universal-constant-away-from-1}}}$.
\end{proof}

\section{Proofs of main results}
\label{sec:proofs-of-main-results}

\subsection{Proof of~\Cref{thm:covering-numbers}}

We begin by establishing the following estimate in $\Bobkov$-position.

\begin{proposition}
\label{prop:covering-numbers-in-bobkov-position}
    There are two constants $C_{\ref{prop:covering-numbers-in-bobkov-position}}, c_{\ref{prop:covering-numbers-in-bobkov-position}} \in (0, \infty)$ such that for any $n \geq 1$ and any convex body $K \subset \R^n$ which contains the origin in its interior with $K^\circ$ of unit volume and in the $\Bobkov$-position, it holds that
    \[
    \mathcal{H}(K, r) \leq C_{\ref{prop:covering-numbers-in-bobkov-position}} \, \inf_{\psi'} \mathcal{H}(\psi'(B^n_1), c_{\ref{prop:covering-numbers-in-bobkov-position}} r), \quad \mbox{for all}~r > 0.
    \]
    Here the infimum ranges over all volume-preserving affine maps $\psi' \colon \R^n \to \R^n$. The same bound holds with $B^n_1$ replaced by a regular $n$-simplex $S_n$ of the same volume.
\end{proposition}

\begin{remark}
\label{rem:blashcke-santalo}
The estimate in~\Cref{prop:covering-numbers-in-bobkov-position} also yields for a convex body $K \subset \R^n$ with $K^\circ$ of unit volume and in $\Bobkov$-position that
\begin{equation}
\label{eqn:entropy-like-Milman}
\max\{N(K, B), N(K^\circ, B^\circ), N(B, K), N(B^\circ, K^\circ)\} \leq \e^{Cn}.
\end{equation}
This shows that, even in the non-symmetric case, $\Bobkov$-position is Milman's position. An immediate consequence of~\eqref{eqn:entropy-like-Milman} is that
\[
\vr(K) \simeq \frac{1}{\sqrt{n}}, \quad \vr(K^\circ) \simeq \sqrt{n}, \quad \mbox{and hence,}\quad \vr(K)\vr(K^\circ) \simeq 1.
\]
If $L \subset \R^n$ is a centrally symmetric convex body, then, as discussed in \S\ref{sec:bobkov-non-symmetric}, we may find $T \in \GL(n)$ such that $K = TL$ is such that $K^\circ$ is of unit volume and in $\Bobkov$-position. Consequently,
\[
\vr(L)\vr(L^\circ) =\vr(K)\vr(K^\circ) \simeq 1.
\]
The statement above implies the so-called reverse Blaschke-Santal\'o inequality due to Bourgain-Milman~\cite{BouMil87}. Similarly, in the non-symmetric case, if $L \subset \R^n$ is a convex body, then by~\Cref{prop:existence-uniqueness-of-Bobkov}, we may find $T \in \GL(n)$ and $x \in \intr(L)$ such that $K =T(L-x)$ has volume one and is in $\Bobkov$-position. We then conclude 
\[
\vr(L- x) \vr((L-x)^\circ) = \vr(K)\vr(K^\circ) \simeq 1.
\]
\end{remark}

\begin{remark}
\label{rem:optimal-regularity-bobkov}
Following the notation and discussion in \Cref{rem:pisier-comparison},
\Cref{prop:covering-numbers-in-bobkov-position} implies that if $K \subset \R^n$ is a convex body which contains the origin in its interior and $K^\circ$ has unit volume and is in $\Bobkov$-position, then $\vr(K) \simeq 1/\sqrt{n}$ and
\begin{equation}
\label{eqn:optimal-entropy-numbers-bobkov}
\mathcal{E}_k\Big(\sqrt{n} K\Big) \lesssim 
\inf_{\psi} \, \mathcal{E}_k(\psi(B^n_1)) \simeq 
\begin{cases}
\sqrt{n}, & k \in [1, \log(\e n)] \\[0.5em] 
\sqrt{\frac{n}{k} \log\big(\frac{\e n}{k}\big)}, & k \in (\log(\e n), n] \\[0.5em] 
2^{-k/n}, & k \geq n + 1,
\end{cases}
\end{equation}
simultaneously for all $k \geq 1$. The infimum ranges over affine transformations $\psi \colon \R^n \to \R^n$ such that ${\vol_n(\psi(B^n_1)) = \vol_n(B^n_2)}$. Note that we can also replace $B^n_1$ above by any regular $n$-simplex. In other words, the estimate~\eqref{eqn:optimal-entropy-numbers-bobkov} implies that $\sqrt{n}K$ and its polar $K^\circ/\sqrt{n}$ are in optimal regular Milman's position, in the sense of Pisier~\cite[Ch.~7]{Pis89}.
\end{remark}

To establish~\Cref{prop:covering-numbers-in-bobkov-position}, we need the following estimate on the outradius of $K^\circ$ in $\Bobkov$ position.
\begin{lemma}
\label{lem:outradius-bobkov-k-polar}
There is a constant $C_{\ref{lem:outradius-bobkov-k-polar}} \in (0, \infty)$ such that for any $n \geq 1$ and any convex body $K \subset \R^n$ with the origin in its interior and with $K^\circ$ in $\Bobkov$-position and $\vol_n(K^\circ) = 1$, it holds that ${\rad(K^\circ) \leq C_{\ref{lem:outradius-bobkov-k-polar}} n}$.
\end{lemma}
\begin{proof}
Let $X \sim \gamma_{K^\circ}$ and consider the moment generating function 
\[
m_X(t) = \E \e^{\langle t, X \rangle}, \quad t \in \R^n.
\]
By~\Cref{thm:Bobkov-non-symmetric}, $X$ is centered and has covariance equal to $\alpha_{K^\circ}^2 I_n$, with $\alpha_{K^\circ}^2 \leq 1$. Moreover, it is a log-concave random vector, and it holds that $\|\langle X, \theta \rangle \|_{L^p} \lesssim p$ for every $p \geq 1$. This implies that there exist constants $C_1, c_1 > 0$ such that
\[
m_X(t) \leq 1 + \frac{\|t\|_2^2}{2} \alpha_{K^\circ}^2 + C_1\|t\|_2^3, \quad \mbox{provided}~\|t\|_2 \leq c_1.
\]
In fact, since we additionally have $1 - \alpha_{K^\circ}^2 \gtrsim 1$, by~\Cref{prop:universal-constant-away-from-1}, there are further constants $c_2, c_2' > 0$ sufficiently small such that 
\begin{equation}
\label{eqn:moment-generating-function-expansion}    
m_X(t) \leq \e^{\|t\|_2^2/2} - c_2 \|t\|_2^2,\quad \mbox{provided}~\|t\|_2 \leq c_2'.
\end{equation}
Above, we used the inequality $1+u\leq \e^u$. In the other direction, observe that for $z, t \in \R^n$, the standard Gaussian density satisfies 
\[
\frac{\phi(z - t)}{\phi(z)} = \exp\Big\{-\frac{\|z - t\|_2^2 - \|z\|_2^2}{2}\Big\} = \e^{\langle t, z \rangle -\|t\|_2^2/2}.
\]
Hence, for every $t \in \R^n$,
\begin{multline}
\label{eqn:mgf-of-condtioned}
m_X(t) \\= \frac{1}{\gamma_n(K^\circ)}\int_{\R^n} \1_{K^\circ}(z) \e^{\langle t, z\rangle} \phi(z) \, \ud z = 
\frac{\e^{\|t\|_2^2/2}}{\gamma_n(K^\circ)}\int_{\R^n} \1_{K^\circ}(z) \phi(z-t) \, \ud z
= 
\e^{\|t\|_2^2/2}
\frac{\gamma_n(K^\circ - t)}{\gamma_n(K^\circ)}.
\end{multline}
Take $t = \lambda x$ for $\lambda \in (0, 1)$ and $x \in K^\circ$. 
It obviously holds that 
\[
K^\circ - \lambda x \supset (1-\lambda) K^\circ.
\]
Hence, it holds that 
\begin{equation}
\label{ineq:lower-on-translate}    
\gamma_n(K^\circ - \lambda x) \geq \gamma_n((1-\lambda)K^\circ) \geq (1-\lambda)^n\gamma_n(K^\circ).
\end{equation}
The last inequality follows by integrating and using the fact that $\phi(\lambda x) \geq \phi(x)$ for all $x \in \R^n, \lambda \in (0, 1)$. Consequently, combining~\eqref{eqn:mgf-of-condtioned} and~\eqref{ineq:lower-on-translate},
for every $\lambda \in (0, 1)$ and $x \in K^\circ$, it holds that 
\[
m_X(\lambda x) \geq \e^{\lambda^2 \|x\|_2^2/2}(1-\lambda)^n.
\]
Suppose that $\lambda \leq \min\{1, \tfrac{c_2'}{\|x\|_2}\}$. Then, combined with
\eqref{eqn:moment-generating-function-expansion}, the display above implies that 
\[
c_2 \lambda^2 \|x\|_2^2 \e^{-\lambda^2 \|x\|_2^2/2}\leq 1 - (1-\lambda)^n \leq n\lambda.
\]
Suppose $\|x\|_2 > c_2'$. Then, we may take $\lambda = c_2'/\|x\|_2$. This implies that for some $C_3 > 0$, 
\[
\|x\|_2 \leq C_3 n.
\]
Hence, there exists a sufficiently large $C_4 > 0$ such that for every $x \in K^\circ$, 
\[
\|x\|_2 \leq \max\{c_2', C_3 n\} \leq C_4 n,
\]
as required.
\end{proof}

\begin{proof}
We prove the crosspolytope case; the simplex case follows verbatim by~\Cref{prop:lower-bound-entropy}.
Throughout we assume $n \geq 2$; the case $n=1$  can be computed directly (\ie~for intervals). Set
\begin{equation}
\label{eqn:surrogate-for-crosspolytope-entropy}
H_n(r) =
\begin{cases}
 n\log(\frac{\e}{r}),& 0<r<1,\\[0.5em]
 \frac{n}{r^2}\log(\e r^2),& r\geq1.
\end{cases}
\end{equation}
By \Cref{thm:Bobkov-non-symmetric}, \Cref{cor:bobkov}, and~\Cref{lem:outradius-bobkov-k-polar}, it holds that 
\begin{equation}
\label{eqn:bounds-on-outradii}
\rad(K^\circ) \lesssim n, \quad\mbox{and} \quad
\rad(K) \lesssim 1. 
\end{equation}
By~\Cref{cor:truncation-body} and by applying
\Cref{prop:truncation-width-to-entropy} to $L = nK$, we obtain
\begin{equation}
\label{eqn:first-entropy-pair}
\max\Big\{
 \log N\Bigl(K,\frac r{\sqrt n}B_2^n\Bigr),
 \log N\Bigl(B_2^n,\frac r{\sqrt n}K^\circ\Bigr)
\Big\}\lesssim H_n(r), \quad \mbox{for all}~r > 0.
\end{equation}
Similarly, by \Cref{cor:truncation-body-polar} and by applying
\Cref{prop:truncation-width-to-entropy} to $L = K^\circ$, we obtain
\begin{equation}
\label{eqn:second-entropy-pair}
\max\Big\{
 \log N(K^\circ,r\sqrt nB_2^n),
 \log N(B_2^n,r\sqrt nK)
\Big\}\lesssim H_n(r), \quad \mbox{for all}~r > 0.
\end{equation}
Thus \eqref{eqn:first-entropy-pair} and~\eqref{eqn:second-entropy-pair} imply that
\[
\mathcal H(K,r)\leq C_1 H_n(r),\quad \mbox{for all}~r>0,
\]
for a universal constant $C_1>0$.
By the inequalities~\eqref{eqn:bounds-on-outradii}, 
there exists a universal constant $R \geq 1$ such that 
\[
K \subset RB_2^n,\qquad K^\circ\subset RnB_2^n, \quad \mbox{and hence,}\quad R^{-1}B_2^n \subset K^\circ.
\]
Therefore, if $r \geq R \sqrt{n}$, the inclusions above imply that 
\[
K\subset\frac r{\sqrt n}B_2^n,
\quad B_2^n\subset\frac r{\sqrt n}K^\circ,
\quad K^\circ\subset r\sqrt nB_2^n, \quad \mbox{and} \quad
\quad B_2^n\subset r\sqrt nK.
\]
Thus $\mathcal{H}(K,r)=0$ for $r\geq R\sqrt n$. Take
$c_{\ref{prop:covering-numbers-in-bobkov-position}}
\leq \frac{c_{\ref{prop:lower-bound-entropy}}}{R}$. Then, 
for $0<r\leq R\sqrt n$, the monotonicity of $H_n$ and~\Cref{prop:lower-bound-entropy} give
\begin{multline*}
\mathcal{H}(K,r)
\leq C_1 H_n(r)
\\\leq C_1 H_n\big(c_{\ref{prop:covering-numbers-in-bobkov-position}}r\big)
\leq \frac{C_1}{c_{\ref{prop:lower-bound-entropy}}'}
\inf_{\psi'}\log N\!\Big(\psi'(B_1^n),
\frac{c_{\ref{prop:covering-numbers-in-bobkov-position}}r}{\sqrt n}B_2^n\Big)
\leq \frac{C_1}{c_{\ref{prop:lower-bound-entropy}}'}
\inf_{\psi'}\mathcal H\big(\psi'(B_1^n),c_{\ref{prop:covering-numbers-in-bobkov-position}}r\big).
\end{multline*}
Taking $C_{\ref{prop:covering-numbers-in-bobkov-position}}
=\max\{1,C_1/c_{\ref{prop:lower-bound-entropy}}'\}$ completes the proof.
\end{proof}

We now deduce~\Cref{thm:covering-numbers} from~\Cref{prop:covering-numbers-in-bobkov-position}. 

\begin{proof}[Proof of~\Cref{thm:covering-numbers}]
We prove the statement involving the crosspolytope. Note that the argument for a regular $n$-simplex with the same volume as $B^n_1$ follows verbatim, by simply replacing $B^n_1$ by $S_n$ in the argument below.
Choose an affine image $L=\psi_0(K)$ of volume one in $\Bobkov$-position. By~\Cref{thm:Bobkov-non-symmetric}, $\gamma_L$ is centered, so $0\in\intr(L)$.
Applying~\Cref{prop:covering-numbers-in-bobkov-position} to $L^\circ$ gives
\begin{equation}
\label{eqn:covering-polar-b-position}
\mathcal H(L^\circ,r)
\leq C_{\ref{prop:covering-numbers-in-bobkov-position}}
\inf_{\psi'}\mathcal H\big(\psi'(B_1^n),c_{\ref{prop:covering-numbers-in-bobkov-position}}r\big),
\qquad \mbox{for any}~r>0.
\end{equation}
We set 
\[
\lambda_n=\vol_n(B_1^n)^{1/n}=\frac{2}{(n!)^{1/n}}, \quad \mbox{and} \quad 
 \psi=\lambda_n\psi_0.
\] 
With the normalization above, $\vol_n(\psi(K))=\vol_n(B_1^n)$. Since $2\leq n\lambda_n\leq2\e$, comparison of the four covering numbers in the definition of $\mathcal H$, followed by~\eqref{eqn:covering-polar-b-position}, yields
\[
\mathcal{H}(\psi(K),r)
=\mathcal{H}(\lambda_nL,r)
\leq\mathcal{H}\Big(L^\circ,\frac{r}{2\e}\Big)
\leq C_{\ref{prop:covering-numbers-in-bobkov-position}}
\inf_{\psi'}
\mathcal{H}\Big(\psi'(B_1^n),\frac{c_{\ref{prop:covering-numbers-in-bobkov-position}}}{2\e}
r\Big).
\]
The claim follows with
$c_{\ref{thm:covering-numbers}}=\frac{c_{\ref{prop:covering-numbers-in-bobkov-position}}}{2\e}$ and $
C_{\ref{thm:covering-numbers}}=C_{\ref{prop:covering-numbers-in-bobkov-position}}$.
\end{proof}

\subsection{Proof of~\Cref{thm:quermassintegrals}}

We again begin by establishing the following estimate in $\Bobkov$-position.

\begin{proposition}
\label{prop:wj-in-bobkov-position}
   There is $C_{\ref{prop:wj-in-bobkov-position}} \in (0, \infty)$ such that for any $n \geq 1$ and any convex body $K \subset \R^n$ which contains the origin in its interior with $K^\circ$ of unit volume and in the $\Bobkov$-position, it holds that 
    \[
   \max\Big\{\mathcal{U}_k(K), \, \mathcal{U}_k(K^\circ) \Big\} \leq C_{\ref{prop:wj-in-bobkov-position}} \sqrt{\log\Big(\frac{\e n}{k}\Big)},
   \]
   for every integer $k \in \{1, 2, \dots, n\}$. 
\end{proposition}
\begin{proof}
We apply~\Cref{prop:quermassintegrals-to-entropy} to the body $L = K^\circ$, using~\Cref{cor:truncation-body-polar}, and to the body $L = n K$, using~\Cref{cor:truncation-body}. We conclude that 
\begin{equation}
\label{eqn:bound-on-quermass}
\max\bigg\{\frac{W_{[k]}(K^\circ)}{\sqrt{n}}, \sqrt{n} \, W_{[k]}(K)\bigg\}
\lesssim \sqrt{\log\Big(\frac{\e n}{k}\Big)}, \quad \mbox{for}~k = 1, \ldots, n.
\end{equation}
Note that by assumption 
\begin{equation}
\label{eqn:volume-radius-ineq}
\vr(K^\circ) = \vol_n(B^n_2)^{-1/n} \simeq \sqrt{n}.
\end{equation}
Moreover, applying~\Cref{prop:truncation-width-to-entropy} 
to $L=K^\circ$, using again~\Cref{cor:truncation-body-polar}, and taking $r=1$, we obtain
\[
\log N(B_2^n,\sqrt n K)\leq Cn.
\]
The standard volumetric argument yields
\[
\vol_n(B_2^n)^{1/n}
\leq \exp(C)\,\vol_n(\sqrt n K)^{1/n}
=\exp(C)\sqrt{n}\vol_n(K)^{1/n}.
\]
Consequently, $\sqrt{n} \gtrsim \vr(K)^{-1}$. Combining this with inequalities~\eqref{eqn:bound-on-quermass} and~\eqref{eqn:volume-radius-ineq}, we obtain the claim.
\end{proof}

We now deduce~\Cref{thm:quermassintegrals} from~\Cref{prop:wj-in-bobkov-position,prop:lower-quermassintegrals}.

\begin{proof}[Proof of~\Cref{thm:quermassintegrals}]
We prove the statement involving the crosspolytope; note that the argument for a regular $n$-simplex with the same volume as $B_1^n$ follows verbatim, by simply replacing $B_1^n$ by $S_n$ in the argument below.
The case $n=1$ follows from $\mathcal U_1=1$, so assume $n\geq2$.

Let $L = \psi(K)$ be such that $\vol_n(L) = 1$ and $\gamma_n(L) \geq \gamma_n(\psi'(K))$ for all invertible affine maps $\psi'$ with $\vol_n(\psi'(K)) = 1$.
By construction, $L$ is in $\Bobkov$-position, and hence $\gamma_L$ is centered by~\Cref{thm:Bobkov-non-symmetric}. In particular, $0 \in \intr(L)$.
Applying~\Cref{prop:wj-in-bobkov-position} to $L^\circ$, we obtain
\[
\max\Big\{\mathcal U_k(L),\mathcal U_k(L^\circ)\Big\}
\leq C_{\ref{prop:wj-in-bobkov-position}}
\sqrt{\log\Big(\frac{\e n}{k}\Big)}, \quad \mbox{for all}~k = 1, \ldots, n.
\]
Since $\vr(B_1^n)\simeq \tfrac{1}{\sqrt{n}}$, \Cref{prop:lower-quermassintegrals} gives, for every volume-preserving affine map $\psi'$,
\[
\mathcal U_k(\psi'(B_1^n))
=\frac{W_{[k]}(\psi'(B_1^n))}{\vr(B_1^n)}
\gtrsim 
\sqrt{\log\Big(\frac{\e n}{k}\Big)}.
\]
Taking the maximum with the ratio for the polar and then passing to the infimum over $\psi'$ yields the comparison in~\eqref{THM3-1}.
Note that the same construction as above, applied to $B_1^n$, gives
\[
\inf_{\psi'}\max\Big\{\mathcal U_k(\psi'(B_1^n)),\mathcal U_k(\psi'(B_1^n)^\circ)\Big\}
\leq C_{\ref{prop:wj-in-bobkov-position}}
\sqrt{\log\Big(\frac{\e n}{k}\Big)},
\]
as required.
\end{proof}

\subsection{Further bounds on the Gaussian width}

First, we record 
an important corollary of~\Cref{thm:comparison-theorem} and~\Cref{thm:comparison-theorem-polar}. We define the Dvoretzky number of a convex body $A \subset \R^n$ as follows: 
\[
\DualDvoretzky(A) = \Big(\frac{w(A)}{\crad(A)}\Big)^2, 
\quad \mbox{where} \quad 
\crad(A) = \inf_{x \in A} \rad(A-x).
\]
From the definition, the quantity $k_\ast$ is translation invariant.
We refer to $\crad(A)$ as the circumradius of $A$. 

\begin{corollary}
\label{cor:subsets-of-K}
There exist universal constants
$C_{\ref{cor:subsets-of-K}},
C_{\ref{cor:subsets-of-K}}' \in (0,\infty)$
such that the following holds for any $n\geq1$.
Let $K\subset\R^n$ be a convex body containing the origin
in its interior, with $K^\circ$ in $\Bobkov$-position
and $\vol_n(K^\circ)=1$.
\begin{enumerate}[label=(\roman*)]
\item For every convex body $L\subset K$, we have
\[
w(L)
\leq C_{\ref{cor:subsets-of-K}}
\sqrt{\log\big(\e+n\,\crad(L)^2\big)}
\leq C_{\ref{cor:subsets-of-K}}'
\sqrt{\log\Big(\frac{\e n}{\DualDvoretzky(L)}\Big)}.
\]
\item For every convex body $T\subset K^\circ$, we have
\[
w(T)
\leq C_{\ref{cor:subsets-of-K}} n
\sqrt{\log\Big(\e+\frac{\crad(T)^2}{n}\Big)}
\leq C_{\ref{cor:subsets-of-K}}' n
\sqrt{\log\Big(\frac{\e n}{\DualDvoretzky(T)}\Big)}.
\]
\end{enumerate}
\end{corollary}
\begin{proof}
\Cref{thm:comparison-theorem} implies that for $x \in L$ and $r = \rad(L-x)$ 
\[
w(L-x) \lesssim \inf_{0 < u \leq n} \Big\{\, 
\sqrt{\log(\e n/u)} + 
r \sqrt{u}\Big\} = \inf_{t \geq 1} 
\Big\{\, \sqrt{\log(\e t^2)} + \frac{r\sqrt{n}}{t}\,\Big\} \simeq \sqrt{\log(\e + r^2 n)}.
\]
The last relation follows with $t^2 = 1 + r^2n$. 
\Cref{thm:comparison-theorem-polar}
implies that for $y \in T$ and $s = \rad(T-y)$, we have 
\[
w(T-y) \lesssim \inf_{0 < u \leq n}
 \Big\{\, n\sqrt{\log(\e n/u)} + s\sqrt{u} \,\Big\} = 
 \inf_{t \geq 1} 
  \Big\{\, n\sqrt{\log(\e t^2)} + \frac{s\sqrt{n}}{t} \,\Big\} 
  \simeq 
  n \sqrt{\log\big(\e + \tfrac{s^2}{n}\big)}.
\]
The last relation follows with $t^2 = 1 + s^2/n$. The first set of inequalities follows now by passing to the infimum over $x$ and $y$, respectively; the second set of claims follows by observing that the choices $u \simeq (w(L)/r)^2$ and $u \simeq (w(T)/s)^2$, respectively, are admissible in the infima above and again passing to the infimum over $x \in L, y \in T$. 
\end{proof}

\begin{remark}
\label{rem:bound-for-truncation-bodies}
\Cref{cor:subsets-of-K} implies the following estimate for truncation bodies. For any convex body $K \subset \R^n$, there exists $T \in \SL(n)$ such that 
\[
\sup_{x \in K} M^\ast\Big(T(K-x) \cap r\vr(K)\,B^n_2\Big)\lesssim 
\vr(K) \sqrt{\log(\e r^2)} \quad \mbox{for}~r \geq 1. 
\]
\end{remark}

For the next statement, we recall the Peetre $K$-functional, defined here for the couple $(\ell_2^n, \ell_\infty^n)$. It is defined by $\mathcal{K}_{2,\infty} \colon \R_+ \times \R^n \to \R_+$, 
\[
\mathcal{K}_{2,\infty}(t, x) = \inf_{y \in \R^n} \Big\{\,\|x - y\|_2 + t \|y\|_\infty \,\Big\} = \inf_{\lambda \geq 0} \bigg\{\, \Big(\sum_{i=1}^n (|x_i| - \lambda)_+^2\Big)^{1/2} + t \lambda \, \bigg\}.
\]
The second relation follows by writing $y = \lambda u$ for $\lambda \geq 0$, $u \in B^n_\infty$, and computing the infimum over the parameter $u$. For a convex body $A$ containing the origin in its interior, we use the shorthand notation
\[
r_A = \crad(A), \qquad \mbox{and} \qquad d_A = \sqrt{\log\Big(\frac{\e n}{\DualDvoretzky(A)}\Big)}.
\]

\begin{corollary}
\label{cor:general-operator}
Let $K \subset \R^n$ be a convex body containing the origin in its interior. Select maps $T_K,T_{K^\circ} \in \GL(n)$ and points $x_K, x_{K^\circ} \in \R^n$ such that the bodies
\[
\tilde K = T_K(K+x_{K}), \qquad \tilde K' = T_{K^\circ}(K^\circ+x_{K^\circ})
\]
are in $\Bobkov$-position and have unit volume. Then, it holds that
\[
w(K) \lesssim r_{\tilde K}\,\mathcal{K}_{2,\infty}\Big(\tfrac{n\sqrt{\log(\e+r_{\tilde K}^2/n)}}{r_{\tilde K}},\sigma(T_K^{-1})\Big) \lesssim r_{\tilde K}\,\mathcal{K}_{2,\infty}\Big(\tfrac{nd_{\tilde K}}{r_{\tilde K}},\sigma(T_K^{-1})\Big) \lesssim n\,\mathcal{K}_{2,\infty}\big(d_{\tilde K},\sigma(T_K^{-1})\big).
\]
Moreover,
\[
w(K^\circ) \lesssim r_{\tilde K'}\,\mathcal{K}_{2,\infty}\Big(\tfrac{n\sqrt{\log(\e+r_{\tilde K'}^2/n)}}{r_{\tilde K'}},\sigma(T_{K^\circ}^{-1})\Big) \lesssim r_{\tilde K'}\,\mathcal{K}_{2,\infty}\Big(\tfrac{nd_{\tilde K'}}{r_{\tilde K'}},\sigma(T_{K^\circ}^{-1})\Big) \lesssim n\,\mathcal{K}_{2,\infty}\big(d_{\tilde K'},\sigma(T_{K^\circ}^{-1})\big).
\]
All the implicit constants above are universal. 
\end{corollary}
\begin{proof}
The maps in the statement above exist by \Cref{prop:existence-uniqueness-of-Bobkov}.
For any convex body $A$ containing the origin in its interior and any linear operator $R \colon \R^n \to \R^n$, the Sudakov--Fernique inequality, translation invariance of the Gaussian width, and Cauchy--Schwarz give
\begin{equation}
\label{ineq:upper-bound-on-G-w}
w(RA) \leq \min\big\{\opnorm{R}w(A),\,r_A w(RB^n_2)\big\}\leq
\min\big\{\opnorm{R}w(A),\,r_A\fronorm{R}\big\}.
\end{equation}
We can in fact interpolate this inequality. Fix a linear operator $S \colon \R^n \to \R^n$ with singular value decomposition
\[
S = \sum_{i=1}^n \sigma_i u_i \otimes v_i.
\]
Fix $\lambda \geq 0$ and decompose  $S=S_1(\lambda)+S_2(\lambda)$, where
\[
S_1(\lambda) = \sum_{i=1}^n (\sigma_i-\lambda)_+u_i\otimes v_i,
\quad \mbox{and} \quad  S_2(\lambda) = \sum_{i=1}^n \min\{\sigma_i,\lambda\}u_i\otimes v_i.
\]
Since $SA\subset S_1(\lambda)A+S_2(\lambda)A$, additivity of the Gaussian width under Minkowski sums and the bound~\eqref{ineq:upper-bound-on-G-w} imply
\[
w(SA) \leq w(S_1(\lambda)A)+w(S_2(\lambda)A) 
\leq r_A\fronorm{S_1(\lambda)}+w(A)\opnorm{S_2(\lambda)} 
\leq r_A\fronorm{S_1(\lambda)}+w(A)\lambda.
\]
Passing to the infimum over $\lambda \geq 0$, we obtain 
\begin{equation}
\label{ineq:general-interpolation-bound}
w(SA) \leq \inf_{\lambda\geq0}\big\{r_A\fronorm{S_1(\lambda)}+w(A)\lambda\big\} = r_A\,\mathcal{K}_{2,\infty}\Big(\frac{w(A)}{r_A},\sigma(S)\Big)
=r_A\,\mathcal{K}_{2,\infty}\big(\sqrt{\DualDvoretzky(A)},\sigma(S)\big).
\end{equation}
For $B$ a unit-volume convex body in $\Bobkov$-position,~\Cref{cor:subsets-of-K} yields
\[
w(B)\lesssim n\sqrt{\log(\e+r_B^2/n)}\lesssim nd_B.
\]
By inequality \eqref{ineq:general-interpolation-bound}, we obtain
\[
w(SB)\lesssim r_B\,\mathcal{K}_{2,\infty}\Big(\tfrac{n\sqrt{\log(\e+r_B^2/n)}}{r_B},\sigma(S)\Big)
\lesssim r_B\,\mathcal{K}_{2,\infty}\Big(\frac{nd_B}{r_B},\sigma(S)\Big) \lesssim n\,\mathcal{K}_{2,\infty}\big(d_B,\sigma(S)\big).
\]
The final inequality holds by~\Cref{lem:outradius-bobkov-k-polar}. 
Applying these estimates with $(B,S)=(\tilde K,T_K^{-1})$ or $(B,S)=(\tilde K',T_{K^\circ}^{-1})$ and recognizing that 
$
w(K)=w(T_K^{-1}\tilde K)$ and $ w(K^\circ)=w(T_{K^\circ}^{-1}\tilde K')$,
we obtain the claim. 
\end{proof}

\subsection{Bounds for the Gelfand numbers}

We now turn to estimates for the Gelfand numbers which follow from our main results. Recall that for an origin-symmetric convex body $K \subset \R^n$ and a linear operator $u \colon (\R^n, \|\cdot\|_K) \to (\R^n, \|\cdot \|_2)$, the Gelfand numbers are defined as
\[
c_k(u) = \inf_{F \in \Gr_{n, n-k+1}} \sup_{x \in F \cap K} \|u(x)\|_2, \quad \mbox{for}~k \in \{1, 2, \dots, n\}.
\]
In particular, for the identity operator, this is the outradius of a minimal $(k-1)$-codimensional section of the body $K$. Thus, we define the Gelfand numbers for a convex body $K \subset \R^n$, which contains the origin in its interior, to be
\[
c_k(K) = \inf_{F \in \Gr_{n, n-k+1}} \rad(F \cap K), \quad \mbox{for}~k\in\{1, 2, \dots, n\}.
\]
Our first result follows by combining our bounds on the radii and the Gaussian widths of the truncated bodies with Gordon's lemma.

\begin{proposition}
\label{prop:gelfand-in-bobkov-position}
   There is $C_{\ref{prop:gelfand-in-bobkov-position}} \in (0, \infty)$ such that for any $n \geq 1$ and any convex body $K \subset \R^n$ which contains the origin in its interior with $K^\circ$ of unit volume and in the $\Bobkov$-position, it holds that 
    \[
   \max\Big\{\frac{c_k(K^\circ)}{n}, c_k(K) \,\Big\} 
   \leq C_{\ref{prop:gelfand-in-bobkov-position}} \, \min\bigg\{1, 
   \sqrt{\frac{\log(\e n/k)}{k}}
   \bigg\},
   \]
   for every integer $k \in \{1, 2, \dots, n\}$. 
\end{proposition}

To prove the result, we require the following observation due to Gordon~\cite[Theorem 1.5]{Gor88}, which relates the Gelfand numbers to the Gaussian width of the truncated bodies. Set
\[
a_k = \E\Big[\sqrt{Z_1^2 + \cdots + Z_k^2}\Big], \quad \mbox{where} \quad 
Z_1, \dots, Z_k \stackrel{\rm i.i.d.}{\sim} N(0, 1).
\]
It is standard that $a_k \simeq \sqrt{k}$.

\begin{lemma}
\label{lem:gordon}
Let $K \subset \R^n$ be a convex body which contains the origin in its interior. Let $k \in \{1, 2, \dots, n-1\}$. If, for some $r > 0$, it holds that ${w(K \cap r B^n_2) < r a_k}$,
then ${c_{k+1}(K) \leq r}$. 
\end{lemma}

\begin{proof}[Proof of~\Cref{prop:gelfand-in-bobkov-position}]
By combining~\Cref{cor:truncation-body},~\Cref{cor:truncation-body-polar},~\Cref{cor:bobkov}, and~\Cref{lem:outradius-bobkov-k-polar}, for $L \in \{n^{-1} K^\circ, K\}$, it 
holds that 
\[
\rad(L) \lesssim 1, \quad \mbox{and} \quad 
w\Big(L \cap \tfrac{s}{\sqrt{n}}B^n_2\Big) \lesssim \sqrt{\log(\e s^2)}, \quad \mbox{for any}~s\geq 1.
\]
For $k=1$,
\begin{equation}
\label{eqn:first-bound-gelfand}
c_1(L) = \rad(L) \lesssim 1.
\end{equation}
This also proves the claim when $n=1$, so assume $n\geq2$. Fix a sufficiently large universal constant $C_1\geq1$, and for an integer $k \in \{1, 2, \dots, n-1\}$, set
\[
s^2 = C_1 \frac{n}{k} \log\Big( \frac{\e n}{k}\Big).
\]
Then $s\geq1$ and $\log(\e s^2)\leq(2+\log C_1)\log(\e n/k)$. Since $a_k\simeq\sqrt{k}$, taking $C_1$ sufficiently large ensures that $r=\tfrac{s}{\sqrt n}$ satisfies
$w(L \cap r B^n_2) < r a_k$. 
As $k \mapsto c_k(L)$ is nonincreasing, 
\begin{equation}
\label{eqn:second-bound-gelfand}
c_{k+1}(L) \leq \min\Big\{c_1(L), \frac{s}{\sqrt{n}}\Big\} \lesssim \min\bigg\{1, \sqrt{\frac{\log(\e n/k)}{k}}\bigg\}
\simeq 
\min\bigg\{1, \sqrt{\frac{\log(\e n/(k+1))}{k+1}}\bigg\}.
\end{equation}
Above, we applied~\Cref{lem:gordon} for the first inequality. Combining~\eqref{eqn:first-bound-gelfand} and~\eqref{eqn:second-bound-gelfand} proves the bound for every $k\in\{1,\dots,n\}$ and both choices of $L$, as $c_k(n^{-1}K^\circ)=n^{-1}c_k(K^\circ)$.
\end{proof}

\Cref{prop:gelfand-in-bobkov-position} implies the following result for general convex bodies. 
\begin{corollary}
\label{cor:gelfand}
   There is $C_{\ref{cor:gelfand}} \in (0, \infty)$ such that for any $n \geq 1$ and any convex body $K \subset \R^n$, there is a volume-preserving affine transformation $\psi \colon \R^n \to \R^n$ such that    
\[  
\max\bigg\{\frac{c_k(\psi(K))}{\vr(\psi(K))},\frac{c_k(\psi(K)^\circ)}{\vr(\psi(K)^\circ)} \bigg\} 
   \leq C_{\ref{cor:gelfand}} 
\, \min\bigg\{\sqrt{n}, 
   \sqrt{\frac{n\log(\e n/k)}{k}}\,\bigg\}
\]
for every integer $k \in \{1, 2, \dots, n\}$. 
\end{corollary}

Note that the estimate above is optimal, at least for the position in which we establish the estimate. In particular, by a result of Garnaev-Gluskin~\cite[p.~1049]{GarGlu84}, it holds that
\[
\max\bigg\{\frac{c_k(B^n_1)}{\vr(B^n_1)},\frac{c_k(B^n_\infty)}{\vr(B^n_\infty)} \bigg\}
\gtrsim \min\bigg\{\sqrt{n}, 
   \sqrt{\frac{n\log(\e n/k)}{k}}\,\bigg\}.
\]
As we show in~\Cref{prop:optimal-m-star-and-gelfand}, the estimate above is also optimal for linear images $TB^n_1$ where $T \in \GL(n)$ satisfies the normalization $\rad(TB^n_1) \lesssim \sqrt{n} \vr(TB^n_1)$. That result also demonstrates that the corresponding estimates for the spherical mean width of the truncation bodies of $B^n_1$ are also optimal, even after optimizing over linear images.  

\begin{proof}[Proof of~\Cref{cor:gelfand}]
Choose an invertible affine map $\psi_0$ such that $L=\psi_0(K)$ has unit
volume and is in $\Bobkov$-position. Then, 
$\vr(L) \simeq \sqrt{n}$ and $\vr(L^\circ) \simeq 1/\sqrt{n}$. 
Applying~\Cref{prop:gelfand-in-bobkov-position} to $L^\circ$,
we therefore obtain
\[
\max\bigg\{
\frac{c_k(L)}{\vr(L)},
\frac{c_k(L^\circ)}{\vr(L^\circ)}
\bigg\}
\lesssim 
\min\bigg\{\sqrt n,
\sqrt{\frac{n\log(\e n/k)}{k}}\bigg\},
\qquad \mbox{for}~k \in \{1, 2, \dots, n\}. 
\]
Set $\psi(x)=\vol_n(K)^{1/n}\psi_0(x)$. Then $\psi$ is volume-preserving. The ratios above are invariant under positive dilations, so the result follows. 
\end{proof}

\begin{remark}
\label{rem:gelfand-to-regularity}
In the case that $K \subset \R^n$ is a centrally symmetric convex body, we note that~\Cref{prop:gelfand-in-bobkov-position} implies~\Cref{thm:covering-numbers}. Indeed, recall the entropy numbers $e_k(K), e_k^\ast(K)$ of a convex body $K \subset \R^n$ are defined 
as in display~\eqref{eqn:entropy-of-body}.
These can be compared to the Gelfand numbers via Carl's inequality~\cite{Carl81}, which implies
\begin{equation}
\label{eqn:Carls-ineq}
\max_{1 \leq j \leq k} j\, \max\{e_j(K), e_j^\ast(K)\} \lesssim 
\max_{1 \leq j \leq k} 
j \, c_j(K), \quad \mbox{for all}~k\in\{1, 2, \dots, n\}.
\end{equation}
Now suppose that $K \subset \R^n$ is additionally such that $K^\circ$ is in $\Bobkov$-position and such that $\vol_n(K^\circ) = 1$. 
Let $L \in \{n^{-1} K^\circ, K\}$. By combining~\eqref{eqn:Carls-ineq} and~\Cref{prop:gelfand-in-bobkov-position}, we obtain 
\[
\max\{e_k(L), e_k^\ast(L)\} \lesssim \min\Big\{1, \sqrt{\frac{\log(\e n/k)}{k}}\Big\}, \quad \mbox{for}~k\in\{1, 2, \dots, n\}.
\] 
The entropy numbers above and the volumetric argument for $r \in (0, 1)$ imply
\[
\log \max\{ N(L, rB), N(B^\circ, rL^\circ)\}\lesssim H_n(r) \quad \mbox{for}~0 < r \lesssim \sqrt{n},
\]
for either $L = n^{-1} K^\circ$ or $L = K$. Recall that $B = n^{-1/2} B^n_2$ and $H_n(r)$ is defined 
in display~\eqref{eqn:surrogate-for-crosspolytope-entropy}. To conclude for a general centrally symmetric convex body, it is enough to put the polar $K^\circ$ in Bobkov position with unit volume and conclude as in the proof of~\Cref{thm:covering-numbers}. This is essentially the strategy that was used by G. Pisier to establish metric entropy bounds for his ``regular'' Milman positions~\cite[pp.~121-123]{Pis89paper} (see also \cite[Ch.~7]{Pis89}).
\end{remark}

\begin{remark}
\label{rem:qs-rand-gelf}
Define the \emph{randomized} Gelfand numbers, as introduced by Litvak and Tomczak-Jaegermann~\cite{LitTom00}:
\[
\RandGel_k(K) = 
\inf\Big\{r > 0 \mid \Prob\{\rad(F \cap K) \leq r\} \geq 1 - \exp(-ck)\Big\}, 
\quad k\in \{1, 2, \dots, n\}.
\]
Above, $c > 0$ is a sufficiently small universal constant, and $F \subset \R^n$ is a uniform $(k-1$)-codimensional linear subspace. 
Combining the mean width estimates in \Cref{cor:subsets-of-K} together with the form of the low-$M^\ast$ estimate of Pajor and Tomczak-Jaegermann~\cite{PajTom86} (see also~\cite[Theorems 7.3.1 and 9.3.8]{ArtAviGiaMil15}), we obtain that for any convex body $K \subset \R^n$, there exists an invertible affine map $\psi \colon \R^n \to \R^n$ such that $\psi(K)$ contains the origin in its interior and
\[
\max\bigg\{\frac{\RandGel_k(\psi(K))}{\vr(\psi(K))},\frac{\RandGel_k(\psi(K)^\circ)}{\vr(\psi(K)^\circ)} \bigg\} 
   \lesssim 
\, \min\bigg\{\sqrt{n}, 
   \sqrt{\frac{n\log(\e n/k)}{k}}\,\bigg\}, 
\] 
simultaneously for $k \in \{1, 2, \dots, n\}$. The result above strengthens the result of~\Cref{cor:gelfand}. 
In fact, by adapting the argument of E.\ Milman and Yifrach~\cite[Theorem~5.2]{MilYif21}, it also implies that
if $E \subset F \subset \R^n$ are uniform $2(k-1)$-codimensional and $(k-1)$-codimensional subspaces (specifically, drawn according to the Haar measure on the corresponding flag manifold; see~\cite[Section 5]{MilYif21}), then 
for $1 \leq k \leq n/2$ and with probability at least $1 - 2\exp(-c'k)$, it holds that 
\[
d_{g}(P_E(\psi(K) \cap F), B^E_2) \lesssim \min\Big\{n, \frac{n}{k} \log\frac{\e n}{k}\Big\}.
\] 
Similarly, it holds with probability at least $1 - 2\exp(-c'k)$ that 
\[
d_{g}(P_E(\psi(K)^\circ \cap F), B^E_2) 
=
d_{g}(P_F(\psi(K)) \cap E, B^E_2)
\lesssim \min\Big\{n, \frac{n}{k} \log\frac{\e n}{k}\Big\}.
\]
Above, $d_g$ denotes the geometric distance between convex bodies, and $B^E_2$ the ball $B^n_2 \cap E$. 
Notably, the estimates above arise by taking $\psi(K)$ to be an affine image of $K$ in $\Bobkov$-position, up to the volume normalization. Importantly, the estimates hold, simultaneously for the body $L = \psi(K)$ and the polar $L^\circ$ and the position is independent of $k$, thereby improving on the result of Milman-Yifrach~\cite[Theorem 1.1]{MilYif21}. A consequence of the result described above is that for every convex body $K\subset\R^n$, there
is an invertible affine map $\psi$, so that for every $\varepsilon\in(0,1)$, there are
linear subspaces $E\subset F\subset\R^n$ such that
\[
\dim(E)\geq(1-\varepsilon)n,
\qquad
d_g\big(P_E(\psi(K)\cap F),B_2^E\big)
\lesssim 
\frac{\log(\e/\varepsilon)}{\varepsilon}.
\]
We emphasize that the position $\psi$ is \emph{independent} of $\varepsilon \in (0, 1)$. Moreover, in the non-symmetric case, this improves on the previous estimate of $\eps^{-2} \mathrm{polylog}(\eps^{-1})$ of Rudelson~\cite[Theorem 4]{Rud00}.
\end{remark}

\bibliographystyle{abbrv}
\bibliography{references}

\appendix

\section{Computations for the simplex and crosspolytope}
\label{app:computations-for-crosspolytope-and-simplex}

In this appendix, we collect some computations for the simplex and crosspolytope. Throughout we use the function $H_n$ defined in~\eqref{eqn:surrogate-for-crosspolytope-entropy}.
First, we show that, at least up to universal constants, the metric entropy for any affine image of the regular $n$-simplex or crosspolytope cannot be smaller than the order established in the paper of Sch\"utt~\cite{Sch84}. 

\begin{proposition}\label{prop:lower-bound-entropy}
There are universal constants $c_{\ref{prop:lower-bound-entropy}},c_{\ref{prop:lower-bound-entropy}}'>0$ such that for every 
integer $n \geq 2$, every volume-preserving affine transformation $\psi \colon \R^n \to \R^n$, and every
$r \in (0, c_{\ref{prop:lower-bound-entropy}}\sqrt n]$,
\[
\min\bigg\{\,\log
N\Big(\psi(B_1^n),\frac{r}{\sqrt n}B_2^n\Big),~
\log N\Big(\psi(S_n),\frac{r}{\sqrt n}B_2^n\Big)
\,\bigg\}
\geq 
c_{\ref{prop:lower-bound-entropy}}'
\,
H_n(r),
\]
where $S_n$ is any regular $n$-simplex with $\vol_n(S_n) = \vol_n(B^n_1)$.
\end{proposition}

Similarly, at least up to universal constants, the quermassintegrals of an affine image of the regular $n$-simplex and crosspolytope cannot be less than those of the crosspolytope itself.

\begin{proposition}
\label{prop:lower-quermassintegrals}
There is a universal constant $c_{\ref{prop:lower-quermassintegrals}} > 0$ such that for every integer $n \geq 2$, every volume-preserving affine transformation $\psi \colon \R^n \to \R^n$, and every integer $k \in \{1, 2, \dots, n\}$, it holds that 
\[
\min\Big\{W_{[k]}(\psi(B^n_1)), W_{[k]}(\psi(S_n))\Big\} 
\geq 
c_{\ref{prop:lower-quermassintegrals}} \, \sqrt{\frac{\log(\e n/k)}{n}},
\]    
where $S_n$ is any regular $n$-simplex with $\vol_n(S_n) = \vol_n(B^n_1)$. 
\end{proposition}

Finally, we establish that the spherical mean width of the truncated bodies and the Gelfand numbers are essentially minimized over all  positions with the identity map when the body is the crosspolytope or simplex.

\begin{proposition}
\label{prop:optimal-m-star-and-gelfand}
There are constants $c_{\ref{prop:optimal-m-star-and-gelfand}}, c_{\ref{prop:optimal-m-star-and-gelfand}}'> 0$ such that for any integer $n \geq 2$, and any $\gamma \geq 1$, if $T \in \GL(n)$ satisfies 
\[
\rad(TB^n_1) \leq \gamma \sqrt{n} \vr(TB^n_1)
\]
then 
\[
\frac{c_k(TB^n_1)}{\vr(TB^n_1)} \geq \frac{c_{\ref{prop:optimal-m-star-and-gelfand}}}{\gamma} \min\Big\{\sqrt{n}, \sqrt{\frac{n}{k} \log\Big(\frac{\e n}{k}\Big)}\Big\}, \quad \mbox{for}~1 \leq k \leq c_{\ref{prop:optimal-m-star-and-gelfand}}' n, 
\]
and 
\[
\sup_{x \in B^n_1} M^\ast\Big(T(B^n_1-x) \cap r \vr(TB^n_1)B^n_2\Big) \geq 
\frac{c_{\ref{prop:optimal-m-star-and-gelfand}}}{\gamma} \, 
\vr(TB^n_1) \sqrt{\log(\e r^2)} \quad \mbox{for}~1 \leq r\leq \sqrt{n}.
\]  
\end{proposition}

\subsection{Proof of~\Cref{prop:lower-bound-entropy}}

We will use the following lower bound, which follows from bounds on entropy numbers between $\ell_n^p$ and $\ell_n^r$ for $p = 1, r = 2$ as proved by Sch\"utt~\cite[Theorem 1]{Sch84}.

\begin{lemma}
\label{lem:schutt}
There are constants $c_{\ref{lem:schutt}},c_{\ref{lem:schutt}}'\in(0,\infty)$ 
such that for every integer $d\geq1$ 
and every $s\in[1,c_{\ref{lem:schutt}}\sqrt d]$,
\[
 \log N\!\left(B_1^d,\frac s{\sqrt d}B_2^d\right)
 \geq c_{\ref{lem:schutt}}'\frac d{s^2}\log(\e s^2).
\]
\end{lemma}

Additionally we require the following form of the restricted invertibility of partial isometries. The version stated below is taken from~\cite[Theorem 8.3.4]{ArtAviGiaMil21}; it is originally due to Spielman-Srivastava~\cite{SpiSri12} (following improvements by Vershynin~\cite{Ver01} of a theorem of Bourgain-Tzafriri~\cite{BouTza87}). 

\begin{lemma}[Restricted invertibility for partial isometries]
\label{lem:restricted-invertibility}
For any integers $1 \leq k \leq p \leq q$ and any $U \in \R^{p \times q}$ with $UU^\T = I_p$, there exists some $\sigma \subset \{1, 2, \dots, q\}$ with $|\sigma| = k$ such that 
\[
\|Ux\|_2 \geq \Big(1 - \sqrt{\tfrac{k}{p}}\Big) \sqrt{\frac{p}{q}} \, \|x\|_2, \quad \text{whenever }\supp(x) \subset \sigma.
\]
\end{lemma}

Without loss of generality, we may write $\psi(x) = Tx + v$ for $T \in \SL(n)$ and $v \in \R^n$. 
By the translation invariance of metric entropy, it holds that 
\begin{equation}
\label{ineq:lower-for-cross-polytope}
\log
N\Big(\psi(B_1^n),\frac{r}{\sqrt n}B_2^n\Big)
= 
\log N\Big(TB^n_1,\frac{r}{\sqrt n}B_2^n\Big), 
\end{equation}
for all $r > 0$. Let $\Delta_n = \operatorname{conv}\{0, 2e_1, \dots, 2e_n\}$. Since $|S_n| = |\Delta_n| = |B^n_1|$, we may write $S_n = A\Delta_n + b$ for some $A \in \SL(n)$ and $b \in \R^n$. The inclusion $B^n_1 \subset \tfrac{1}{2}(\Delta_n - \Delta_n)$ gives
\[
\frac{S_n - S_n}{2} = A\Big(\frac{\Delta_n - \Delta_n}{2}\Big) \supset AB^n_1.
\]
Set $L = \psi(S_n)$. If $x_i$ denote the centers of a cover of $L$ by translates of $sB^n_2$, then the differences $x_i - x_j$ are the centers of a cover of $L-L$ by translates of $2sB^n_2$. Thus, for every $s > 0$,
\[
N(L, sB^n_2)^2 \geq N(L-L, 2s B^n_2)=
N\bigg(T \Big(\frac{S_n - S_n}{2}\Big), 
sB^n_2\bigg) \geq 
N(TAB^n_1, sB^n_2).
\]
Taking logarithms and setting $s = r/\sqrt{n}$, we obtain
\begin{equation}
\label{eqn:lower-for-simplex}
\log N\Big(\psi(S_n), \frac{r}{\sqrt{n}} B^n_2\Big) 
\geq \frac{1}{2} \, 
\log N\Big(TAB^n_1, \frac{r}{\sqrt{n}} B^n_2\Big).
\end{equation}
Since $TA \in \SL(n)$, it follows from~\eqref{ineq:lower-for-cross-polytope} and~\eqref{eqn:lower-for-simplex} that it suffices to show
$\log N(TB^n_1, \tfrac{r}{\sqrt{n}} B^n_2) \gtrsim H_n(r)$ for every $T \in \SL(n)$ and $0 < r \lesssim \sqrt{n}$.

First consider the case that $r \in (0, 1]$. Then, by the volumetric argument,
\[
\log N(TB^n_1, \tfrac{r}{\sqrt{n}} B^n_2) \geq n \log \Big(\frac{\sqrt{n}}{r} \vr(B^n_1)\Big) 
\geq c_1 n \log\Big(\frac{\e}{r}\Big),
\]
where we used that $n \mapsto \vr(\sqrt{n} B^n_1)$ is increasing and strictly larger than $1$ for $n = 2$.

Now suppose $r > 1$. Let $V_k \in \R^{k \times n}$ denote 
the matrix whose rows are the top-$k$ right singular vectors of $T$; we denote the singular values of $T$ by $s_1 \geq s_2 \geq \cdots \geq s_n > 0$. By the inclusion $B^n_1\supset \tfrac{1}{\sqrt{n}} B^n_2$, 
the volumetric argument gives
\begin{equation}
\label{eqn:volumetric-bound}
\log N(TB^n_1, \tfrac{r}{\sqrt{n}}B^n_2) \geq 
\log N(TB^n_2, rB^n_2) \geq \sum_{k : s_k > r} \log \frac{s_k}{r},
\end{equation}
for any $r > 0$. An alternative estimate can be found from restricted invertibility. Then, observe that by~\Cref{lem:restricted-invertibility}, it holds that for every integer $1 \leq m \leq k$, there is an $m$-subset $\sigma_m \subset \{1,2,\dots, n\}$ such that
\[
\|V_k x\|_2 \geq \Big(1 - \sqrt{\tfrac{m}{k}} \Big)\sqrt{\frac{k}{n}} \|x\|_2, \quad \mbox{whenever}~\supp(x) \subset \sigma_m.
\] 
If $k \geq 4$, taking $m = \floor{\tfrac{k}{4}}$, we get for $\sigma \equiv \sigma_m$, by singular value decomposition, that 
\[
\|Tx\|_2 \geq s_k \|V_k x\|_2 \geq \Big(\sqrt{\frac{k}{m}} - 1\Big) \, s_k \sqrt{\frac{m}{n}} \|x\|_2 \geq s_k \sqrt{\frac{m}{n}} \|x\|_2, \quad \mbox{whenever}~\supp(x) \subset \sigma. 
\]
We write $\log_+(x) = \max\{0, \log x\}$. Comparing packing and covering, we obtain
\begin{multline}
\label{ineq:lower-bound-for-other-k}
\log N\Big(TB^n_1,\tfrac{r}{\sqrt{n}} B^n_2\Big)
\geq 
\log M(TB^n_1,\tfrac{2r}{\sqrt{n}} B^n_2) 
\geq 
\log M(B^m_1, \tfrac{2r}{s_k \sqrt{m}} B^m_2) 
\\ \geq  
\log N(B^m_1, \tfrac{2r}{s_k \sqrt{m}} B^m_2)
\geq 
\frac{c_{\ref{lem:schutt}}'}{4} \frac{m s_k^2}{r^2} \log_+\Big(\e \frac{r^2}{s_k^2}\Big)
\geq c_2 \frac{k s_k^2}{r^2} 
\log_+\Big(\e \frac{r^2}{s_k^2}\Big).
\end{multline}
The penultimate inequality holds, provided that $r \leq c_3 s_k \sqrt{k}$ by \Cref{lem:schutt}.
The final inequality uses our choice of $m$ and our assumption on $k$. 

To conclude, we only need to combine the volumetric bound~\eqref{eqn:volumetric-bound} (which holds for any $r > 0$) and the inequality formed by passing to subspaces~\eqref{ineq:lower-bound-for-other-k}, which holds provided that $k \geq c_4 \max\{1, r^2/s_k^2\}$. Indeed, fix $c_5 > 0$. Of course, we may  assume that
\[
\sum_{k : s_k > r} \log \frac{s_k}{r} < c_5 \frac{n}{r^2} \log(\e r^2);
\]
otherwise there is nothing to prove. In that case, 
we have for $t_k = \min\{s_k, r\}$ that 
\[
\Big(\prod_{k=1}^n t_k\Big)^{-1} = \frac{1}{r^n}\Big(\prod_{k : s_k < r} \frac{s_k}{r}\Big)^{-1} = \prod_{k : s_k > r} \frac{s_k}{r}  < \exp\Big(c_5 \frac{n}{r^2} \log(\e r^2) \Big).
\]
In particular, using $n! \geq (n/\e)^n$, 
\[
\Big(\prod_{k=1}^n k t_k\Big)^{1/n} = (n!)^{1/n}\Big(\prod_{k=1}^n t_k \Big)^{1/n} > \frac{n}{\e} \exp\Big(-\frac{c_5}{r^2} \log(\e r^2)\Big) \geq c_6 n.
\]
The last inequality uses $r \geq 1$ and holds provided that $c_5, c_6$ are sufficiently small. Assuming also that $r \leq c_7 \sqrt{n}$, there is $k$ such that the following inequalities hold:
\[
s_k \geq \min\{s_k, r\} = t_k \geq c_6 \frac{n}{k} \geq c_6, \quad k \geq c_6 \frac{n}{r} \geq c_4, \quad 
\frac{s_k^2 k}{r^2} \geq c_6^2\frac{n}{r^2} \geq \frac{c_6^2}{c_7^2} \geq c_4.
\]
The last inequalities hold provided that, after fixing $c_6 \leq 1$, we choose $c_7$ so small that $c_7^2 \leq c_6^2/c_4$.
Then, with this choice, 
\begin{equation}
\label{ineq:lower-bound-for-k}
\frac{k s_k^2}{r^2} \log_+ \frac{\e r^2}{s_k^2} \geq 
c_6 \frac{n}{r^2}\log(\e r^2) \cdot \frac{s_k \log_+(\e r^2/s_k^2)}{\log(\e r^2)} \geq c_8  \frac{n}{r^2}\log(\e r^2) .
\end{equation}
To justify the last inequality, note that $c_6 \leq s_k$. Moreover, if $s_k > r$, then
\[
\log \frac{s_k}{r} \leq \frac{1}{k}\sum_{j=1}^k \log \frac{s_j}{r} < c_5 \,\frac{n}{kr^2} \log(\e r^2) \leq \frac{c_5}{c_6} \frac{\log(\e r^2)}{r} \leq \frac{1}{100},
\]
by taking $c_5$ sufficiently small.
Hence, using $c_6 \leq s_k \leq \e^{1/100} r$,
\[
s_k \log_+(\e r^2/s_k^2)
\geq 
\frac{c_6}{2} \log(\e^{99/50} r^2) \geq 
\frac{c_6}{2} \log(\e r^2),
\]
which yields~\eqref{ineq:lower-bound-for-k}, as needed. To complete the proof, we combine this with~\cref{ineq:lower-bound-for-other-k}. We may
take \[
c_{\ref{prop:lower-bound-entropy}} = c_7, \quad \mbox{and} \quad  c_{\ref{prop:lower-bound-entropy}}' = \tfrac{1}{2}\min\{c_1, c_5, c_2 c_8\}. \qedhere
\]

\subsection{Proof of~\Cref{prop:lower-quermassintegrals}}

We will make use of the following computation. For $p > 0$ and a convex body $K \subset \R^n$ with $0 \in \intr(K)$, we define the negative spherical mean width of order $p$ by
\[
M^\ast_{-p}(K) =
\E_{\theta \sim \sigma_{n-1}}
\Big[h_{K}^{-p}(\theta)\Big]^{-1/p}.
\]

\begin{lemma}[{\cite[Proposition 6.1]{PaoPiv13}}]
\label{lem:lower-for-negative-moments}
There exists a constant $c_{\ref{lem:lower-for-negative-moments}} > 0$ such that for all integers $1 \leq k \leq n$, it holds that 
\[
M^\ast_{-k}(B^n_1) \geq c_{\ref{lem:lower-for-negative-moments}} \; \sqrt{\frac{\log(\e n/k)}{n}}.
\]
\end{lemma}

In fact, the estimate in~\Cref{lem:lower-for-negative-moments} can also be reversed, up to adjusted universal constants; however, we will not require that for the remainder of the argument. 

\begin{lemma}    
\label{lem:lower-quermass-to-neg-moments}
There exists a constant $c_{\ref{lem:lower-quermass-to-neg-moments}} > 0$ such that for any integer $1 \leq k \leq n$ and every origin-symmetric convex body $K \subset \R^n$, it holds that 
\[
W_{[k]}(K) \geq c_{\ref{lem:lower-quermass-to-neg-moments}} \, M^\ast_{-k}(K).
\]
\end{lemma}
\begin{proof}
Let $F$ be a uniformly distributed $k$-dimensional subspace of $\R^n$. Polar coordinates, the identity ${(P_FK)^\circ = K^\circ \cap F}$, the reverse Blaschke--Santal\'o inequality, and Jensen's inequality give
\begin{align*}
\big(M^\ast_{-k}(K)\big)^{-k}
= \E \vr(K^\circ \cap F)^k
\geq c_1^k \, \E[\vr(P_FK)^{-k}]
\geq c_1^k \E[\vr(P_F K)^k]^{-1}
= c_1^k W_{[k]}(K)^{-k},
\end{align*}
where $c_1 > 0$ is universal. Taking negative $\tfrac{1}{k}$-th powers proves the claim.    
\end{proof}

Without loss of generality, we may write $\psi(x) = Tx + v$ for $T \in \SL(n)$ and $v \in \R^n$. 
By the translation invariance of the $k$th quermassintegral, we have ${W_{[k]}(\psi(B^n_1)) = W_{[k]}(T B^n_1)}$. 
On the other hand, let $\Delta_n = \operatorname{conv}\{0, 2e_1, \dots, 2e_n\}$. Since $|S_n| = |\Delta_n| = |B^n_1|$, we may write $S_n = A\Delta_n + b$ for some $A \in \SL(n)$ and $b \in \R^n$. For every $k$-dimensional subspace $F \subset \R^n$, the inclusion $B^n_1 \subset \tfrac{1}{2}(\Delta_n - \Delta_n)$ and the Rogers--Shephard inequality give
\[
\vol_k(P_F T A B^n_1)^{1/k}\leq  \vol_k\Big(P_F T\big(\tfrac{S_n - S_n}{2}\big)\Big)^{1/k} \leq 
\frac{1}{2}\binom{2k}{k}^{1/k} 
\vol_k(P_F \psi(S_n))^{1/k}.
\]
Using $\binom{2k}{k} \leq 4^k$ and integrating over $F$, we obtain
\[
W_{[k]}(\psi(S_n)) \geq \frac{1}{2} \, 
W_{[k]}(TA B^n_1) \geq 
\frac{1}{2} \inf_{\psi'} W_{[k]}(\psi'(B^n_1)), 
\]
where the infimum ranges over all volume-preserving affine transformations $\psi'\colon \R^n \to \R^n$. 
Consequently, to establish the claim, it suffices to show that 
\begin{equation}
\label{eqn:suff-condition-for-quermass}    
 \inf_{\psi'} W_{[k]}(\psi'(B^n_1)) = \inf_{T \in \SL(n)} W_{[k]}(T B^n_1)\gtrsim 
 \sqrt{\frac{\log(\e n/k)}{n}}, 
\end{equation}
for all $1 \leq k \leq n$. 
Observe that for every $S \in \SL(n)$ and every $r > 0$, since $rB^n_\infty$ is in $\Bobkov$-position (by the $1$-symmetry of $rB^n_\infty$), it holds that
\[
\gamma_n(rSB^n_\infty) \leq \gamma_n(rB^n_\infty).
\]
For any $T \in \SL(n)$ and $0 < p < n$, let $G \sim \gamma_n$ and $S = (T^\ast)^{-1}$. Then 
\[
M^\ast_{-p}(TB^n_1)^{-p} \E[\|G\|_2^{-p}] = \E[h_{TB^n_1}^{-p}(G)] = 
\int_0^\infty \frac{p}{r^{p+1}} \gamma_n(r SB^n_\infty) \, \ud r.
\] 
Combining the previous two displays, we obtain $M^\ast_{-p}(TB^n_1) \geq M^\ast_{-p}(B^n_1)$ for $0 < p < n$. For $p = n$, polar coordinates and $|\det S| = 1$ yield
\[
M^\ast_{-n}(TB^n_1)
= \Big(\frac{\vol_n(B^n_2)}{\vol_n(SB^n_\infty)}\Big)^{1/n}=
\Big(\frac{\vol_n(B^n_2)}{\vol_n(B^n_\infty)}\Big)^{1/n}
= M^\ast_{-n}(B^n_1).
\]
Thus, for any $T \in \SL(n)$ and any integer $1 \leq k \leq n$, it holds that 
\[
M^\ast_{-k}(TB^n_1) \geq M^\ast_{-k}(B^n_1) \geq c_{\ref{lem:lower-for-negative-moments}} \; \sqrt{\frac{\log(\e n/k)}{n}},
\]
by~\Cref{lem:lower-for-negative-moments}. Finally, we conclude by applying~\Cref{lem:lower-quermass-to-neg-moments}, thereby obtaining inequality~\eqref{eqn:suff-condition-for-quermass}, as needed. 

\subsection{Proof of~\Cref{prop:optimal-m-star-and-gelfand}}

The idea is essentially to combine the restricted invertibility theorem with the lower bound on the Gelfand numbers as established by Garnaev-Gluskin. 

In particular, we will use the following lemma. 

\begin{lemma}[Garnaev-Gluskin~\cite{GarGlu84}]
\label{lem:gar-glu}
There is a universal constant 
$c_{\ref{lem:gar-glu}} > 0$ such that 
\[
c_\ell(B^d_1) 
\geq c_{\ref{lem:gar-glu}} 
\, \min \Big\{1, \sqrt{\frac{1}{\ell} \log\frac{\e d}{\ell}}\Big\}, 
\quad \mbox{for}~\ell \in \{1, 2, \dots, d \}. 
\]
\end{lemma}

Now fix an operator $T \in \GL(n)$ satisfying 
\[
\rad(TB^n_1) \leq \gamma \sqrt{n} \vr(T B^n_1) \leq C\gamma |\det T|^{1/n}.
\]
By rescaling, we may assume without loss of generality that $|\det T| = 1$. Denote the singular values of $T$ by $s_1 \geq s_2 \geq \cdots s_n$. The display above implies that
\[
\rad(TB^n_1) = \max_{1 \leq i \leq n} \|Te_i\|_2 \leq C \gamma. 
\]
Hence, the (normalized) Frobenius norm satisfies
\[
\frac{1}{n} \sum_{i=1}^n s_i^2 = \frac{1}{n} \sum_{i=1}^n \|Te_i\|_2^2 \leq C^2 \gamma^2. 
\]
Let $k \in \{1, 2, \dots, n-1\}$. The ordering of the singular values followed by the AM-GM inequality yields
\[
1 = \prod_{i=1}^n s_i \leq 
s_{k+1}^{n-k} \Big(\frac{1}{k} \sum_{i=1}^n s_i^2\Big)^{k/2} \leq s_{k+1}^{n-k} \Big(\frac{n}{k} C^2 \gamma^2\Big)^{k/2}.
\]
Rearranging this inequality, we obtain for $k = \floor{n/2}$ that 
\[
s_{k+1} \geq \Big(\frac{k}{n C^2 \gamma^2} \Big)^{k/(2(n-k))} \geq \frac{c'}{\gamma}.
\]
Let $m = \max\{1, \floor{(k+1)/4}\}$. Let $V_{k+1} \in \R^{(k+1) \times n}$ have rows corresponding to the right singular vectors of $T$ associated to the singular values $s_1, \dots, s_{k+1}$. By applying the restricted invertibility (\Cref{lem:restricted-invertibility}), there exists a subset $\sigma \subset \{1, 2, \dots, n\}$ with $|\sigma| = m$ such that if $\supp(x) \subset \sigma$ and if $k+1 \geq 4$, 
\[
\|Tx\|_2 \geq s_{k+1} \|V_{k+1} x\|_2 
\geq \frac{c'}{\gamma} \Big(1 - \sqrt{\frac{m}{k+1}}\Big)\sqrt{\frac{k+1}{n}}\,\|x\|_2\geq \frac{c''}{\gamma} \|x\|_2. 
\]
If $k + 1 < 4$, then $m = 1$. In this case, Hadamard's inequality gives $\max_i\|Te_i\|_2\geq1$, so the same conclusion follows by taking $\sigma=\{i\}$ for an index attaining this maximum.
Fix a $(\ell-1)$-codimensional subspace $F \subset \R^n$ for $1 \leq \ell \leq m$. Consider the subspace 
\[
\tilde F = \{x \in \R^n : \supp(x) \subset \sigma, x \in T^{-1} F\}.
\]
Observe that 
\[
\rad(TB^n_1 \cap F) \geq \sup_{x \in B^n_1 \cap \tilde F} \|Tx\|_2 
\geq \frac{c''}{\gamma} \, \sup_{x \in B^n_1 \cap \tilde F} \|x\|_2 \geq  \frac{c''}{\gamma} \, c_\ell(B^m_1).
\]
The last inequality holds since the subspace $\tilde F$ (viewed inside of $\R^\sigma$) is at most $(\ell-1)$-codimensional. 
Passing to the infimum over $F \in \Gr_{n, n-\ell+1}$, we obtain for $1 \leq \ell \leq m$,
\begin{equation}
\label{eqn:lower-on-gelfand}
c_\ell(TB^n_1) \geq \frac{c''}{\gamma} c_\ell(B^m_1) \gtrsim \frac{1}{\gamma} 
\min\Big\{1, \sqrt{\frac{1}{\ell} \log\frac{\e m}{\ell}}\Big\} 
\simeq \frac{1}{\gamma} 
\min\Big\{1, \sqrt{\frac{1}{\ell} \log\frac{\e n}{\ell}}\Big\}.
\end{equation}
The penultimate inequality follows from~\Cref{lem:gar-glu}, and the last relation holds as $m \simeq k \simeq n$. As $\vr(TB^n_1)\simeq 1/\sqrt{n}$, this establishes the claim regarding the Gelfand numbers. 

Regarding the spherical mean width, first observe that since $\vr(B^n_1)\gtrsim 1/\sqrt{n}$, by \Cref{lem:gordon}, 
\[
\frac{M^\ast(TB^n_1 \cap r \vr(TB^n_1) B^n_2)}{\vr(TB^n_1)} \gtrsim \max_{1 \leq \ell \leq m} \sqrt{\ell} \, 
c_\ell(TB^n_1 \cap \tfrac{r}{\sqrt{n}} B^n_2)
= 
\max_{1 \leq \ell \leq m} 
\sqrt{\ell} 
\min\{c_\ell(TB^n_1), r/\sqrt{n}\}.
\]
We will take 
\[
\ell = \max\Big\{1, \floor{ c' \frac{n \log(\e r^2)}{r^2}}\Big\},
\]
for a sufficiently small universal constant $c'>0$. Then necessarily, by adjusting $c'$ if necessary, $1 \leq \ell \leq m$ as $1 \leq r \leq \sqrt{n}$. With this choice, it is easy to see from \eqref{eqn:lower-on-gelfand} ${c_\ell(TB^n_1) \gtrsim r/(\gamma\sqrt{n})}$. Consequently, 
\[
M^\ast(TB^n_1 \cap r \vr(TB^n_1) B^n_2) \gtrsim \frac{\vr(TB^n_1)}{\gamma} \sqrt{\frac{\ell r^2}{n}} \simeq \frac{\vr(TB^n_1)}{\gamma} \sqrt{\log(\e r^2)}, 
\]
as required. 
\end{document}